\documentclass[twoside,12pt]{article}
\usepackage[usenames,dvipsnames]{color}
\usepackage{amsmath}
\usepackage{amsthm}
\usepackage{amssymb,bbm}
\usepackage[mathcal]{euscript}

\usepackage{cite}
\usepackage{hyperref}
\usepackage{enumitem}
\usepackage{tikz}

\definecolor{rosso}{rgb}{0.85,0,0}

\def\mau #1{{\color{magenta}#1}}
\newtheorem{theorem}{Theorem}[section]
\newtheorem{remark}[theorem]{Remark}
\newtheorem{property}[theorem]{Property}
\newtheorem{corollary}[theorem]{Corollary}

\newtheorem{proposition}[theorem]{Proposition}
\newtheorem{lemma}[theorem]{Lemma}

\numberwithin{equation}{section}

\let\non\nonumber

\def\Lip{Lip\-schitz}

\def\lhs{left-hand side}
\def\rhs{right-hand side}

\def\multibold #1{\def\arg{#1}%
  \ifx\arg\pto \let\next\relax
  \else
  \def\next{\expandafter
    \def\csname #1#1#1\endcsname{{\bf #1}}%
    \multibold}%
  \fi \next}

\def\pto{.}

\def\multical #1{\def\arg{#1}%
  \ifx\arg\pto \let\next\relax
  \else
  \def\next{\expandafter
    \def\csname cal#1\endcsname{{\cal #1}}%
    \multical}%
  \fi \next}

\def\multimathop #1 {\def\arg{#1}%
  \ifx\arg\pto \let\next\relax
  \else
  \def\next{\expandafter
    \def\csname #1\endcsname{\mathop{\rm #1}\nolimits}%
    \multimathop}%
  \fi \next}

\multibold
qwertyuiopasdfghjklzxcvbnmQWERTYUIOPASDFGHJKLZXCVBNM.

\multical
QWERTYUIOPASDFGHJKLZXCVBNM.

\multimathop
diag dist div dom mean meas sign supp .

\def\<#1>{\mathopen\langle #1\mathclose\rangle}
\def\norma #1{\left\| #1\right\|}
\newcommand{\vertiii}[1]{{\left\vert\kern-0.25ex\left\vert\kern-0.25ex\left\vert #1 \right\vert\kern-0.25ex\right\vert\kern-0.25ex\right\vert}}
\def\opnorm #1{\vertiii{#1}}

\def\I2 #1{\int_{Q_t}|{#1}|^2}
\def\IT2 #1{\int_{Q_t^T}|{#1}|^2}
\def\IO2 #1{\norma{{#1(t)}}^2}
\def\IOT2 #1{\norma{{#1(T)}}^2}
\def\ov #1{{\overline{#1}}}

\def\iO{\int_\Omega}

\def\dt{\partial_t}

\def\dn{\partial_{\bf n}}
\def\checkmmode #1{\relax\ifmmode\hbox{#1}\else{#1}\fi}

\def\erre{{\mathbb{R}}}

\def\VV{{\bf{V}}}
\def\bv{{\boldsymbol{v}}}
\def\bw{{\boldsymbol{w}}}
\def\bstar{{\boldsymbol{*}}}
\def\HH{{\bf{H}}}
\def\WW{{\bf{W}}}
\def\LL{{\bf{L}}}

\def\CC{{\mathbb{C}}}

\def\Vp{{V^*}}
\def\VVp{{{\bf V}^*}}

\def\aeQ{\checkmmode{a.e.\ in~$Q$}}

\def\genspazio #1#2#3#4#5{#1^{#2}(#5,#4;#3)}
\def\spazio #1#2#3{\genspazio {#1}{#2}{#3}T0}

\def\spazios #1#2#3{\genspazio {#1}{#2}{#3}T \sigma}
\def\L {\spazio L}
\def\H {\spazio H}

\def\Ls {\spazios L}

\def\Lx #1{L^{#1}(\Omega)}
\def\Hx #1{H^{#1}(\Omega)}

\let\eps\varepsilon

\def\d{\delta}

\def\s{\sigma}
	
\def\ph{\varphi}

\def\cd{C_{\d}}

\newcommand{\0}{\boldsymbol{0}}
\newcommand{\1}{\boldsymbol{1}}

\renewcommand{\S}{\mathbf{S}}

\newcommand{\Sph}{\S_{\bph}}
\newcommand{\SP}{S_P}
\newcommand{\SR}{\S_{\br}}

\newcommand{\A}{\mathbf{A}}
\newcommand{\B}{\mathbf{B}}

\newcommand{\DD}{\mathbb{D}}
\newcommand{\SB}{\mathbb{S}}

\newcommand{\NN}{{\cal N}}
\newcommand{\RR}{{\cal R}}
\newcommand{\bRR}{{\boldsymbol{\RR}}}
\newcommand{\bNN}{{\boldsymbol{\NN}}}

\def\bj{{\boldsymbol{J}}}

\newcommand{\bph}{{\boldsymbol{\varphi}}}
\newcommand{\bphi}{{\boldsymbol{\phi}}}
\newcommand{\bmu}{{\boldsymbol{\mu}}}

\newcommand{\br}{{\bf {R}}}

\newcommand{\rr}{{\bf {r}}}

\newcommand{\bu}{\boldsymbol{u}}
\newcommand{\ds}{{\rm ds}}

\newcommand\emb {{\hookrightarrow}}

\newcommand\simcl {{\Delta_\bullet}}
\newcommand\simap {{\Delta_\circ}}

\def\wto{\rightharpoonup}
\def\wstarto{\stackrel{*}{\rightharpoonup}}

\def\Accorpa #1#2 #3 {\gdef #1{\eqref{#2}--\eqref{#3}}%
  \wlog{}\wlog{\string #1 -> #2 - #3}\wlog{}}

\begin{document}

\title{On a phase field model for RNA-Protein dynamics
	\\
}
\author{Maurizio Grasselli \footnotemark[1] \and Luca Scarpa \footnotemark[1] \and Andrea Signori \footnotemark[1] }
\date{}
\maketitle

\renewcommand{\thefootnote}{\fnsymbol{footnote}}
\footnotetext[1]{Dipartimento di Ma\-te\-ma\-ti\-ca, Politecnico di Milano, 20133 Milano, Italy 
({\tt mau\-ri\-zio.gras\-sel\-li@polimi.it, luca.scarpa@polimi.it, andrea.signori@polimi.it}).}

\centerline{{\sl to GIANNI GILARDI on the occasion of his 75th birthday}}

\medskip
\centerline{{\sl keen mathematician, inspiring teacher, friend}}

\begin{abstract}\noindent
A phase field model which describes the formation of protein-RNA complexes subject to phase segregation is analyzed.
A single protein, two RNA species, and two complexes are considered. Protein and RNA species are governed by coupled reaction-diffusion equations which also depend on the two complexes. The latter ones are driven by two Cahn--Hilliard equations with Flory--Huggins potential and reaction terms depending on the solution variables. The resulting nonlinear coupled system is equipped with no-flux boundary conditions and suitable initial conditions. The former ones entail some mass conservation constraints which are also due to the nature of the reaction terms. The existence of global weak solutions in a bounded {(two- or)} three-dimensional domain is established. In dimension two, some weighted-in-time regularity properties are shown. Moreover, making a suitable approximation of the potential, the complexes instantaneously get uniformly separated from the pure phases.
Taking advantage of this result, a unique continuation property is proven.
Among the many technical difficulties, the most significant one arises from the fact that the two complexes are initially nonexistent,
so their initial conditions are zero i.e., they start from a pure phase.
Thus we must solve, in particular, a system of two coupled Cahn--Hilliard equations with singular potential and nonlinear sources without the usual assumption on the initial datum, i.e., the initial phase cannot be pure. This novelty requires a new approach to estimate the chemical potential in a suitable $L^p(L^2)$-space with $p\in(1,2)$. This technique can be extended to other models like, for instance, the well-known Cahn--Hilliard--Oono equation.
\end{abstract}


\vspace{2mm}


\vspace{2mm}

\section{Introduction}

Phase separation has recently become a paradigm in Cell Biology and, in particular, in intracellular organization (see, e.g., \cite{Dol} and
references therein, see also \cite{Al, BTP, CDLLS, HWJ, Kar, SB}). The latter can be described through a molecular mechanism based on proteins and RNA. They combine to form protein-RNA complexes which then form droplets through a phase separation process. A relatively simple but effective phase field model was recently proposed in \cite{GFGN}. The ingredients are a free protein, two RNA species, and two protein-RNA complexes (see \cite{GFGNrev} for a simpler model with just one RNA species). The nondimensional concentrations of the first three variables are denoted by $P$, $R_1$, and $R_2$, while $\varphi_1$ and $\varphi_2$ stand for the volume fractions of the complexes formed by $P$ and $R_1$, and $P$ and $R_2$, respectively. The interaction of the protein with the RNA species in a given solvent is ruled by a coupled system of reaction-diffusion equations whose diffusion coefficients depend on $\varphi_1$
and $\varphi_2$, while the reaction terms depend on all the variables. The complexes instead are governed by a system of Cahn--Hilliard equations
with singular potential and reaction terms related to the previous ones. More precisely, letting $\Omega \subset \erre^d$, $d \in \{2,3\}$, be a bounded domain with boundary $\Gamma$, $T>0$ be a fixed final time, and setting
\begin{align*}
	\bph:=(\ph_1, \ph_2),
	\quad 	
	\bmu:=(\mu_1, \mu_2),
	\quad 	
	\br:=(R_1, R_2),
\end{align*}
the resulting system reads as follows
\begin{alignat}{2}
	\label{SYS:1}
	&\dt \bph
	-\div (\CC(\bph) \nabla \bmu)
	=  \Sph(\bph, P, \br)
		&&\qquad \text{in $Q$},
		\\	\label{SYS:2}
	&\bmu = -\eps^2 \Delta \bph + A \Psi_{\bph}(\bph)
		&&\qquad \text{in $Q$},
		\\
	\label{SYS:3}
	&\dt P -\lambda_P\div (  m(\bph) \nabla P)
	=  \SP(\bph, P, \br)
		&&\qquad \text{in $Q$},\\
	\label{SYS:4}
	&\dt \br -\div ( \DD(\bph) \nabla  \br)
	=  \SR(\bph, P, \br)
		&&\qquad \text{in $Q$},
\end{alignat}
where $Q:= \Omega \times (0,T)$. Here, $\eps>0$ is related to the thickness of the diffuse interface separating the two complexes, $A>0$ is a scaling factor and $\lambda_P$ denotes the diffusion rate of free protein
(see \cite{GFGN} and references therein),
while $\CC(\cdot), \DD(\cdot)$ and $m(\cdot)$ denote the mobility tensors and the scalar mobility associated with $\bph, \br$, and $P$, respectively.
For instance, as suggested in \cite{GFGN},
$\CC(\cdot)$ and $\DD(\cdot)$ may be uniformly positive definite second-order tensors
of the form
\begin{align*}
	\CC(\rr) =
		\begin{pmatrix}
			\lambda_{R_1}m(\rr) & 0  \\
			0 & \lambda_{R_2} m(\rr)
		\end{pmatrix},
		\qquad
		\DD(\rr) =
		\begin{pmatrix}
			\lambda_{\ph_1}m(\rr) & 0  \\
			0 & \lambda_{\ph_2} m(\rr)
		\end{pmatrix},
		\qquad \rr \in \erre^2,
\end{align*}
where $m:\erre^2 \to \erre$ denotes a mobility function and $\lambda_i$, $i \in \{\ph_1,\ph_2,R_1,R_2\}$, are some fixed and positive constants. Here, if $\SB = [S_{hk}]$ is a second-order
tensor and ${\bf F}=[F_k]$ is a vector-valued function, $h,k\in\{1,2\}$  then $\SB\nabla {\bf F}= [S_{hk}\nabla F_k]$.
We recall that, in the multiphase approach, the natural space where the vector-valued order parameter $\bph$ takes its values is
the subset of $\erre^2$ of vectors with nonnegative components that, at most, add up to one (cf. \eqref{simplex}).

The reaction terms $\Sph, \SP$, and $\SR$ are defined as follows
\begin{align}	
	\label{def:sorgenti:intro}
	\Sph (\bph, P, \br)& = - \SR (\bph, P, \br)=(c_1 P R_1 - c_2 \ph_1,c_3 P R_2 - c_4 \ph_2),
	\\
	\label{def:Sp:intro}
	\SP(\bph, P, \br) & = -c_1 P R_1 + c_2 \ph_1 - c_3 P R_2 + c_4 \ph_2,
\end{align}
where $c_i$, $i=1,\dots,4$, are fixed and positive constants,
while $\Psi$ denotes a  Flory--Huggins potential. More precisely, we have $\Psi=\Psi^{(1)} + \Psi^{(2)}$, where $\Psi^{(1)}$ stands for the mixing entropy (extended by continuity, as customary)
\begin{align}
	\label{Flog}
	\Psi^{(1)}(\bph) & =
	\begin{cases}
	\sum_{i=1}^2\ph_i \ln \ph_i + S \ln S
	\quad&\text{if } \ph_1\geq0,\;\ph_2\geq0,\;\ph_1+\ph_2 { \leq1},
	\\
	+\infty \quad&\text{otherwise,}
	\end{cases}
\end{align}
where $S$ indicate the solvent which is defined as
\begin{align}
	\label{solvent}
	S & = 1- \sum_{i=1}^2 \ph_i,
\end{align}
whereas $\Psi^{(2)}$ denotes the demixing term which satisfies the following assumptions
\begin{align*}
	\Psi^{(2)}\in C^2(\erre^2),
	\qquad
	\Psi^{(2)}_\bph:\erre^2 \to \erre^2 \text{ is \Lip\ continuous}.
\end{align*}
For simplicity, we use the notation $\Psi_\bph := \nabla \Psi$ and $\Psi_{\bph\bph} := D^2 \Psi$ for the gradient and
the Hessian of $\Psi$, respectively. We recall that in the original model the authors take (see \cite[Eq.~(6)]{GFGN})
$$
{\Psi^{(2)}(\bph)}= \chi_{\ph_1 \ph_2} \ph_1\ph_2 + \chi_{\ph_1 S} \ph_1 S + \chi_{\ph_2 S}\ph_2S{,}
$$
where $\chi_{\ph_1 \ph_2},\chi_{\ph_1 S}$ and $\chi_{\ph_2 S}$ are nonnegative constants.
Moreover, system \eqref{SYS:1}--\eqref{SYS:4} is equipped with the following boundary and initial conditions
\begin{alignat}{2}
	\label{SYS:5}
	& \dn\bph =  ( \CC(\bph) \nabla  \bmu){\bf n} = ( \DD(\bph) \nabla  \br){\bf n} =
		 \0
		&&\qquad \text{on $\Sigma$},\\
	\label{SYS:6}
	& (m(\bph) \nabla P) \cdot {\bf n} = 0
		&&\qquad \text{on $\Sigma$},\\
	\label{SYS:7}
	&\bph(0)
	= \bph_0, \quad
	\br(0)
	= \br_0
	&&\qquad \text{in $\Omega$},\\
	\label{SYS:8}
	&P(0)
	= P_0
	&&\qquad \text{in $\Omega$},
	\end{alignat}
\Accorpa\SYS {SYS:1} {SYS:8}
where $\Sigma:= \Gamma \times (0,T)$, $\bf n$ denotes the outer unit normal vector to $\Gamma$,
and $\dn$ the associated normal derivative. Observe that the multi-component
Cahn--Hilliard system \eqref{SYS:1}--\eqref{SYS:2} differs from the three-component
system where also the third component undergoes phase separation
(see \cite{EG97,EL91}).

Concerning the initial data, a biologically relevant choice {(see \cite[Sec.~2]{GFGN})} is the following
\begin{align*}
	\bph(0) = \0,
	\quad
	\br(0) = \Big(\frac {1-P_0}2\Big) \1,
	\quad P(0)=P_0 \in (0,1),
\end{align*}
where $\0=(0,0)$ and $\1=(1,1)$.
The reason is that the initial scenario presents a certain amount of protein and RNA, but no complexes:
{indeed, complexes are formed exclusively} by the
interaction between protein and RNA species. As we {will } see, the initial absence of complexes is a major technical difficulty from a mathematical viewpoint.




Observe that the reaction terms \eqref{def:sorgenti:intro}--\eqref{def:Sp:intro} satisfy the identities
\begin{align}\label{sources:struct}
	\Sph + \SR =\0, \quad -\Sph \cdot \1 =  \SR \cdot \1 = \SP.
\end{align}
These  conditions play a fundamental role in the mass conservation properties of the occurring variables.
Indeed, consider the equation \eqref{SYS:1} and add it to \eqref{SYS:4}.
Integration over $\Omega\times [0,t]$, for an arbitrary $t \in [0,T]$, and multiplication by the vector $\1$, along with \eqref{SYS:5}, and the initial conditions \eqref{SYS:7}--\eqref{SYS:8}, entail
the following
\begin{align*}
	\iO  (\bph + \br)(t)\cdot\1 =
	\iO  \br(0)\cdot\1 = \iO \br_0\cdot\1 \quad \text{for all $t \in [0,T]$.}
\end{align*}
This means that the mass of the linear combination $(\bph+ \br)\cdot\1$ is preserved along the evolution.
Arguing similarly, we now multiply \eqref{SYS:1} by $\1$, \eqref{SYS:3} by $1$, integrate over $\Omega \times [0,t]$, for $t \in [0,T]$, and add the resulting equalities. This gives
\begin{align*}
	\iO  (\bph\cdot\1 + P)(t) =
	\iO   P(0) = \iO P_0 \quad \text{for all $t \in [0,T].$}
\end{align*}
Thus the total mass of $\bph\cdot\1+ P$ is also conserved. On the other hand, repeating the same argument with \eqref{SYS:3} tested by $1$ and \eqref{SYS:4} tested by $-\1$, we deduce that
\begin{align*}
	\iO  ( P - \br\cdot\1)(t) =\iO  (P -\br\cdot\1 )(0) = \iO (P_0 -\br_0\cdot\1) \quad \text{for all $t \in [0,T]$}.
\end{align*}
Thus, we set
\[
  {\hat \ph_\Omega}:=\frac1{|\Omega|}\int_\Omega\bph\cdot\1, \qquad
 {\hat  R_\Omega}:=\frac1{|\Omega|}\int_\Omega\br\cdot\1, \qquad
  \hat P_\Omega:=\frac1{|\Omega|}\int_\Omega P{,}
\]
where $|\Omega|$ indicates the Lebsesgue measure of $\Omega$. Upon
combining the properties {above}, we derive that the global mass balance for the system \SYS\ is ruled by
\begin{align*}
	\dt (2 {\hat \ph_\Omega} + {\hat R_\Omega} + \hat P_\Omega) = 0
	\quad \text{ in $[0,T]$, }
\end{align*}
{while the partial mass balances read}
\begin{align*}
	\dt ({\hat \ph_\Omega} + {\hat R_\Omega}) & = 0,
	\quad
	\dt ({\hat \ph_\Omega} + \hat P_\Omega)  = 0,
	\quad
	\dt ({\hat R_\Omega} -\hat P_\Omega)=0
	\quad \text{ in $[0,T]$. }
\end{align*}


Without loss of generality we can take $\eps =A= \lambda_P = 1$ (see \eqref{SYS:2} and \eqref{SYS:3}). Concerning mobility, for the sake of simplicity, we assume the mobility tensors are given by the identity tensor and we take constant scalar mobility. More precisely, we assume hereafter that $\DD\equiv \CC\equiv \mathbb{I}$ and $m \equiv 1$.

Cahn--Hilliard equations with singular potential and reaction terms are usually difficult to handle. This can be understood even by looking at one of the simplest, albeit significant, cases, namely, the Cahn--Hilliard--Oono equation (see \cite{GGM}).
On the other hand, reaction terms arise quite naturally, for instance, in solid tumor growth modeling
{\cite{CL, GLSS}}. In this context, when dealing with biological materials reaction terms depend, for instance, on nutrients whose evolution is governed by suitable reaction-diffusion equations. Therefore, these scenarios are quite similar to the present one. In all these models,
the standard assumption on the initial datum is that the total mass must belong to $(0,1)$, that is, one cannot start from a pure phase. This is trivial in the case of no reaction terms, since
in {such a circumstance} no phase separation can take place.
Nevertheless, the reaction may activate phase separation as in the present case or, also, in the case of tumor growth. Thus it makes sense that the phase (i.e., the complex) is initially zero.
{From a mathematical viewpoint,} in presence of no-flux boundary conditions, the usual approach does not work. Indeed, an essential ingredient is to recover an $L^2$-bound for the chemical potential. This bound cannot be recovered from the $L^2$-bound of its gradient because of the boundary conditions. Hence, one is forced to estimate the spatial average of the chemical potential and, in particular, of each component of $\Psi_{\bph}(\bph)$. This is the point where the assumption about the total mass of each complex must belong to $(0,1)$ comes into play. Here we show how this assumption can be avoided and prove nonetheless the existence of global weak solutions. This is the main result of the paper and its proof is carried out in Section \ref{sec:ex_weak}: {one of the crucial mathematical tools is the refined
Property~\ref{MZ} of the mixing entropy $\Psi^{(1)}$ contained in Proposition~\ref{prop:log},
which is carefully proved in the Appendix~\ref{app}.}
Then, in Section~\ref{SEC:PAR3d}, we establish some regularity results. Moreover, in dimension two, making a suitable approximation of $\Psi^{(1)}$, we show a weighted-in-time regularity and prove that each complex is always uniformly separated from pure phases after an arbitrarily short time, i.e., the strict separation property holds (see, \cite{GGM,MZ,GGG}). In other words, no complex fully prevails or disappears.
This property is used in Section~\ref{SEC:UQ} to show unique continuation for weak solutions. This means that if two different patterns originating from the same initial data are equal at a certain time, then they coincide from that time onward.

All the described results will be stated in the next section. Section \ref{SEC:CONC} is devoted to some remarks on further possible applications of the technique developed here, {to Cahn--Hilliard-type equations with reaction terms}
starting from a pure initial phase. We analyze, in particular, the Cahn--Hilliard--Oono equation
(see Section~\ref{SEC:CHO}) and some tumor growth models (see Section~\ref{SEC:TUMOR}).

Let us conclude this Introduction by pointing out that uniqueness in three dimensions remains an open issue for the moment: {indeed, } we are not able to avoid the use of the strict separation property whose validity is unknown in dimension three (see Remark~\ref{REM:UQ} below) and, also in dimension two, its validity requires a further modification of $\Psi^{(1)}$. Another, possibly challenging, open issue is the extension of our results to the case of nonconstant, albeit nondegenerate, mobilities (see, for instance, \cite[Eq.~(9)]{GFGN}).

\section{Notation and main results}

\subsection{Notation} \label{SEC:NOT}

For a given (real) Banach space $X$, we denote its associated norm by $\norma{\,\cdot\,}_X$, its topological dual space by $X^*$, and the associated duality pairing
by $\<\cdot,\cdot>_X$. If $X$ is, in particular, a Hilbert space, we denote its inner product by $(\cdot, \cdot)_X$.
For any $1 \leq p \leq \infty$ and $k \geq 0$, the standard Lebesgue and Sobolev spaces defined on $\Omega$ are denoted as $L^p(\Omega)$ and $W^{k,p}(\Omega)$, and the corresponding norms are indicated with $\norma{\,\cdot\,}_{L^p(\Omega)}$ and $\norma{\,\cdot\,}_{W^{k,p}(\Omega)}$, respectively.
In the special case $p = 2$, these turn into Hilbert spaces and we use the standard notation $H^k(\Omega) := W^{k,2}(\Omega)$ for any $k \geq 1$.
To denote the corresponding spaces of $\erre^2$-vector- or $\erre^{2\times2}$-matrix-valued functions we employ the notation $\LL^p(\Omega)$, $\WW^{k,p}(\Omega)$, and $\HH^{k}(\Omega)$.
Bold letters are also used to denote vector- or matrix-valued elements.

For given matrices ${\A,\B}\in \erre^{2\times 2 }$, we define the scalar product
\begin{align*}
	{\A : \B} \,:= \sum_{i=1}^2 \sum_{j=1}^2 [{\A}]_{ij}\, [{\B}]_{ij}\;,
\end{align*}
and the operator norm
\begin{align*}
	\opnorm{\cdot} : \erre^{2 \times 2} \to [0, + \infty), \qquad \opnorm{{\A}} := \sup_{\rr \in \erre^2:\,\, |\rr|=1} |{\A}\rr| = \sup_{\rr\in \erre^2\setminus \{\0\}} \frac {|{\A}\rr|}{|\rr|}\mau{,}
\end{align*}
where $|\cdot |$ stands for the Euclidean norm.

We will also make use of the following notation
\begin{align*}
  & H := \Lx2, \quad
  V := \Hx1,
  \quad
  \HH : = \LL^2(\Omega),
  \quad
  \VV: = \HH^1(\Omega),
\end{align*}
and for both $H$ and $\HH$ we indicate by $(\cdot,\cdot)$ and $\norma{\,\cdot\,}$ their inner product and norm, respectively.
Besides, we need to introduce the following spaces
\begin{align*}
	W := \{ f \in \Hx2 : \dn f = 0 \quad \text{a.e.~on $\Gamma$} \},
	\quad
	\WW := \{ {\bf f} \in \HH^2(\Omega) : \dn {\bf f} = \0 \quad \text{a.e.~on $\Gamma$} \}.
\end{align*}

For elements $v\in V^*$ and $ \bv \in \VVp$ we set
\begin{align*}
	v_\Omega:=\frac1{|\Omega|}\<{v},{1}>_V,
	\qquad
	\bv_\Omega:=\frac1{|\Omega|}
	(\<{\bv},{(1,0)}>_\VV, \<{\bv},{(0,1)}>_\VV),
	\qquad  v \in V,\; \bv \in \VV,
\end{align*}
and we recall the Poincar\'e--Wirtinger inequalities
\begin{alignat}{2}
  \label{poincare}
  \norma{v}^2_V &\leq c_{\Omega} \bigl( \norma{\nabla v}_H^2 + |v_\Omega|^2 \bigr),
  \qquad&&v\in V,\\
  \norma{\bv}^2_\VV &\leq c_{\Omega} \bigl( \norma{\nabla \bv}_\HH^2 + |\bv_\Omega|^2 \bigr),
  \qquad&&\bv\in \VV,
  \label{poincare2}
\end{alignat}
where the constant $c_\Omega>0$ depends only on $\Omega$ and on the space dimension $d$.
Moreover, we set
\begin{alignat*}{2}
	V_0 &:= \{ v \in V : v_\Omega =0\}, \qquad V_0^*&&:= \{ v^* \in V^* : v^*_\Omega =0\},\\
	\VV_0 &:= \{ \bv \in \VV : \bv_\Omega = \0\}, \qquad
	\VV_0^*&&:= \{ \bv^* \in \VV^* : \bv^*_\Omega =\0\},
\end{alignat*}
and define the following linear bounded operators $\RR : V \to \Vp$ and $\bRR : \VV \to \VV^*$
\begin{alignat*}{2}
\<{\RR  u},{v}>_V &:= \int_\Omega \nabla u \cdot \nabla v  , &&\quad u,v \in V,\\
\<{\bRR  \bu},{\bv}>_\VV &:= \int_\Omega \nabla \bu : \nabla \bv  , &&\quad \bu,\bv \in \VV.
\end{alignat*}
It is well-known that the restrictions $\RR \vert_{V_0}$ and
$\bRR \vert_{\VV_0}$
are isomorphisms from $V_0$ to $V_0^*$
and from $\VV_0$ to $\VV_0^*$, respectively,
with well-defined  inverses
$\RR^{-1}=:\NN : V_0^* \to V_0$ and
$\bRR^{-1}=:\bNN : \VV_0^* \to \VV_0$.
Furthermore, it holds that
\begin{alignat*}{2}
	\<{\RR  u},{\NN v^*}>_V &= \<{v^*},{u}>_V,
	\qquad && u \in V, \; v^* \in V_0^*,\\
	\<{\bRR  \bu},{\bNN \bv^*}>_\VV &= \<{\bv^*},{\bu}>_\VV,
	\qquad && \bu \in V, \; \bv^* \in \VV_0^*,
\end{alignat*}
and
\begin{alignat*}{2}
	\<{v^*},{\NN w^*}>_V &= \iO  \nabla (\NN v^*) \cdot \nabla (\NN w^*),
	\qquad  && v^*,w^* \in V_0,\\
	\<{\bv^*},{\bNN \bw^*}>_\VV &= \iO  \nabla (\bNN \bv^*) : \nabla (\bNN \bw^*),
	\qquad  && \bv^*,\bw^* \in \VV_0.
\end{alignat*}
Again, it is well-known that
\begin{alignat*}{2}
	\norma{v^*}_* &:= \norma{\nabla (\NN v^*)} =
	({\nabla (\NN v^*)},{\nabla (\NN v^*)}) ^{1/2}=\<v^*, \NN v^*>_V^{1/2},
	\qquad &&v^*\in V_0^*,\\
	\norma{\bv^*}_\bstar &:= \norma{\nabla (\bNN \bv^*)} =
	({\nabla (\bNN \bv^*)},{\nabla (\bNN \bv^*)}) ^{1/2}=\<\bv^*, \bNN \bv^*>_\VV^{1/2},
	\qquad &&\bv^*\in \VV_0^*,
\end{alignat*}
are norms in $V_0^*$ and $\VV_0^*$, respectively,
with associated inner products denoted by
$(\cdot, \cdot )_{*}:= (\nabla (\NN \cdot),\nabla (\NN \cdot))$, and
$(\cdot, \cdot )_{\bstar}:= (\nabla (\bNN \cdot),\nabla (\bNN \cdot))$.
Similarly, setting
\begin{alignat*}{2}
	\norma{v^*}_{-1} &:= \norma{v^*-(v^*)_\Omega}_* +|(v^*)_\Omega|,
	\qquad&&v^*\in V^*,\\
	\norma{\bv^*}_{-\1} &:= \norma{\bv^*-(\bv^*)_\Omega}_\bstar +|(\bv^*)_\Omega|,
	\qquad&&\bv^*\in \VV^*,
\end{alignat*}
we get equivalent norms in $V^*$ and $\VV^*$, respectively.

We identify $H$ and $\HH$ with their duals through
their natural Riesz isomorphisms, so that $V\emb H\emb V^*$
and $\VV\emb\HH\emb\VV^*$ continuously, compactly, and densely.
In particular, we have the identifications
\begin{alignat*}{2}
\<{u},{v}>_V &= \int_\Omega u v, \quad &&u\in H,\;v\in V, \\
\<{\bu},{\bv}>_\VV &= \int_\Omega \bu \cdot\bv, \quad &&\bu\in \HH,\;\bv\in \VV,
\end{alignat*}
which imply the following interpolation inequalities:
\begin{align}
\label{interpol_one}
\norma{v} &= \<{v},{v}>_V ^{1/2}= \<{\RR  v},{\NN v}>_V^{1/2} \leq \norma{v}_{V}^{1/2} \norma{\NN  v}_{V}^{1/2} \leq \norma{v}_{V}^{1/2} \norma{v}_*^{1/2},\\
\label{interpol_two}
\norma{\bv} &= \<{\bv},{\bv}>_\VV ^{1/2}= \<{\bRR  \bv},{\bNN \bv}>_\VV^{1/2}
\leq \norma{\bv}_{\VV}^{1/2} \norma{\bNN  \bv}_{\VV}^{1/2}
\leq \norma{\bv}_{\VV}^{1/2} \norma{\bv}_\bstar^{1/2}.
\end{align}
Furthermore, we recall the identities
\begin{alignat*}{2}
\<{\dt v^*(t)},{\NN v^*(t)}>_V &= \frac 12 \frac d {dt} \norma{v^*(t)}^2_{*}
\quad \text{ for a.e.~}t \in (0,T), \quad&\forall\,v^* \in \H1 {V^*_0},\\
\<{\dt \bv^*(t)},{\bNN \bv^*(t)}>_\VV &= \frac 12 \frac d {dt} \norma{\bv^*(t)}^2_{\bstar}
\quad \text{ for a.e.~}t \in (0,T), \quad&\forall\,\bv^* \in \H1 {\VV^*_0},
\end{alignat*}
and the inequalities
\begin{align}\label{oss:mean}
	| v_{\Omega}|\leq c \norma{v}_{\Vp}
	\quad\forall\, v\in V^* \qquad\text{and}\qquad
	| \bv_{\Omega}|\leq c \norma{\bv}_{\VV^*}
	\quad\forall\, \bv\in \VV^*.
\end{align}
Finally, we will also use the following inequalities which
readily follows from Ehrling's lemma and the compactness of the inclusions
$V\emb H$ and $\VV\emb\HH$:
\begin{alignat}{2}	
	\label{comp:ineq1}
	\forall\,\delta>0,\quad\exists\,\cd>0:\quad
	&\norma{v} \leq \d  \norma{\nabla v} + \cd \norma{ v}_{-1} \qquad&&\forall\,v \in V,\\
	\label{comp:ineq2}
	\forall\,\delta>0,\quad\exists\,\cd>0:\quad
	&\norma{\bv} \leq \d  \norma{\nabla \bv} + \cd \norma{ \bv}_{-\1} \qquad&&\forall\,\bv \in\VV.
\end{alignat}

\subsection{Main result and some regularity properties}

In order to state our main result, we need to introduce some further notation as well as some structural assumptions on the nonlinear terms.

The simplex of admissible configurations for the complexes is defined by
\begin{align}\label{simplex}
	\simap := \{ \rr=(r_1,r_2) \in \erre^2 : \min\{r_1,r_2\} > 0, r_1 +r_2 < 1\},  \quad 	\simcl  := \overline{\simap}.
\end{align}

\begin{enumerate}[label={\bf A\arabic{*}}, ref={\bf A\arabic{*}}]
\item \label{ass:pot}
We suppose that the configuration potential can be written as $\Psi=\Psi^{(1)}+ \Psi^{(2)} $, where
\begin{align*}
  & \Psi^{(1)} \in C^0(\simcl )\cap C^2(\simap),\qquad \Psi^{(2)}\in {C^2(\erre^2)};\\
  & \Psi^{(1)}:\simcl \to [0,+\infty] \quad\text{is proper, convex and satisfies Properties~\ref{MZ0}--\ref{MZ} listed below};\\
  &\Psi_{\bph}^{(2)}:{\erre^2}\to\erre^2 \quad\text{is $L_0$-Lipschitz continuous for some $L_0>0$.}
\end{align*}

\begin{property}\label{MZ0}
  For every compact subset $K\subset \simap$,
  there exist constants $c_K, C_K>0$ such that,
  for every $\rr\in\simap$ and for every $\rr_0\in K$ it holds that
  \begin{equation}\label{MZ_base}
  c_K|\Psi^{(1)}_\bph(\rr)|
  \leq \Psi^{(1)}_\bph(\rr)\cdot(\rr-\rr_0)
  +C_K.
  \end{equation}
\end{property}
\begin{property}\label{MZ}
  There exist constants $c_\Psi, C_\Psi>0$, $R\in(0,1/2)$, $q\in(2,+\infty)$,
  and a decreasing positive function $F\in L^q(0,R)\cap C^0(0,R)$ such that, for every measurable
  $\bphi=(\phi^1,\phi^2):\Omega\to\simap$ satisfying
  \[
  0<\min\{\phi^1_\Omega, \phi^2_\Omega\}<\phi^1_\Omega + \phi^2_\Omega \leq R,
  \]
  it holds that
  \begin{equation}\label{MZ_gen}
  \begin{split}
  &c_\Psi\min\{\phi^1_\Omega, \phi^2_\Omega\}\int_\Omega|\Psi^{(1)}_\bph(\bphi)|\\
  &\quad \leq \int_\Omega\Psi^{(1)}_\bph(\bphi)\cdot(\bphi-\bphi_\Omega)
  +C_\Psi(\phi^1_\Omega + \phi^2_\Omega)
  \left[1 + F(\min\{\phi^1_\Omega, \phi^2_\Omega\})\right].
  \end{split}
  \end{equation}
\end{property}
Let us point out straightaway that Property~\ref{MZ0}
is a well-known inequality in the literature on Cahn--Hilliard equations
{(see e.g.~\cite[Sec.~5]{GMS09} and \cite{MZ}).}
Property~\ref{MZ} is a sharp generalization of it, in which
the compact $K$ shrinks to the origin, and where the constants
are explicitly estimated: {due to the vanishing initial datum for the phase variables,
this will be crucial in order to handle
the spatial mean of the chemical potential close to the initial time.}
It is well-known that the physically relevant potential \eqref{Flog} satisfies Property~\ref{MZ0}: see, for instance, \cite[Lem.~2.2]{FLRS} and references therein (see also \cite[(3.36)]{FG} for the corresponding scalar case).
Actually, we prove that it also satisfies the more general condition Property~\ref{MZ} which helps us to describe the time integrability properties of weak solutions for small times (see Proposition~\ref{prop:log} below).

\item \label{ass:sources}
The reaction terms
\begin{align*}
	\Sph, \SR:\erre^2 \times \erre \times \erre^2 \to \erre^2,
	\qquad
	\SP:\erre^2 \times \erre \times \erre^2 \to \erre,
\end{align*}
are defined for $(\bphi,p,\rr)\in\erre^2\times\erre\times\erre^2$ as
\begin{align}	
	\label{def:sorgenti}
	\Sph (\bphi, p, \rr)& = - \SR (\bphi, p, \rr)=(c_1 p r_1 - c_2 \phi_1,c_3 p r_2 - c_4 \phi_2),
	\\
	\label{def:Sp}
	\SP(\bphi, p, \rr) & = -c_1 p r_1 + c_2 \phi_1 - c_3 p r_2 + c_4 \phi_2,
\end{align}
where $c_i$, $i=1,\dots,4$, are fixed and positive constants satisfying
\begin{align}\label{c}
	\min\{c_1,c_3\}>c_2+c_4.
\end{align}

\end{enumerate}

The following result will be proven in Appendix~\ref{app}.
\begin{proposition}\label{prop:log}
The Flory--Huggins potential defined in \eqref{Flog} satisfies Property~\ref{MZ} with the choices
$R=1/8$, $F(r):=| \ln (r/2) \,|$, $r\in(0,R)$, and for any $q\in(2,+\infty)$.
\end{proposition}

It is worth noting that assumption \ref{ass:pot} entails that $\Psi$ is a quadratically perturbed convex function,
so $\Psi$ is generally not purely convex, namely, the demixing part is accounted for.

We can now state our main result.

\begin{theorem}[Existence of global weak solutions]
\label{THM:EX:WEAK}
Let \ref{ass:pot}--\ref{ass:sources} hold.
Moreover, suppose that
\begin{align}
	\label{ass:exweak:initialdata}
	&\bph_0:=\0\in\VV, \qquad P_0 \in H, \qquad \br_0 = (R_{0}^{1},R_{0}^{2}):=\Big(\frac{1-P_0}2\Big)\1 ,\\
	\label{P:minmax:ini}
	&0 \leq P_0 (x)\leq 1
	\quad\textrm{for a.a.~$x\in \Omega$},\\
    \label{max:ini:Linf}
	&\Big \| P_0-\frac 12 \Big\|_{L^\infty(\Omega)} \leq
	\frac 12 \Big(1 - \frac {c_2 + c_4}{\min\{c_1,c_3\}}\Big).
\end{align}
Then, there exists $(\bph, \bmu, P, \br)$ such that
\begin{align*}
	& \bph \in \H1 {\VVp} \cap \L\infty {\VV} \cap \L2 {\WW},\\
	&\bph \in \Ls4 {\WW} \cap \Ls2 {\WW^{2,r}(\Omega)} \quad\forall\,\sigma\in(0,T),\\
	&\bph \in L^{2p}(0,T; \WW) \cap L^{p}(0,T;\WW^{2,r}(\Omega))\quad\forall\,p\in(1,2),\\
	& \bph \in  \LL^\infty(Q): \quad
	\bph(x,t) \in \simap \quad \textrm{for a.a.~$(x,t) \in Q$},\\
	& \bmu \in  L^p(0,T;\VV)\cap L^2(\sigma,T;\VV) \quad\forall\,p\in(1,2)\quad\forall\,\sigma\in(0,T),
	\quad \nabla\bmu \in L^2(0,T;\HH),\\
	& P  \in	\H1 {\Vp} \cap \L2 {V}\cap L^\infty(Q),
	\\
	& \br  \in \H1 {\VVp} \cap \L2 {\VV}\cap \LL^\infty(Q),
	\\
	&\exists\,c_*>0:\quad c_* < P(x,t) \leq1,  \quad c_* < R_i(x,t) \leq1 \quad
	\textrm{for a.a.~$(x,t) \in Q$}, \quad i=1,2,
\end{align*}
with $c_*\in \big(0,\frac{c_2+c_4}{8\min\{c_1,c_3\}}\big]$,
where $r=6$ if $d=3$ and $r$ is arbitrary in $(1,+\infty)$ if $d=2$,
which is a weak solution to problem \eqref{SYS:1}--\eqref{SYS:4}, \eqref{SYS:5}--\eqref{SYS:8} in the following sense{:}
\begin{align}
	\label{wf:mu}
	\bmu = -\Delta \bph + \Psi_{\bph}(\bph) \qquad \text{\aeQ},
\end{align}
the variational identities
\begin{align}
	\label{wf:1}
	& \<\dt \bph ,\bv>_{\VV}
	+ \iO \nabla \bmu : \nabla \bv
	=
	\iO \Sph (\bph, P, \br)\cdot \bv,
	\\ \label{wf:P}
	& \<\dt P,v> _V +\iO  \nabla  P \cdot \nabla v = \iO \SP(\bph, P, \br) v,
	\\
	\label{wf:R}
	& \<\dt \br,\bv>_{\VV}  +\iO \nabla \br : \nabla \bv= \iO \SR(\bph, P, \br)\cdot \bv ,
\end{align}
are satisfied for every test function
$ \bv\in \VV$, $v\in V$, almost everywhere in $(0,T)$, and the following initial conditions hold
\begin{align*}
	\bph(0)=\0,
	\quad
	P(0)=P_0,
	\quad
	\br(0)=\Big(\frac {1-P_0}2\Big) \1
	\qquad \text{a.e.~in $\Omega.$}
\end{align*}
\end{theorem}

\begin{remark}
\label{iniP}
Let us point out that condition \eqref{max:ini:Linf} ensures
that the initial protein concentration $P_0$ allows the complex formation process to take place.
In particular, on account of \eqref{c}, $P_0$ is strictly separated from $0$ and $1$.
This fact is essential
in order to prove that the variables take their values in biologically-relevant range $[0,1]$.
Without \eqref{c}, the variables do not represent concentrations anymore. In addition,
the problem becomes mathematically intractable due to a lack of $L^\infty$-control on $P$ and $\br$.
Let us stress, nonetheless, that as strange as conditions \eqref{c} and \eqref{max:ini:Linf} may appear,
in the typical biological applications they are actually always satisfied. Indeed,
referring to {\cite[Table~1]{GFGN}} and recalling \eqref{c}, the threshold of separation
$(c_2+c_4)/\min\{c_1,c_3\}$ is of the order of {$10^{-2}$},
so that \eqref{max:ini:Linf} looks quite natural.
Besides, condition \eqref{max:ini:Linf} also plays a basic role in proving the min-max principles for $P$, $R_1$ and $R_2$ as well as
in analyzing the mass evolution of $\bph$ whose behavior is ruled by the source term $\Sph$ (see Section \ref{SEC:MEAN}).

\end{remark}

\begin{remark}
As mentioned in the Introduction, the full $\L2 \VV$-norm of $\bmu$ cannot be recovered as it usually happens
for similar Cahn--Hilliard type systems.
The main reason is that classical energy estimates give only an $\L2\HH$-bound of $\nabla \bmu$ (cf.~\eqref{est4}).
In the usual cases, for the Flory--Huggins potential \eqref{Flog} it is still possible to
obtain the expected $\L2 \VV$-regularity on $\bmu$ by using well-known methods based on the Property~\ref{MZ0}
(see also \cite{MZ}) combined with the Poincar\'e inequality,
provided that the mean value of
$\bph$ is uniformly confined in the interior of the simplex $\simap$.
This last condition is what is missing in our scenario, as the initial condition
$\bph(0)=\0$ in \eqref{ass:exweak:initialdata} entails $\bph_{\Omega}(0)=\0$.
Nonetheless, as soon as the evolution takes place, i.e., for $t\geq\sigma$
for any $\sigma\in(0,T)$, we are able to show that $\bph_\Omega(t)$
is contained in some compact $K_\sigma\subset\simap$.
This fact allows us to recover the classical $L^2(\sigma,T;\VV)$-regularity for every $\sigma>0$
by employing Property~\ref{MZ0}.
Actually, we can prove something much sharper.
By analyzing the behavior of $\bph_\Omega$ close to the initial time
by means of suitable qualitative estimates, and observing that the
Flory--Huggins potential \eqref{Flog} satisfies Property~\ref{MZ}, we can explicitly obtain integrability estimates
also when $\bph_\Omega$ squeezes to $\0$ in $\simap$, i.e., for small times.
It is worthwhile pointing out that this idea is almost sharp,
in the sense that, although one cannot recover the typical $L^2(0,T;\VV)$-regularity for
$\bmu$, the $L^p(0,T;\VV)$-regularity for any $p\in(1,2)$ can be achieved.
This seems a very reasonable price to pay, considering the apparently unusual, but reasonable for certain models,
initial condition $\bph(0)=\0$.
\end{remark}

\begin{remark}
As pointed out above, the crucial point concerns the initial condition $\bph(0)=\0$.
This forces us to distinguish the regularity framework at zero and after zero.
However, it is worth observing that  if the initial value $\bph_0$ is not $\0$, but of the form, for instance, $(\eps,\eps)$
for some small constant $\eps>0$, then all the regularities in Theorem~\ref{THM:EX:WEAK} hold also with the choice $\sigma=0$.
This can be easily deduced from the proof of Theorem~\ref{THM:EX:WEAK}.
\end{remark}

Further regularity properties for $P$ and $\br$ are given by

\begin{theorem}
\label{THM:REGULARITY}
Suppose that the assumption of Theorem \ref{THM:EX:WEAK} hold.
In addition, let
\begin{align}
	\label{P:regpar}
	P_0 \in V.
\end{align}
Then, every weak solution $(\bph,\bmu,P,\br)$ given by Theorem \ref{THM:EX:WEAK} is such that
\begin{align}
	\label{reg:P}
	P & \in	\H1 {H} \cap \L\infty {V} \cap \L2 {W},
	\\
	\label{reg:R}
	\br & \in \H1 {\HH} \cap \L\infty {\VV} \cap \L2 \WW.
\end{align}
\end{theorem}

We can also prove some weighted-in-time regularity of $\bph$, namely,

\begin{theorem}
\label{THM:REGULARITY:2d}
Suppose that \ref{ass:pot}--\ref{ass:sources}, \eqref{ass:exweak:initialdata}--\eqref{max:ini:Linf}, and
\eqref{P:regpar} hold, and assume that in Property~\ref{MZ} it also holds that
\begin{equation}
\label{ip_F}
 s \mapsto s^{\frac32-\alpha}F(s) \in L^\infty(0,R)\quad\forall\,\alpha\in(0,1).
\end{equation}
Then, there exists a global weak solution $(\bph,\bmu,P,\br)$,
in the sense of Theorems~\ref{THM:EX:WEAK}--\ref{THM:REGULARITY},
which also satisfies the following:
for every $\alpha\in(0,1)$, there exists a positive constant $C$,
depending on $\alpha$ and $(\bph,\bmu,P,\br)$, such that
\begin{align}
	&\norma{t^{\frac12}\partial_t\bph}_{L^{\infty}(0, T; \VVp)\cap L^2(0, T; \VV)}
	\label{reg:extra'}
	+\norma{t^{\frac32-\alpha}\bmu}_{L^\infty(0,T;\VV)}
	+\norma{t^{\frac12}\nabla\bmu}_{L^\infty(0,T;\HH)}
	\leq C.
\end{align}
\end{theorem}

\subsection{On the unique continuation property of strong solutions}

In this section, we aim to present a continuous dependence estimate for regular enough solutions.
There are two major obstructions: the singular potential which can be exploited only through monotonicity and the initial condition which is a pure phase. This combination prevents the use of standard estimates.
Thus continuous dependence can be possibly achieved if, for instance, the potential is globally \Lip\ continuous.
This is the case if the solutions are regular enough and the components of $\bph$ are (strictly) separated from pure phases, i.e., they take values in a closed subset of $(0,1)$.
Since $\bph(0)=\0 \in \partial \Delta_\circ$, the best we can expect is a ``delayed" separation principle. Recall that $\bph(\s) \in \Delta_\circ$ for every $\sigma\in(0,T)$.
Thus, we should prove the following version of the separation principle: for every $\sigma\in(0,T)$,
there exists a compact subset $K=K(\sigma)$ of $\simap$ such that
\begin{align*}
	\bph(x,t) \in K \quad \text{ for a.e.~} x\in \Omega,\quad\forall\, t \in [\s,T].
\end{align*}
Now, let us recall that the model is characterized by three phases, namely, $\bph=(\varphi_1,\varphi_2)$ and $S = 1- \sum_{i=1}^2 \varphi_i$. Also,  $\Psi^{(1)}(\bph)$ is defined in \eqref{Flog}.
The separation property has been established so far in dimension two and for scalar
Cahn--Hilliard equations only (see \cite{GGM,MZ}). In order to adapt that argument to the present case, the term $S$ is critical since it
introduces mixed terms in the gradient and Hessian of $\Psi$ (as easy computations show) so that, for instance, \cite[Cor.~4.1]{GGM} cannot be extended by arguing componentwise.
Despite the problem concerning the initial condition can be overcome allowing some time delay, the validity of the separation principle for the present model seems to be an open problem.

Motivated by the above comments, following a standard approach, we replace $S \ln S$ in the definition \eqref{Flog} with a suitable polynomial approximation.
For such a regularized version, we can actually argue componentwise and reproduce the arguments employed in the two-components case.
Once we have the separation principle, it is not hard to prove a unique continuation property as a consequence of a delayed continuous dependence estimate.


More precisely, we restrict our attention to the Flory-Huggins potential \eqref{Flog} only,
and replace the mixing entropy $\Psi^{(1)}(\bph) $ by
\begin{align*}
	\widehat \Psi^{(1)}(\bph) & :=
	\begin{cases}
	\sum_{i=1}^2\varphi_i \ln \varphi_i + (1- \varphi_1-  \varphi_2) (- \varphi_1-  \varphi_2)
	&\text{if } \ph_i \in [0,1], \, i=1,2,
	\\
	+\infty &\text{otherwise.}
	\end{cases}
\end{align*}
Observe that, formally speaking, $\widehat \Psi^{(1)}\sim  \Psi^{(1)} $ as $S=1- \varphi_1-  \varphi_2 \sim 0$.
Now, the first order approximation $(1- \varphi_1-  \varphi_2) (- \varphi_1-  \varphi_2) $ can be absorbed into the quadratic perturbation $ \Psi^{(2)} $.
Upon setting $\widetilde  \Psi = \widetilde \Psi^{(1)}+\widetilde \Psi^{(2)} $ with
\begin{align}
	\label{psitilde:1}
	\widetilde \Psi^{(1)}(\bph) & :=
	\begin{cases}
	\sum_{i=1}^2\varphi_i \ln \varphi_i 	\quad&\text{if } \ph_i \in [0,1], \, i=1,2,
	\\
	+\infty \quad&\text{otherwise,}
	\end{cases}
	\\
	\widetilde \Psi^{(2)}(\bph) & :=
	\Psi^{(2)}(\bph)
	+ (1- \varphi_1-  \varphi_2) (- \varphi_1-  \varphi_2),
	\label{psitilde:2}
\end{align}
we consider system \eqref{SYS:1}--\eqref{SYS:4} subject to \eqref{SYS:5}--\eqref{SYS:8} where $\Psi$ is replaced by $\widetilde \Psi$. A corresponding solution to this problem
will be denoted by $(\widetilde \bph, \widetilde\bmu, \widetilde P, \widetilde\br)$.
All the results presented so far can be straightforwardly adapted with the only difference that the order parameter $\widetilde \bph$ now takes values in
\begin{align*} 
	 {\cal Q}_\circ := \{ \rr=(r_1,r_2) \in  (0,1) \times  (0,1)  \}\subset \erre^2,  \quad {\cal Q}_{\bullet}  := \overline{\cal Q_\circ},
\end{align*}
instead of $\Delta_\circ$ and $\Delta_\bullet$. This means that we can no longer guarantee that $S$ takes its values in $[0,1]$.
Observe that $\widetilde \Psi^{(1)}$ and $\widetilde \Psi^{(2)}$ satisfy \ref{ass:pot} with ${\cal Q}_{\circ}$ and ${\cal Q}_{\bullet}$ instead of $\Delta_\circ$ and $\Delta_\bullet$.
Moreover,  system \eqref{SYS:1}--\eqref{SYS:2} is now decoupled since $D_i \widetilde \Psi^{(1)}(\bph) =\ln \ph_i +1$, $i=1,2$, and no mixed terms appear.
In particular, the (easiest) scalar version of Property \ref{MZ} must be used.

Summing up, the following result can be thought as a corollary of our main existence theorem.

\begin{corollary}\label{approxmain}
Let the assumptions of Theorems~\ref{THM:EX:WEAK}, \ref{THM:REGULARITY}, and \ref{THM:REGULARITY:2d} be fulfilled
with ${\cal Q}_{\circ}$ and ${\cal Q}_{\bullet}$ instead of $\Delta_\circ$ and $\Delta_\bullet$.
Then, there exists a weak solution $(\widetilde \bph, \widetilde\bmu, \widetilde P, \widetilde\br)$, in the sense of Theorem \ref{THM:EX:WEAK}
to system \eqref{SYS:1}--\eqref{SYS:4} with \eqref{SYS:5}--\eqref{SYS:8} where $\Psi$ is replaced by $\widetilde \Psi$.
Moreover, the regularities in \eqref{reg:P}, \eqref{reg:R} as well as \eqref{reg:extra'} are fulfilled.

\end{corollary}

We recall that \ref{ass:pot} (with ${\cal Q}_\circ$ and ${\cal Q}_\bullet $ instead of $\Delta_\circ$ and $\Delta_\bullet$) and \eqref{ip_F} are fulfilled by $\widetilde \Psi^{(1)}$ (see \eqref{psitilde:1}).
Moreover, note that
the potential $\widetilde \Psi^{(1)}$ satisfies the growth condition
\begin{equation}\label{growth}
\exists\,C>0:\quad\opnorm{{{\widetilde \Psi}_{\bph\bph}^{(1)}}  ( \rr )}
\leq e^{C |{\widetilde \Psi^{(1)}_{\bph}}(\rr)| +C} \qquad\forall\,\rr \in {\cal Q}_\circ,
\end{equation}
and is singular in the sense that
\begin{align}\label{growth2}
	\lim_{\substack{\rr\to\rr_0 \\ \rr\in {\cal Q}_\circ}}|\widetilde \Psi^{(1)}_\bph(\rr)|=+\infty \quad\forall\,\rr_0\in {\cal Q}_\circ \setminus{\cal Q}_\bullet.
\end{align}
These properties are essential in proving the separation principle.

\begin{theorem}
\label{THM:2d:separation}
Let $d=2$ and the assumptions of Theorems~\ref{THM:EX:WEAK}, \ref{THM:REGULARITY}, and \ref{THM:REGULARITY:2d} be fulfilled with ${\cal Q}_{\circ}$ and
${\cal Q}_{\bullet}$ instead of $\Delta_\circ$ and $\Delta_\bullet$.
Then, it holds that
\begin{align*}
		\norma{t^{\frac32-\alpha}\bph}_{L^\infty (0,T ; \WW^{2,{r}}(\Omega) )}
		+
		\norma{t^{\frac32-\alpha}\Psi_\bph(\bph)}_{L^\infty(0,T;\LL^r(\Omega))}
	 \leq C,
\end{align*}
and
\begin{align}
	\nonumber
&\norma{t^{\frac 12} e^{-C_\alpha t^{\alpha-\frac32}}\partial_t \widetilde \bph}_{L^{\infty}(0, T; \HH)\cap L^2(0, T; \WW)}
	+ \norma{t^{\frac 12}e^{-C_\alpha t^{\alpha-\frac32}}\widetilde \bmu}_{L^\infty (0,T; \WW )}\\
\label{reg:extra''}
	&\qquad +\norma{e^{-C_\alpha t^{\alpha-\frac32}}\widetilde \Psi_{\bph\bph}(\widetilde \bph)}_{L^\infty(0,T;\LL^r(\Omega))}\leq C
\end{align}
for any $r\in (1,\infty)$.
Eventually, for every $\sigma\in(0,T)$,
there exists a compact subset $K=K(\sigma)$ of ${\cal Q}_\circ$ such that
\begin{align}\label{sep}
	\widetilde \bph(x,t) \in K \quad \text{ for a.e.~} x\in \Omega,\quad\forall\, t \in [\s,T].
\end{align}
In particular, this entails that
\begin{align}\label{reg:extra}
	\norma{\widetilde \bph}_{L^\infty(\sigma, T; \HH^4(\Omega))} \leq C_\sigma \quad\forall\,\sigma\in(0,T).
\end{align}
\end{theorem}

The final result of this section is the claimed unique continuation property.

\begin{theorem}\label{THM:UNIQ:2d}
Let $d=2$ and suppose that the assumptions of Theorem~\ref{THM:2d:separation} hold.
Let $(\widetilde \bph^i, \widetilde \bmu^i, \widetilde P^i, \widetilde \br^i)_i$, $i=1,2$, be two solutions
given by Corollary~\ref{approxmain} associated with the initial data
$P_0^i $ and $ \br_0^i$, $i=1,2$, satisfying the assumptions \eqref{ass:exweak:initialdata}--\eqref{max:ini:Linf} and \eqref{P:regpar}.
Then, for every $\sigma\in(0,T)$,
there exists a constant $C_\sigma>0$ depending on $(\widetilde \bph^i, \widetilde \bmu^i, \widetilde P^i, \widetilde \br^i)_i$, $i=1,2$, such that
\begin{align}
	\non
	& \norma{\widetilde \bph^1-\widetilde  \bph^2}_{L^\infty(\s,T; \HH) \cap \Ls2 {\HH^2(\Omega)}}\\
	&\quad + \big\|{\widetilde P^1- \widetilde P^2}\big\|_{\Ls\infty H \cap \Ls2 V}
	+ \big\|{\widetilde \br^1-\widetilde \br^2}\big\|_{\Ls\infty \HH \cap \Ls2 \VV}
	\non
	\\
	& \quad  \label{est:cd:2d}
	\leq C_\sigma \big({\norma{\widetilde \bph^1(\sigma)-\widetilde \bph^2(\sigma)}}+ \big\|{\widetilde P^1(\sigma)-\widetilde P^1(\sigma)}\big\|+\big\|{\widetilde \br^1(\sigma)-\widetilde \br^2(\sigma)}\big\| \big).
\end{align}
\end{theorem}

\begin{remark}
Here we have approximated $S\ln S$ with a first-order truncation. However, any polynomial approximation should work, possibly changing the approximation of the singular part.
\end{remark}


\color{black}

\section{Proof of Theorem~\ref{THM:EX:WEAK}}
\label{sec:ex_weak}
The proof is divided into several steps. First{,} all the nonlinearities are suitably approximated and the corresponding
problem is shown to have a unique solution $(\bph_\lambda, P_\lambda,\br_\lambda)$ via a fixed-point argument, $\lambda>0$ being the approximation parameter that eventually will go to zero.
Then we show that the approximating variables $P_\lambda$, $R_{\lambda,1}$, and $R_{\lambda,2}$ take value in an interval $(c_*,1]$, $c_*>0$, which is independent of $\lambda$. The strict separation from $0$ requires a Moser-type argument.
After that, further a priori bounds on the approximating solutions are obtained.
Concerning the Cahn--Hilliard system, we point our that some estimates are nonstandard. In particular, we need to control the behavior of $\bph$ as $t\to 0$ because of the null initial condition. This will be helpful for the last crucial bounds.
Moreover, the first basic estimate for $\bph_\lambda$ and $\nabla \bmu_\lambda$ will be handled by using the convex conjugate of $\Psi^{(1)}_\lambda$ because of the presence of the source and the null initial condition, being $\Psi^{(1)}_{\lambda}$ a suitable regularization of $\Psi^{(1)}$. Before passing to the limit as $\lambda \to 0$, we are left to find convenient integrable bounds of $\bph_\lambda$, $\bmu_\lambda$, and $\Psi^{(1)}_{\lambda,\bph}(\bph_\lambda)$ in a small right neighborhood of the initial time, where $\Psi^{(1)}_{\lambda,\bph}$ denotes the gradient of the aforementioned regularization $\Psi^{(1)}_\lambda$.
This is the main novelty of the proof and of the paper itself. Taking advantage of the previous bounds, we give careful estimates of the integral means of the $\bph$-components. These estimates are then employed to uniformly control $L^p$-in-time-norms in some time interval $(0,T_0) \subset (0,T)$ for $p\in (1,2)$. Thus we can eventually let $\lambda$ go to zero and get a weak solution to our original problem.

\subsection{The approximating problem}\label{ssec:approx}
We first introduce suitable approximations of the nonlinear terms.
For any $\lambda\in(0,1)$, let $\Psi_\lambda^{(1)}:\erre^2\to\erre$
denote the Moreau--Yosida regularization of the convex function $\Psi^{(1)}$,
and set $\Psi_\lambda:=\Psi^{(1)}_\lambda + \Psi^{(2)}$.
In particular, we have
\[
  \Psi_\lambda\in C^2(\erre^2), \qquad
  \Psi_{\lambda,\bph} := \nabla \Psi_\lambda \in C^{0,1}(\erre^2;\erre^2).
\]
Let also
\[
  \bj_\lambda=(J_\lambda^1, J_\lambda^2):\erre^2\to\simap, \qquad
  \bj_\lambda:=(I_{\erre^2} + \lambda\Psi^{(1)}_\bph)^{-1},
\]
denote the resolvent of the maximal monotone operator $\Psi^{(1)}_\bph$.
This means that, for every $\rr\in\erre^2$, $\bj_\lambda(\rr)$ is the unique
element in $\simap$ such that
$\bj_\lambda(\rr) + \lambda\Psi^{(1)}_\bph(\bj_\lambda(\rr))=\rr$.
We recall that the Yosida approximation $\Psi_{\lambda,\bph}^{(1)}$ satisfies
$\Psi_{\lambda,\bph}^{(1)}(\rr) = \Psi_{\bph}^{(1)}(\bj_\lambda(\rr))$ for every
$\rr\in\erre^2$.

In principle, the expected maximum and minimum inequalities for $P$ and $\br$ given in
Theorem~\ref{THM:EX:WEAK}
may not be preserved in the approximated problem.
To force this, we need to introduce a suitable truncation function:
this will allow us to prove the minimum ad maximum principles
at every step of the approximation (see \eqref{sep_PR}).
Namely, we set
\begin{alignat}{2}
\label{h_def1}
&h:\erre\to\erre, \qquad
&&h(r):=\max\{0, \min\{r,1\}\}, \quad r\in\erre,
\end{alignat}
and define the corresponding truncated source terms
\begin{align*}
	\Sph^\lambda,\SR^\lambda:\erre^2 \times \erre \times \erre^2 \to \erre^2,
	\qquad
	\SP^\lambda:\erre^2 \times \erre \times \erre^2 \to \erre,
\end{align*}
by setting, for every $(\bphi,p,\rr)\in\erre^{2}\times\erre\times\erre^2$,
\begin{align}	
	\label{def:sorgenti_app}
	\Sph^\lambda(\bphi, p, \rr)& =
	- \SR^\lambda(\bphi, p, \rr)=
	(c_1 h(p) h(r_1) - c_2 J_\lambda^1(\bphi),
	c_3 h(p) h(r_2) - c_4 J_\lambda^2(\bphi)),
	\\
	\label{def:Sp_app}
	\SP^\lambda(\bphi, p, \rr) & =
	-c_1 h(p) h(r_1) + c_2 J_\lambda^1(\bphi)
	- c_3 h(p) h(r_2) + c_4 J_\lambda^2(\bphi).
\end{align}
Then the $\lambda$-approximation of our original problem reads
\begin{alignat}{2}
	\label{SYS:1_app}
	\dt \bph_\lambda - \Delta\bmu_\lambda
	&= \Sph^\lambda(\bph_\lambda, P_\lambda, \br_\lambda)
		&&\qquad \text{in $Q$},
		\\	\label{SYS:2_app}
	\bmu_\lambda &= -\Delta \bph_\lambda + \Psi_{\lambda,\bph}(\bph_\lambda)
		&&\qquad \text{in $Q$},
		\\
	\label{SYS:3_app}
	\dt P_\lambda - \Delta P_\lambda
	&=  \SP^\lambda(\bph_\lambda, P_\lambda, \br_\lambda)
		&&\qquad \text{in $Q$},\\
	\label{SYS:4_app}
	\dt \br_\lambda - \Delta \br_\lambda
	&=  \SR^\lambda(\bph_\lambda, P_\lambda, \br_\lambda)
		&&\qquad \text{in $Q$},\\
	\label{SYS:5_app}
	\dn\bph_\lambda &= \dn\bmu_\lambda  =\dn\br_\lambda
		= \0, \quad \dn P_\lambda= 0
		&&\qquad \text{on $\Sigma$},\\
	\label{SYS:7_app}
	\bph_\lambda(0)
	&= \bph_0 =\0, \quad
	\br_\lambda(0)
	= \br_0 , \quad
	P_\lambda(0)
	= P_0
	&&\qquad \text{in $\Omega$}.
	\end{alignat}
In order to show the existence of a weak solution to \eqref{SYS:1_app}--\eqref{SYS:7_app} we employ a fixed point argument.
In this subsection we omit the dependence on $\lambda$ for the sake of simplicity.

Let then $\ov\bph \in L^2(0,T;\HH)$ be fixed,
and consider the following auxiliary system in the
variables $(P_\lambda^{\ov\bph}, \br_\lambda^{\ov\bph})$:
\begin{alignat*}{2}
	\dt P_\lambda^{\ov\bph} - \Delta P_\lambda^{\ov\bph}
	&=  \SP^\lambda(\ov\bph, P_\lambda^{\ov\bph}, \br_\lambda^{\ov\bph})
		&&\qquad \text{in $Q$},\\
	\dt \br_\lambda^{\ov\bph} - \Delta \br_\lambda^{\ov\bph}
	&=  \SR^\lambda(\ov\bph, P_\lambda^{\ov\bph}, \br_\lambda^{\ov\bph})
		&&\qquad \text{in $Q$},\\
	\dn\br_\lambda^{\ov\bph}
		&= \0, \quad
	\dn P_\lambda^{\ov\bph}= 0
		&&\qquad \text{on $\Sigma$},\\
	\br_\lambda^{\ov\bph}(0)
	&= \br_0 , \quad
	P_\lambda^{\ov\bph}(0)
	= P_0
	&&\qquad \text{in $\Omega$},
	\end{alignat*}
where the initial data are chosen as in \eqref{ass:exweak:initialdata}--\eqref{max:ini:Linf}.
By the boundedness and Lipschitz continuity of the cut-off function $h$,
by the fact that $\bj_\lambda(\ov\bph)\in\simap$, and
by \eqref{def:sorgenti_app}--\eqref{def:Sp_app}, one can easily check that the
time-dependent operators
\begin{align*}
&\SP^\lambda(\ov\bph(\cdot),\cdot,\cdot):[0,T]\times H\times \HH\to H, \quad
  (t,p,\rr)\mapsto\SP^\lambda(\ov\bph(t), p,\rr),\\
&\SR^\lambda(\ov\bph(\cdot),\cdot,\cdot):[0,T]\times H\times \HH\to \HH, \quad
  (t,p,\rr)\mapsto\SR^\lambda(\ov\bph(t), p,\rr), \
\end{align*}
for $(t,p,\rr) \in[0,T]\times H\times \HH$,
are Lipschitz continuous and bounded in the last two variables,
uniformly in $t\in[0,T]$.
Consequently, well-known results yield the existence and uniqueness of a weak solution
\[
  P_\lambda^{\ov\bph} \in H^1(0,T; V^*)\cap L^2(0,T; V), \qquad
  \br_\lambda^{\ov\bph} \in H^1(0,T; \VV^*)\cap L^2(0,T;\VV).
\]
At this point, we can plug
$(P_\lambda^{\ov\bph},\br_\lambda^{^{\ov\bph}})$ into the
first equation, and consider the following Cahn--Hilliard equation with source $\Sph^\lambda$ in the variables $(\bph_\lambda^{\ov \bph},\bmu_\lambda^{\ov \bph})$
\begin{alignat*}{2}
	\dt \bph_\lambda^{\ov\bph} - \Delta\bmu_\lambda^{\ov\bph}
	&= \Sph^\lambda(\ov\bph, P_\lambda^{\ov\bph}, \br_\lambda^{\ov\bph})
		&&\qquad \text{in $Q$},
		\\
	\bmu_\lambda^{\ov\bph} &=
	-\Delta \bph_\lambda^{\ov\bph} + \Psi_{\lambda,\bph}(\bph_\lambda^{\ov\bph})
		&&\qquad \text{in $Q$},
		\\
	\dn\bph_\lambda^{\ov\bph} &= \dn\bmu_\lambda^{\ov\bph} =\0
		&&\qquad \text{on $\Sigma$},\\
	\bph_\lambda^{\ov \bph}(0)
	&= \bph_0
	&&\qquad \text{in $\Omega$}.
	\end{alignat*}
Note that here the source term
$\Sph^\lambda(\ov\bph,P_\lambda^{\ov\bph}, \br_\lambda^{\ov\bph})$ is a fixed forcing term which is globally bounded, that is,
\[
\Sph^\lambda(\ov\bph,P_\lambda^{\ov\bph}, \br_\lambda^{\ov\bph}) \in \LL^\infty(Q).
\]
Consequently, since $\Psi_{\lambda,\bph}$ is now Lipschitz continuous,
a standard argument entails that the above problem admits a unique weak solution
\[
  \bph_\lambda^{\ov\bph} \in H^1(0,T; \VV^*)\cap L^\infty(0,T; \VV)\cap L^2(0,T; \HH^3(\Omega)),
  \quad
  \bmu^{\ov \bph}_\lambda \in \L2 \VV.
\]
Thus we can define the map
\[
  \Gamma_\lambda: L^2(0,T;\HH)\to
  H^1(0,T; \VV^*)\cap L^\infty(0,T; \VV)\cap L^2(0,T; \HH^3(\Omega)),
  \quad \Gamma_\lambda:\ov\bph\mapsto\bph_\lambda^{\ov\bph},
\]
and the existence of a weak solution to  \eqref{SYS:1_app}--\eqref{SYS:7_app}
follows from the existence of a fixed point for $\Gamma_\lambda$.
To this end, we use the Banach fixed-point theorem.
Indeed, it is immediate from the
classical parabolic theory to deduce the existence of a
positive constant $C_\lambda$ such that
\begin{align*}
  &\norma{P_\lambda^{\ov\bph_1}-
  P_\lambda^{\ov\bph_2}}_{L^\infty(0,T; H)\cap L^2(0,T; V)}
  +\norma{\br_\lambda^{\ov\bph_1}-
  \br_\lambda^{\ov\bph_2}}_{L^\infty(0,T; \HH)\cap L^2(0,T; \VV)}\\
  &\qquad \leq C_\lambda\norma{\ov\bph_1-\ov\bph_2}_{L^2(0,T;\HH)},
\end{align*}
for every $\ov\bph_1, \ov\bph_2\in L^2(0,T;\HH)$. Moreover,
well-known results on the Cahn--Hilliard equation with regular potential,
combined with the Lipschitz continuity of $\Psi_{\lambda,\bph}$, and the
Lipschitz continuity and boundedness of $h$ and $\bj_\lambda$ yield
(possibly updating the value of $C_\lambda$ at each step)
\begin{align*}
&\norma{\bph_\lambda^{\ov\bph_1}-
\bph_\lambda^{\ov\ph_2}}_{L^\infty(0,T;\HH)\cap L^2(0,T; \HH^2(\Omega))}\\
&\qquad\leq C_\lambda
\left(\norma{\Sph^\lambda(\ov\bph_1,P_\lambda^{\ov\bph_1}, \br_\lambda^{\ov\bph_1})-
\Sph^\lambda(\ov\bph_2,P_\lambda^{\ov\bph_2}, \br_\lambda^{\ov\bph_2})}_{L^2(0,T; \HH)}\right)\\
&\qquad\leq C_\lambda
\left(\norma{\ov\bph_1-\ov\bph_2}_{L^2(0,T;\HH)}
+\norma{P_\lambda^{\ov\bph_1}-P_\lambda^{\ov\bph_2}}_{L^2(0,T;H)}
+\norma{\br_\lambda^{\ov\bph_1}-\br_\lambda^{\ov\bph_2}}_{L^2(0,T;\HH)}\right).
\end{align*}
Collecting the estimates above, we then find $C_\lambda>0$ such that
\[
  \norma{\Gamma_\lambda(\ov\bph_1)-
  \Gamma_\lambda(\ov\bph_2)}_{L^\infty(0,T;\HH)\cap L^2(0,T; \HH^2(\Omega))}
  \leq C_\lambda\norma{\ov\bph_1-\ov\bph_2}_{L^2(0,T;\HH)}
\]
for all $\ov\bph_1,\ov\bph_2\in L^2(0,T;\HH)$.
In particular, using the H\"older inequality, we find that
\[
  \norma{\Gamma_\lambda(\ov\bph_1)-
  \Gamma_\lambda(\ov\bph_2)}_{L^2(0,T_0;\HH)}
  \leq C_\lambda T_0^{1/2}\norma{\ov\bph_1-\ov\bph_2}_{L^2(0,T_0;\HH)}
  \quad\forall\,\ov\bph_1,\ov\bph_2\in L^2(0,T_0;\HH)
\]
for every arbitrary $T_0\in(0,T]$. Consequently,
provided to choose $T_0$ sufficiently small, e.g., $T_0<C_\lambda^{-2}$,
an application of the Banach fixed point theorem gives the
existence and uniqueness of a weak solution $(\bph_\lambda, \bmu_\lambda,
P_\lambda,\br_\lambda)$
to \eqref{SYS:1_app}--\eqref{SYS:7_app} on $[0,T_0]$.
Noting that this procedure is independent of the initial time
but only depends on the length $T_0$ of the time-interval, so that a standard
argument allows to extend the solution to a global weak solution $(\bph_\lambda, \bmu_\lambda, P_\lambda,\br_\lambda)$
to the whole time interval $[0,T]$. Moreover, such a solution satisfies the following properties
\begin{align*}
  \bph_\lambda &\in H^1(0,T; \VV^*)\cap L^\infty(0,T; \VV)\cap L^2(0,T; \HH^3(\Omega)),\\
  \bmu_\lambda &\in L^2(0,T; \VV),\\
  P_\lambda &\in H^1(0,T; V^*)\cap L^2(0,T; V), \\
  \br_\lambda &\in H^1(0,T; \VV^*)\cap L^2(0,T;\VV).
\end{align*}

\subsection{Boundedness of the approximating concentrations}
Here we prove the following boundedness result for the approximating concentrations $P_\lambda$ and $\br_\lambda$.
\begin{lemma}\label{lem1}
In the current setting, there exists a threshold
\begin{equation}
\label{smallness}
c_*\in\left(0,\frac{c_2+c_4}{8\min\{c_1,c_3\}}\right],
\end{equation}
independent of $\lambda$, such that, for every $t\in[0,T]$,
\begin{align}
  \label{sep_PR}
  c_*\leq P_\lambda(t)\leq 1, \quad
  c_*\leq R_{\lambda,i}(t)\leq 1\quad\text{a.e.~in $\Omega$}, \quad i=1,2.
\end{align}
\end{lemma}
This result will be very helpful in the sequel. Furthermore, the above bounds agree with the physical interpretation of $P$ and the components
of $\br$ as concentrations. The strict separation from zero is a (non obvious) consequence of the initial conditions
(see Remark~\ref{iniP}).

\begin{proof}[Proof of Lemma~\ref{lem1}]
Also this proof is divided into several steps and it is based on a qualitative analysis of the parabolic system
\begin{alignat}{2}
	\label{PAR:SYS:1}
	\dt P_\lambda -\Delta P_\lambda+ (c_1 h(R_{\lambda,1}) +
	c_3 h(R_{\lambda,2}))h(P_\lambda)
	&= c_2 J_\lambda^1(\bph_{\lambda})  + c_4 J_\lambda^2(\bph_{\lambda})
		&&\qquad \text{in $Q$},\\
	\label{PAR:SYS:2}
	\dt R_{\lambda,1} - \Delta R_{\lambda,1} + c_1 h(P_\lambda) h(R_{\lambda,1})
	&= c_2 J_\lambda^1(\bph_{\lambda})
		&&\qquad \text{in $Q$},\\
	\label{PAR:SYS:3}
	\dt R_{\lambda,2} - \Delta R_{\lambda,2} + c_3 h(P_\lambda) h(R_{\lambda,2})
	&= c_4 J_\lambda^2(\bph_{\lambda})
		&&\qquad \text{in $Q$},
	\end{alignat}
endowed with homogeneous Neumann boundary conditions
$\dn P_\lambda = \dn R_{\lambda,1}= \dn R_{\lambda,2}=0$
and initial data $P_\lambda(0)= P_0$ and $\br_\lambda(0)= \big(\frac {1-P_0}2\big) \1$.

\noindent
{\sc Step 1.} First of all, we show that $P_\lambda$ and $R_{\lambda,i}$, $i=1,2$, are nonnegative. To this end,
we multiply \eqref{PAR:SYS:1} by $-(P_\lambda)_{-} = P_\lambda \chi_{\{ P_\lambda<0\}}$,
where the negative part function $(\cdot)_{-}$ is defined by $(\cdot)_{-}:=\max\{-\cdot,0\}$,
and integrate over $Q_t:=\Omega\times[0,t],$ for an arbitrary $t \in (0,T)$. This gives
\begin{align*}
	&\frac 12\IO2 {(P_\lambda) _{-}}
	+ \int_{Q_t} |\nabla (P_\lambda)_{-}|^2
	- \int_{Q_t}[c_1 h(R_{\lambda,1}) +c_3 h(R_{\lambda,2})]h(P_\lambda) (P_\lambda)_{-}\\
	&\quad =  -\int_{Q_t}[c_2 J_\lambda^1(\bph_{\lambda})
	+ c_4J_\lambda^2(\bph_{\lambda})](P_\lambda)_{-},
\end{align*}
where the initial term $\tfrac 12\norma {(P_0)_{-}}^2 $ vanishes due to \eqref{P:minmax:ini}.
The third term on the \lhs\ is also zero
since $h(P_\lambda)=0$ on the set $\{P_\lambda<0\}$.
Moreover, the term on the \rhs\ is nonpositive
since $J_\lambda^i\geq0$ for $i=1,2$ by definition. Namely, it holds that
\begin{align*}
	 -\int_{Q_t} [c_2 J_\lambda^1(\bph_{\lambda} )
	 + c_4 J_\lambda^2(\bph_{\lambda})](P_\lambda)_{-}
	= \int_{Q_t \cap \{ P_\lambda < 0\}}
	[c_2 J_\lambda^1(\bph_{\lambda} )
	+ c_4 J_\lambda^2(\bph_{\lambda})]P_\lambda \leq 0.
\end{align*}
Therefore, it follows that
\[
  \frac 12\IO2 {(P_\lambda) _{-}}
	+ \int_{Q_t} |\nabla (P_\lambda)_{-}|^2 \leq0 \qquad\forall\,t\in[0,T],
\]
yielding that, by the arbitrariness of $t$, $P_\lambda \geq 0$ almost everywhere in $Q$.
Arguing along the same lines on the second and third equations
\eqref{PAR:SYS:2}--\eqref{PAR:SYS:3}, namely testing
\eqref{PAR:SYS:2} by $-(R_{\lambda,1})_-$ and
\eqref{PAR:SYS:3} by $-(R_{\lambda,2})_-$,
we obtain the nonnegativity of the components of $\br_\lambda$.
Hence, we have shown that
\begin{align}
	\label{min_princ}
	P_\lambda(x,t) \geq 0, \quad
	R_{\lambda,i}(x,t) \geq 0
	\quad
	\text{for a.e.~$(x,t)\in Q$},
	\quad i=1,2.
\end{align}

\noindent
{\sc Step 2.}
In order to show the separation from zero, we employ a Moser-type argument.
To this end, for every $p\geq2$ and $\eps>0$ we define the monotone functions
\[
  \gamma_{p,\eps}:\erre\to (-\infty,0), \qquad
  \gamma_{p,\eps}(r):=
  \begin{cases}
    -\frac{p}{r^{p+1}} \quad&\text{if } r\geq\eps,\\
    -\frac{p}{\eps^{p+1}} \quad&\text{if } r<\eps,
  \end{cases}
\]
and the convex functions
\[
  \widehat\gamma_{p,\eps}:\erre\to (0,+\infty), \qquad
  \widehat\gamma_{p,\eps}(r):=
  \begin{cases}
    \frac1{r^{p}} \quad&\text{if } r\geq\eps,\\
    -\frac{pr}{\eps^{p+1}} + \frac{1+p}{\eps^{p}} \quad&\text{if } r<\eps.
  \end{cases}
\]
Observe that, in particular, $\widehat\gamma_{p,\eps}'=\gamma_{p,\eps}$
and $\gamma_{p,\eps}$ is monotone Lipschitz continuous.
Testing \eqref{PAR:SYS:1} by $\gamma_{p,\eps}(P_\lambda)$ and integrating over $Q_t$ lead us to
\begin{align*}
  &\int_\Omega \widehat\gamma_{p,\eps}(P_\lambda(t))
  +\int_{Q_t}\gamma_{p,\eps}'(P_\lambda)|\nabla P_\lambda|^2
  =\int_\Omega \widehat\gamma_{p,\eps}(P_0)\\
  &\qquad
  +\int_{Q_t}(c_2J_\lambda^1(\bph_{\lambda})
  + c_4 J_\lambda^2(\bph_{\lambda}))\gamma_{p,\eps}(P_\lambda)
  -\int_{Q_t}(c_1 h(R_{\lambda,1}) + c_3 h(R_{\lambda,2}))
  h(P_\lambda)\gamma_{p,\eps}(P_\lambda).
\end{align*}
By the monotone convergence theorem and \eqref{min_princ} we deduce that
\[
  \lim_{\eps\searrow0}\int_\Omega \widehat\gamma_{p,\eps}(P_\lambda(t))
  =\int_{\Omega}\frac1{|P_\lambda(t)|^p}
  = \norma{\frac1{P_\lambda(t)}}_{L^p(\Omega)}^p
  .
\]
Besides,
the second term on the left-hand side is nonnegative by
the monotonicity of $\gamma_{p,\eps}$.
As for the first term on the right-hand side,
noting that assumption \eqref{max:ini:Linf} yields
$1/P_0 \in L^q(\Omega)$ for all $q\in[2,+\infty]$,
we have
\[
\limsup_{\eps\searrow0}
\int_\Omega \widehat\gamma_{p,\eps}(P_0)
\leq \int_\Omega\frac{1}{P_0^p} =
\norma{\frac1{P_0}}_{L^p(\Omega)}^p.
\]
Furthermore, the nonnegativity of $J_\lambda^i$, i=1,2, entails that
\[
  \int_{Q_t}(c_2J_\lambda^1(\bph_{\lambda})
  + c_4 J_\lambda^2(\bph_{\lambda}))\gamma_{p,\eps}(P_\lambda)\leq0.
\]
Eventually, using the monotone convergence theorem once more
together with the boundedness of $h$, we find that
\begin{align*}
  & \limsup_{\eps\searrow0}-
 \int_{Q_t}(c_1 h(R_{\lambda,1}) + c_3 h(R_{\lambda,2}))
  h(P_\lambda)\gamma_{p,\eps}(P_\lambda)
  \\ & \quad
  \leq {p(c_1+c_3)}\int_{Q_t}\frac{1}{P_\lambda^p}
  = {p(c_1+c_3)}
  \int_0^t\norma{\frac{1}{P_\lambda(s)}}_{L^p(\Omega)}^p \,\ds.
\end{align*}
Putting everything together and letting $\eps\searrow0$, we get
\[
  \norma{\frac1{P_\lambda(t)}}_{L^p(\Omega)}^p
  \leq  \norma{\frac1{P_0}}_{L^p(\Omega)}^p + {p(c_1+c_3)}
  \int_0^t\norma{\frac{1}{P_\lambda(s)}}_{L^p(\Omega)}^p\ds
  \qquad\forall\,t\in[0,T].
\]
By the Gronwall lemma, we infer that
\[
  \norma{\frac1{P_\lambda(t)}}_{L^p(\Omega)}^p \leq
  e^{{p(c_1+c_3)}T}\norma{\frac1{P_0}}_{L^p(\Omega)}^p
  \qquad\forall\,t\in[0,T],
\]
so, by the H\"older inequality and the fact that $p\geq2$, also that
\[
  \norma{\frac1{P_\lambda(t)}}_{L^p(\Omega)}
  \leq e^{(c_1+c_3)T}{|\Omega|^{1/p}}\norma{\frac1{P_0}}_{L^\infty(\Omega)}
  {\leq e^{(c_1+c_3)T}{|\Omega|^{1/2}}\norma{\frac1{P_0}}_{L^\infty(\Omega)}}
  \qquad\forall\,t\in[0,T].
\]
Exploiting assumption \eqref{max:ini:Linf} one has that
$P_0\geq\frac{c_2+c_4}{2\min\{c_1,c_3\}}$ almost everywhere in $\Omega$. Hence,
letting $p\to\infty$, this ensures in particular that
\[
  \norma{\frac1{P_\lambda(t)}}_{L^\infty(\Omega)} \leq
  \max\{4, e^{(c_1+c_3)T}{|\Omega|^{1/2}}\}\cdot\left(\frac{c_2+c_4}{2\min\{c_1,c_3\}}\right)^{-1}
  \quad\forall\,t\in[0,T].
\]
Therefore $P_\lambda$ is indeed separated from $0$ uniformly in $\lambda$, that is,
\[
  P_\lambda(t) \geq c_* \quad\text{a.e.~in } \Omega \quad\forall\,t\in[0,T],
\]
where
\[
  c_*:=\frac1{\max\{4, e^{(c_1+c_3)T}{|\Omega|^{1/2}}\}}\frac{c_2+c_4}{2\min\{c_1,c_3\}} \leq
  \frac{c_2+c_4}{8\min\{c_1,c_3\}}.
\]
The argument for the separation of $R_{\lambda,i}$, $i=1,2$, is entirely analogous.
It suffices to note that $R_{0,i}\geq \frac{c_2+c_4}{4\min\{c_1,c_3\}}$ almost everywhere in $\Omega$, thanks again to \eqref{max:ini:Linf}.

\noindent
{\sc Step 3.}
Let us focus now on the upper bounds.
We begin with the first component of $\br_\lambda$.
Testing \eqref{PAR:SYS:2} by $(R_{\lambda,1}-1)_+ = R_{\lambda,1} \chi_{\{R_{\lambda,1} >1\}}$,
where $(\cdot)_+:=\max\{0,\cdot\}$ denotes the positive part, we obtain
\begin{align*}
	&\frac 12 \IO2 {(R_{\lambda,1} - 1)_+}
	 + \int_{Q_{t,\lambda}^1 } |\nabla (R_{\lambda,1}-1)|^2
	  +\int_{Q_{t,\lambda}^1 } c_1 h(P_\lambda) h(R_{\lambda,1})(R_{\lambda,1}-1)\\
	   &\quad= c_2 \int_{Q_{t,\lambda}^1 }
	   J_\lambda^1(\bph_{\lambda})(R_{\lambda,1}-1),
\end{align*}
where we have set $Q_{t,\lambda}^1 :={Q_t \cap \{R_{\lambda,1} >1\}}$
and we have used $\frac 12 \norma{({R_{0}^{1}} - 1)_+}^2 =0$ (see \eqref{ass:exweak:initialdata}--\eqref{P:minmax:ini}).
Thus, taking into account that $h\equiv1$ on $[1,+\infty)$ and rearranging the terms, we deduce,
for every $t\in[0,T]$, that
\begin{equation}
\label{max:1}
	\frac 12 \IO2 {(R_{\lambda,1} - 1)_+}
	 + \int_{Q_{t,\lambda}^1 } |\nabla (R_{\lambda,1}-1)|^2
	 +\int_{Q_{t,\lambda}^1 } [c_1 h(P_\lambda)
	  - c_2 J_\lambda^1(\bph_{\lambda})](R_{\lambda,1}-1)
	   = 0.
\end{equation}
Besides, combining \eqref{PAR:SYS:1}--\eqref{PAR:SYS:3} we infer that the linear combination
$G  := \br_\lambda \cdot \1 - P_\lambda$ yields a solution to the parabolic system:
\begin{align*}
\begin{cases}
	\dt G - \Delta G = 0 \quad & \text{in $Q$,}
	\\
	\dn G =0 \quad & \text{on $\Sigma,$}
	\\
	G(0) = 1 - 2 P_0 \quad & \text{in $\Omega$.}
\end{cases}
\end{align*}
Therefore, using standard arguments and the minimum principle \eqref{P:minmax:ini}, we get
\begin{align}\label{confronto}
	0 \leq \norma{G(t)}_{L^\infty(\Omega)} \leq \norma{1-2P_0}_{L^\infty(\Omega)}
	\qquad \text{for a.e.~$t \in (0,T)$}.
\end{align}
Then, the definition of $G$ and \eqref{min_princ}
imply, for almost every $(x,t) \in Q$, that
\begin{align*}
	0 \leq R_{\lambda,1}(x,t) \leq
	R_{\lambda,1} (x,t)+ R_{\lambda,2}(x,t) \leq
	P_\lambda(x,t) + \norma{1-2P_0}_{L^\infty(\Omega)},
\end{align*}
and, in particular, we have that
\[
  R_{\lambda,1}, R_{\lambda,2}\leq
  R_{\lambda,1} + R_{\lambda,2}\leq
   P_\lambda + \norma{1-2P_0}_{L^\infty(\Omega)}
  \qquad\text{a.e.~in } Q.
\]
Thus, we can write
\begin{align*}
	Q_{t,\lambda}^1 =  Q_t  \cap \{R_{\lambda,1} >1 \} & =  Q_t  \cap
	\{R_{\lambda,1} >1 \} \cap \{ R_{\lambda,1} \leq
	P_\lambda + \norma{1-2P_0}_{L^\infty(\Omega)}\}
	\\
	& = Q_t  \cap  \{ 1 < R_{\lambda,1} \leq  P_\lambda + \norma{1-2P_0}_{L^\infty(\Omega)}\} .
\end{align*}
Bearing this in mind, we consider the last integrand in \eqref{max:1}. Recalling
the definition of $h$ and the fact that $J_\lambda^1\leq1$, we deduce that
\begin{align*}
	&\int_{Q_{t,\lambda}^1 } [c_1 h(P_\lambda) -
	c_2 J_\lambda^1(\bph_{\lambda})](R_{\lambda,1}-1)\\
	 &\quad=\int_{Q_{t,\lambda}^1\cap\{P_\lambda>1\} }
	[c_1  - c_2 J_\lambda^1(\bph_{\lambda})](R_{\lambda,1}-1)\\
	&\qquad+\int_{Q_{t,\lambda}^1\cap
	\{1-\norma{1-2P_0}_{L^\infty(\Omega)}<P_\lambda\leq1\} }
	[c_1 P_\lambda- c_2 J_\lambda^1(\bph_{\lambda})](R_{\lambda,1}-1)\\
	&\quad\geq \int_{Q_{t,\lambda}^1\cap\{P_\lambda>1\} }
	(c_1  - c_2)(R_{\lambda,1}-1)\\
	&\qquad+\int_{ Q_{t,\lambda}^1\cap\{1-\norma{1-2P_0}_{L^\infty(\Omega)}<P_\lambda\leq1\}}
	 \big[c_1 \big(1 - \norma{1-2P_0}_{L^\infty(\Omega)}\big) - c_2\big](R_{\lambda,1}-1).
\end{align*}
It is worth observing that the first term on the \rhs\ is nonnegative since $c_1>c_2$ by
assumption and $R_{\lambda,1}>1$ in $Q_{t,\lambda}^1$.
As for the second one,
\eqref{max:ini:Linf} implies in particular that $\norma{P_0-
\frac 12 }_{L^\infty(\Omega)} \leq \frac 12 (1 - \frac {c_2}{c_1})$,
whence also $\norma{1 - 2P_0}_{L^\infty(\Omega)} \leq 1 - \frac {c_2}{c_1}$.
It follows that $c_1(1-\norma{1 - 2P_0}_{L^\infty(\Omega)})\geq c_2$,
so that also the second term on the left-hand side is nonnegative.
Collecting all the information, we obtain
\[
  \frac 12 \IO2 {(R_{\lambda,1} - 1)_+}
	 + \int_{Q_{t,\lambda}^1 } |\nabla (R_{\lambda,1}-1)|^2 \leq0 \qquad\forall\,t\in[0,T],
\]
which entails
\begin{align*}
	R_{\lambda,1}(x,t) \leq 1 \quad \text{for a.e.~$(x,t) \in Q.$}
\end{align*}
Arguing in a similar fashion for $R_{2,\lambda}$ in equation \eqref{PAR:SYS:3},
using that from \eqref{max:ini:Linf} we have also
$\norma{P_0-\frac 12 }_{L^\infty(\Omega)} \leq \frac 12 (1 - \frac {c_4}{c_3})$,
where $c_3>c_4$ from \eqref{c}, we also get
\begin{align*}
	R_{\lambda,2}(x,t) \leq 1 \quad \text{for a.e.~$(x,t) \in Q.$}
\end{align*}
The upper bound for $P_\lambda$ can be proven in a similar way.
Let us test \eqref{PAR:SYS:1} by $(P_\lambda-1)_+$. This gives
\begin{align*}
	& \frac 12 \IO2 {(P_\lambda - 1)_+}
	 + \int_{Q_t \cap \{P_\lambda >1\}} |\nabla (P_\lambda-1)|^2\\
	  &\qquad+\int_{Q_t \cap \{P_\lambda >1\}}
	  [c_1h(R_{\lambda,1}) + c_3h(R_{2,\lambda})]h(P_\lambda)(P_\lambda-1)\\
	  &\quad =
	  \int_{Q_t \cap \{P_\lambda >1\}}
	  [c_2J_\lambda^1(\bph_{\lambda})
	  + c_4J_\lambda^2(\bph_{\lambda})](P_\lambda-1),
\end{align*}
where we have used that $\frac 12 \norma{(P_0 - 1)_+}^2=0$
on account of \eqref{P:minmax:ini}.
We know that $0\leq R_{\lambda,i}\leq1$ almost everywhere in $Q$, for $i=1,2$. Thus
by definition of $h$ and the fact that $J_\lambda^i\leq1$ for $i=1,2$, we infer that
\begin{align*}
	& \frac 12 \IO2 {(P_\lambda - 1)_+}
	 + \int_{Q_t \cap \{P_\lambda >1\}} |\nabla (P_\lambda-1)|^2
	 \\ & \quad
	  +\int_{Q_t \cap \{P_\lambda >1\}}
	  [(\min\{c_1,c_3\}(R_{\lambda,1} + R_{\lambda,2})
	    - (c_2 +c_4)](P_\lambda-1) \leq 0.
\end{align*}
On the other hand, from \eqref{min_princ}, \eqref{confronto}, and
the definition of $G$ we also notice that,
for almost every $(x,t) \in Q$, we have
\begin{align*}
	0 \leq P_{\lambda}(x,t) \leq
	R_{\lambda,1} (x,t)+ R_{\lambda,2}(x,t)
	 + \norma{1-2P_0}_{L^\infty(\Omega)},
\end{align*}
so that, in particular, it holds that
\[
  P_{\lambda}\leq
  R_{\lambda,1} + R_{\lambda,2}
  + \norma{1-2P_0}_{L^\infty(\Omega)}
  \qquad\text{a.e.~in } Q.
\]
Bearing this in mind, we notice that
\begin{align*}
	Q_t \cap  \{P_\lambda >1 \}  &  = Q_t \cap
	\{P_\lambda >1 \} \cap \{ R_{1,\lambda}+R_{2,\lambda} \geq
	P_\lambda - \norma{1-2P_0}_{L^\infty(\Omega)}\}
	\\ & = Q_t \cap
	\{R_{1,\lambda}+R_{2,\lambda}
	\geq  P_\lambda - \norma{1-2P_0}_{L^\infty(\Omega)}
	> 1- \norma{1-2P_0}_{L^\infty(\Omega)} \}.
\end{align*}
Thus, it follows that
\begin{align*}
	& \frac 12 \IO2 {(P_\lambda - 1)_+}
	 + \int_{Q_t \cap \{P_\lambda >1\}} |\nabla (P_\lambda-1)|^2
	 \\ & \quad
	  +\int_{Q_t \cap \{P_\lambda >1\}}
	  \big[\min\{c_1,c_3\}\big(1- \norma{1-2P_0}_{L^\infty(\Omega)}\big)
	    - (c_2 +c_4)\big](P_\lambda-1) \leq 0.
\end{align*}
Observing that \eqref{max:ini:Linf} implies
$\min\{c_1,c_3\}(1- \norma{1-2P_0}_{L^\infty(\Omega)})\geq c_2+c_4$,
we conclude that
\begin{align*}
	& \frac 12 \IO2 {(P_\lambda - 1)_+}
	 + \int_{Q_t \cap \{P_\lambda >1\}} |\nabla (P_\lambda-1)|^2 \leq 0
	 \qquad\forall\,t\in[0,T],
\end{align*}
which entails
\begin{align*}
	P(x,t) \leq 1 \quad \text{for a.e. $(x,t) \in Q.$}
\end{align*}
Summing up, we have proven that
\begin{align}
	\label{max_princ}
	P_\lambda(x,t) \leq 1, \quad
	R_{\lambda,i}(x,t) \leq 1
	\quad
	\text{for a.e.~$(x,t)\in Q$},
	\quad i=1,2,
\end{align}
independently of $\lambda$.
\end{proof}
We conclude by pointing out that, owing to Lemma \ref{lem1}, from now on we can avoid writing the truncation function $h$ explicitly because of \eqref{sep_PR}.

\subsection{Preliminary estimates on the Cahn--Hilliard system}
\label{sec:prelim_est}
Here we provide uniform estimates, independent of $\lambda$,
that help in controlling the behavior of the phase variable $\bph_\lambda$ close to the initial time. As we pointed out, this
is necessary to find a suitable integrable control of the chemical potential $\bmu_\lambda$ or, equivalently, of $\Psi_{\lambda,\bph}^{(1)}(\bph_\lambda)$.
We remind that the standard strategy cannot be used because of the initial condition $\bph_\lambda(0)= \0.$
As usual, we will denote by $C$ a generic positive constant, independent of $\lambda$,
whose value may be updated at each step in the following.

We test equation \eqref{SYS:1_app} by $\bph_\lambda$ and
equation \eqref{SYS:2_app} by $-\Delta\bph_\lambda$. Adding the resulting identities together and integrating  by parts, we infer that
\begin{align*}
  &\frac12\norma{\bph_\lambda(t)}^2 + \int_{Q_t}|\Delta\bph_\lambda|^2
  +\int_{Q_t}\Psi_{\lambda,\bph}^{(1)}(\bph_\lambda)\cdot(-\Delta\bph_\lambda)\\
  &\quad =\int_{Q_t}\Sph^\lambda(\bph_\lambda, P_\lambda, \br_\lambda)\cdot\bph_\lambda
  -\int_{Q_t}\Psi_{\bph}^{(2)}(\bph_\lambda)\cdot(-\Delta\bph_\lambda).
\end{align*}
On account of the boundedness of $h$ and $\bj_\lambda$ and the Lipschitz continuity of $\Psi^{(2)}_\bph$,
on the right-hand side we have
\begin{align*}
  &\int_{Q_t}\Sph^\lambda(\bph_\lambda, P_\lambda, \br_\lambda)\cdot\bph_\lambda
  -\int_{Q_t}\Psi_{\bph}^{(2)}(\bph_\lambda)\cdot(-\Delta\bph_\lambda)\\
  &\quad \leq C\int_0^t\norma{\bph_\lambda(s)}\,\ds
  +\frac12\int_{Q_t}|\Delta\bph_\lambda|^2 + C\int_0^t\norma{\bph_\lambda(s)}^2\,\ds.
\end{align*}
The convexity of $\Psi^{(1)}_\lambda$ implies that
also the third term on the left-hand side is nonnegative. To show this{,}
we use the notation
\[
  \Psi^{(1)}_{\lambda,\bph}(\phi_1,\phi_2)=
  \nabla\Psi^{(1)}_{\lambda}(\phi_1,\phi_2) =
  \big(D_1\Psi^{(1)}_{\lambda}(\phi_1,\phi_2), D_2\Psi^{(1)}_{\lambda}(\phi_1,\phi_2) \big),
  \quad(\phi_1,\phi_2)\in\erre^2.
\]
Elementary computations give
\begin{align*}
  &\int_{Q_t}\Psi_{\lambda,\bph}^{(1)}(\bph_\lambda)\cdot(-\Delta\bph_\lambda)
   =\sum_{i=1,2}\int_{Q_t}(D_i\Psi_{\lambda}^{(1)})(\bph_\lambda)(-\Delta\ph_{\lambda,i})\\
  &\quad =\sum_{i=1,2}\int_{Q_t}
  \nabla[(D_i\Psi_{\lambda}^{(1)})(\bph_\lambda)]\cdot\nabla\ph_{\lambda,i}
  \\ \non &  \quad
  =\sum_{i=1,2}\sum_{j=1,2}\int_{Q_t}
  D_j[(D_i\Psi_{\lambda}^{(1)})(\bph_\lambda)]D_j\ph_{\lambda,i}\\
  &\quad =\int_{Q_t}
  (D_1D_1\Psi_\lambda^{(1)})(\bph_\lambda)
  \left(|D_1\ph_{\lambda,1}|^2 + |D_2\ph_{\lambda,1}|^2\right)\\
  &\qquad+
  \int_{Q_t}
  (D_2D_2\Psi_\lambda^{(1)})(\bph_\lambda)
  \left(|D_1\ph_{\lambda,2}|^2 + |D_2\ph_{\lambda,2}|^2\right)\\
  &\qquad+2\int_{Q_t}
  (D_1D_2\Psi_\lambda^{(1)})(\bph_\lambda)
  \left((D_1\ph_{\lambda,1})(D_1\ph_{\lambda,2}) +
  (D_2\ph_{\lambda,1})(D_2\ph_{\lambda,2})\right)\\
  &\quad =\int_{Q_t}(D_1\bph_\lambda)^\top \Psi_{\lambda,\bph\bph}^{(1)}(\bph_\lambda)(D_1\bph_\lambda)
  + \int_{Q_t}(D_2\bph_\lambda)^\top \Psi_{\lambda,\bph\bph}^{(1)}(\bph_\lambda)(D_2\bph_\lambda) \geq0,
\end{align*}
where the last inequality follows from the convexity of $\Psi_\lambda^{(1)}$.
As a consequence, rearranging the terms, we are left with
\begin{equation}\label{est_iteration}
  \frac12\norma{\bph_\lambda(t)}^2 + \frac12\int_{Q_t}|\Delta\bph_\lambda|^2
  \leq C\int_0^t\norma{\bph_\lambda(s)}\,\ds
  + C\int_0^t\norma{\bph_\lambda(s)}^2\,\ds
\end{equation}
and the Gronwall lemma, along with elliptic regularity, yields
\[
  \norma{\bph_\lambda}_{L^\infty(0,T;\HH)\cap L^2(0,T;\WW)}\leq C.
\]
At this point, one can plug this estimate into \eqref{est_iteration} again, to deduce that
\[
  \norma{\bph_\lambda}_{L^\infty(0,t;\HH)\cap L^2(0,t; \WW)}\leq Ct^{1/2}
  \qquad\forall\,t\in[0,T],
\]
possibly updating the value of $C$ independently of $\lambda$.
Iterating the argument{,} i.e., using \eqref{est_iteration} once more{,}
we infer the refinement
\[
  \norma{\bph_\lambda}_{L^\infty(0,t;\HH)\cap L^2(0,t; \WW)}\leq Ct^{3/4}
  \qquad\forall\,t\in[0,T].
\]
Therefore, arguing recursively, we get
\begin{equation}
  \label{est1}
  \forall\,\alpha\in[0,1),\quad\exists\,C_\alpha>0:\quad
  \norma{\bph_\lambda}_{L^\infty(0,t;\HH)\cap L^2(0,t; \WW)}\leq
  C_\alpha t^\alpha
  \qquad\forall\,t\in[0,T],
\end{equation}
where the constants $\{C_\alpha\}_{\alpha>0}$ are independent of $\lambda$.

\subsection{Estimates on the reaction-diffusion subsystem}
Let us now provide uniform estimates on the variables $(P_\lambda, \br_\lambda)$
by exploiting the parabolic structure of the subsystem \eqref{SYS:3_app}--\eqref{SYS:4_app}.
Indeed, thanks to \eqref{est1} and to the boundedness of $h$ and
$\bj_ \lambda$, one has
\[
  \norma{\SP^\lambda(\bph_\lambda, P_\lambda, \br_\lambda)}_{L^\infty(0,T;H)}+
  \norma{\SR^\lambda(\bph_\lambda, P_\lambda, \br_\lambda)}_{L^\infty(0,T;\HH)}
  \leq C,
\]
thanks to the assumption \eqref{ass:exweak:initialdata}--\eqref{P:minmax:ini} on the initial data,
the classical parabolic regularity theory yields
\begin{equation}\label{est2}
  \norma{P_\lambda}_{H^1(0,T; V^*)\cap L^2(0,T; V)\cap L^\infty(Q)}
  +\norma{\br_\lambda}_{H^1(0,T; \VV^*)\cap L^2(0,T; \VV)\cap \LL^\infty(Q)} \leq C.
\end{equation}

\subsection{Dual estimates on the phase variable}
\label{ssec:dual}
In this subsection, we perform a nonstandard Cahn--Hilliard estimate.
The presence of the source $\Sph^\lambda$
forces us to use an alternative strategy to handle it when multiplied by $\bmu_\lambda$. Indeed, as already mentioned, we cannot get an $L^2(0,T;\HH)$-bound for $\bmu_\lambda$
since we cannot control the spatial average of $\Psi^{(1)}_{\lambda,\bph}(\bph_\lambda)$ being $\bph_\lambda(0)=\0$.
We, therefore, exploit the behavior of the convex conjugate $(\Psi^{(1)}_\lambda)^*$.

Given an arbitrary $t\in(0,T)$,
we integrate equation \eqref{SYS:1_app} in time,
we test it by $\bmu_\lambda$, and sum with equation \eqref{SYS:2_app}
tested by $\bph_\lambda$ to obtain that
\begin{align}
	\notag
  &\int_0^t\int_{Q_s}|\nabla\bmu_\lambda|^2\,\ds+
  \int_{Q_t}|\nabla\bph_\lambda|^2
  +\int_{Q_t}\Psi^{(1)}_{\lambda,\bph}(\bph_\lambda)\cdot\bph_\lambda\\
  &\quad=\int_{0}^t\int_{Q_s}
  \Sph^\lambda(\bph_\lambda, P_\lambda, \br_\lambda)\cdot
  \bmu_\lambda\,\ds
  -\int_{Q_t}\Psi^{(2)}_{\bph}(\bph_\lambda)\cdot\bph_\lambda,
  \label{dual:est:1}
\end{align}
where $Q_s:=\Omega\times[0,s]$.
Thanks to the Young inequality and classical results of convex analysis,
on the left-hand side one has
\[
  \int_{Q_t}\Psi^{(1)}_{\lambda,\bph}(\bph_\lambda)\cdot\bph_\lambda
  =\int_{Q_t}\Psi_\lambda^{(1)}(\bph_\lambda) +
  \int_{Q_t}(\Psi_\lambda^{(1)})^*(\Psi^{(1)}_{\lambda,\bph}(\bph_\lambda)),
\]
where the convex conjugate is defined as usual as
\[
  (\Psi_\lambda^{(1)})^*:\erre^2\to(-\infty,+\infty], \qquad
  (\Psi_\lambda^{(1)})^*(\boldsymbol{z}):=\sup_{\rr\in\erre^2}
  \left\{\boldsymbol{z}\cdot\rr - \Psi_\lambda^{(1)}(\rr)\right\}, \quad
  \boldsymbol{z}\in\erre^2.
\]
On the right-hand side, one readily sees that the
second term in \eqref{dual:est:1} is
bounded uniformly in $\lambda$ thanks to assumption \ref{ass:pot}
and estimate \eqref{est1}. As for the first term,
exploiting in the order equation \eqref{SYS:2_app},
the form of $\Sph^\lambda$, the boundedness of
$h$ and $\bj_\lambda$, the bound \eqref{sep_PR},
and estimate \eqref{est1} once more, we deduce
\begin{align*}
  &\int_{0}^t\int_{Q_s}\Sph^\lambda(\bph_\lambda, P_\lambda, \br_\lambda)\cdot
  \bmu_\lambda\,\ds\\
  &\quad = \int_{0}^t\int_{Q_s}\Sph^\lambda(\bph_\lambda, P_\lambda, \br_\lambda)\cdot
  (-\Delta\bph_\lambda)\,\ds
  +\int_{0}^t\int_{Q_s}\Sph^\lambda(\bph_\lambda, P_\lambda, \br_\lambda)\cdot
  \Psi^{(1)}_{\lambda,\bph}(\bph_\lambda)\,\ds +C\\
  &\quad \leq C +
  \int_{0}^t\int_{Q_s}
  \begin{pmatrix}
  c_{1}P_\lambda R_{\lambda,1} - c_2J_\lambda^1(\bph_{\lambda})\\
  c_{3}P_\lambda R_{\lambda,2} - c_4J_\lambda^2(\bph_{\lambda})
  \end{pmatrix}
  \cdot
  \Psi^{(1)}_{\lambda,\bph}(\bph_\lambda)\,\ds.
\end{align*}
Collecting the above inequalities and rearranging the terms, we obtain that
\begin{align*}
  &\int_0^t\int_{Q_s}|\nabla\bmu_\lambda|^2\,\ds+
  \int_{Q_t}|\nabla\bph_\lambda|^2
  +\int_{Q_t}\Psi_\lambda^{(1)}(\bph_\lambda) +
  \int_{Q_t}(\Psi_\lambda^{(1)})^*(\Psi^{(1)}_{\lambda,\bph}(\bph_\lambda))\\
  &\qquad+ \int_{0}^t\int_{Q_s}
 \begin{pmatrix}
  c_{2}J_\lambda^1(\bph_\lambda)\\
  c_{4}J_\lambda^2 (\bph_\lambda)
  \end{pmatrix}
  \cdot
  \Psi^{(1)}_{\lambda,\bph}(\bph_\lambda)\,\ds\\
  &\quad \leq C + \int_{0}^t\int_{Q_s}
  \begin{pmatrix}
  c_{1}P_\lambda R_{\lambda,1}\\
  c_{3}P_\lambda R_{\lambda,2}
  \end{pmatrix}
  \cdot
  \Psi^{(1)}_{\lambda,\bph}(\bph_\lambda)\,\ds.
\end{align*}
At this point, thanks to \eqref{sep_PR}, there exists $\eps_0>0$
independent of $\lambda$ such that
\[
  \eps_0
  \begin{pmatrix}
  c_{1}P_\lambda R_{\lambda,1}\\
  c_{3}P_\lambda R_{\lambda,2}
  \end{pmatrix}
  \in \simcl \qquad\forall\,\lambda\in(0,1).
\]
Hence, using the Young inequality and the properties of the Moreau--Yosida
regularization, on the right-hand side we have
\begin{align*}
  &\int_{0}^t\int_{Q_s}
  \begin{pmatrix}
  c_{1}P_\lambda R_{\lambda,1}\\
  c_{3}P_\lambda R_{\lambda,2}
  \end{pmatrix}
  \cdot
  \Psi^{(1)}_{\lambda,\bph}(\bph_\lambda)\,\ds\\
  &\quad =\frac1{\eps_0}
  \int_{0}^t\int_{Q_s}
  \eps_0
  \begin{pmatrix}
  c_{1}P_\lambda R_{\lambda,1}\\
  c_{3}P_\lambda R_{\lambda,2}
  \end{pmatrix}
  \cdot
  \Psi^{(1)}_{\lambda,\bph}(\bph_\lambda)\,\ds\\
  &\quad \leq\frac1{\eps_0}
  \int_{0}^t\int_{Q_s}
  \left[
  \Psi_\lambda^{(1)}\left(
  \eps_0
  \begin{pmatrix}
  c_{1}P_\lambda R_{\lambda,1}\\
  c_{3}P_\lambda R_{\lambda,2}
  \end{pmatrix}
  \right)
  +(\Psi_\lambda^{(1)})^*
  (\Psi^{(1)}_{\lambda,\bph}(\bph_\lambda))
  \right]\,\ds\\
  &\quad \leq\frac1{\eps_0}T|Q|\max_{\rr\in\simcl}\Psi^{(1)}(\rr)
  +\frac1{\eps_0}\int_{0}^t\int_{Q_s}(\Psi_\lambda^{(1)})^*
 (\Psi^{(1)}_{\lambda,\bph}(\bph_\lambda))\,\ds.
\end{align*}
On the left-hand side, recalling that
$\Psi_{\lambda,\bph}^{(1)}=\Psi_\bph^{(1)}\circ\bj_\lambda=
(\nabla\Psi^{(1)})\circ\bj_\lambda$,
we have in turn that
\begin{align*}
  \int_{0}^t\int_{Q_s}
  \begin{pmatrix}
  c_2J_\lambda^1(\bph_{\lambda})\\
  c_4J_\lambda^2(\bph_{\lambda})
  \end{pmatrix}
  \cdot
  \Psi^{(1)}_{\lambda,\bph}(\bph_\lambda)
  &=\int_{0}^t\int_{Q_s}
  \begin{pmatrix}
  c_2J_\lambda^1(\bph_{\lambda})\\
  c_4J_\lambda^2(\bph_{\lambda})
  \end{pmatrix}
  \cdot
  \Psi^{(1)}_{\bph}(\bj_\lambda(\bph_\lambda))\\
  &=\sum_{i=1,2}c_{2i}\int_{0}^t\int_{Q_s}J_\lambda^i(\bph_{\lambda})
  (D_i\Psi^{(1)})(\bj_\lambda(\bph_\lambda)).
\end{align*}
To handle this expression,
for every $x,y\in(0,1)$, we introduce the sections
\begin{alignat*}{2}
  &f_{1}^y:[0,1-y]\to[0,+\infty), \qquad
  &&f_{1}^y(r):=\Psi^{(1)}(r,y), \quad r\in\erre,\\
  &f_{2}^x:[0,1-x]\to[0,+\infty), \qquad
  &&f_{2}^x(r):=\Psi^{(1)}(x,r), \quad r\in\erre.
\end{alignat*}
Clearly, for all $x,y\in(0,1)$,
the functions $f_1^y$ and $f_2^x$
are convex, differentiable in $(0,1-y)$ and $(0,1-x)$, respectively,  and satisfy
\begin{align*}
  &(f_1^y)'(r)=(D_1\Psi^{(1)})(r,y) \qquad\forall\,r\in(0,1-y),\\
  &(f_2^x)'(r)=(D_2\Psi^{(1)})(x,r) \qquad\forall\,r\in(0,1-x).
\end{align*}
Moreover, their convex conjugates satisfy
\begin{alignat*}{2}
  (f_1^y)^*(z) &= \sup_{r\in[0,1-y]}\left\{rz-\Psi^{(1)}(r,y)\right\}
  \geq - \Psi(0,y)\geq - \max_{\rr\in\simcl}\Psi^{(1)}(\rr) \qquad&\forall\,z\in\erre,\\
  (f_2^x)^*(z) &= \sup_{r\in[0,1-x]}\left\{rz-\Psi^{(1)}(x,r)\right\}
  \geq - \Psi(x,0)\geq - \max_{\rr\in\simcl}\Psi^{(1)}(\rr)
  \qquad&\forall\,z\in\erre.
\end{alignat*}
Therefore,
owing to classical results of convex analysis, we infer that
\begin{align*}
  (D_1\Psi^{(1)})(x,y)x &=
  (f_1^y)'(x)x = f_1^y(x) + (f_1^y)^*((f_1^y)'(x)) \geq - \max_{\rr\in\simcl}\Psi^{(1)}(\rr),\\
  (D_2\Psi^{(1)})(x,y)y &=
  (f_2^x)'(y)y = f_2^x(y) +
  (f_2^x)^*((f_2^x)'(y)) \geq - \max_{\rr\in\simcl}\Psi^{(1)}(\rr),
\end{align*}
for every $(x,y)\in\simap$.
Taking the previous considerations into account
and using the fact that $\bj_\lambda(\bph_\lambda)\in\simap$
almost everywhere in $Q$, we deduce that
\begin{align*}
  \int_{0}^t\int_{Q_s}
  \begin{pmatrix}
  c_2J_\lambda^1(\bph_{\lambda})\\
  c_4J_\lambda^2(\bph_{\lambda})
  \end{pmatrix}
  \cdot
  \Psi^{(1)}_{\lambda,\bph}(\bph_\lambda)
  &=\sum_{i=1,2}c_{2i}\int_{0}^t\int_{Q_s}J_\lambda^i(\bph_{\lambda})
  (D_i\Psi^{(1)})(\bj_\lambda(\bph_\lambda))\\
  &\geq - T|Q|(c_2+c_4)\max_{\rr\in\simcl}\Psi^{(1)}(\rr).
\end{align*}
Summing up, rearranging all the terms, we have
\begin{align*}
  &\int_0^t\int_{Q_s}|\nabla\bmu_\lambda|^2\,\ds+
  \int_{Q_t}|\nabla\bph_\lambda|^2
  +\int_{Q_t}\Psi_\lambda^{(1)}(\bph_\lambda) +
  \int_{Q_t}(\Psi_\lambda^{(1)})^*(\Psi^{(1)}_{\lambda,\bph}(\bph_\lambda))\\
  &\quad \leq C + C\int_{0}^t\int_{Q_s}(\Psi_\lambda^{(1)})^*
 (\Psi^{(1)}_{\lambda,\bph}(\bph_\lambda))\,\ds.
\end{align*}
Now, noting that $(\Psi^{(1)}_\lambda)^*\geq-\max_{\rr\in\simcl}\Psi^{(1)}(\rr)$ as well,
we have that
\begin{align*}
\int_{\Omega}|(\Psi_\lambda^{(1)})^*
 (\Psi^{(1)}_{\lambda,\bph}(\bph_\lambda))|&=
 \int_{\{(\Psi_\lambda^{(1)})^*
 (\Psi^{(1)}_{\lambda,\bph}(\bph_\lambda))\geq0\}}
 (\Psi_\lambda^{(1)})^*
 (\Psi^{(1)}_{\lambda,\bph}(\bph_\lambda))\\
 &\quad -\int_{\{(\Psi_\lambda^{(1)})^*
 (\Psi^{(1)}_{\lambda,\bph}(\bph_\lambda))<0\}}
 (\Psi_\lambda^{(1)})^*
 (\Psi^{(1)}_{\lambda,\bph}(\bph_\lambda))\\
 &\leq|\Omega|\max_{\rr\in\simcl}\Psi^{(1)}(\rr) +
 \int_{\{(\Psi_\lambda^{(1)})^*
 (\Psi^{(1)}_{\lambda,\bph}(\bph_\lambda))\geq0\}}
 (\Psi_\lambda^{(1)})^*
 (\Psi^{(1)}_{\lambda,\bph}(\bph_\lambda))
\end{align*}
where
\begin{align*}
  &\int_{\{(\Psi_\lambda^{(1)})^*
 (\Psi^{(1)}_{\lambda,\bph}(\bph_\lambda))\geq0\}}
 (\Psi_\lambda^{(1)})^*
 (\Psi^{(1)}_{\lambda,\bph}(\bph_\lambda))\\
 &\quad =\int_\Omega
 (\Psi_\lambda^{(1)})^*
 (\Psi^{(1)}_{\lambda,\bph}(\bph_\lambda))
 - \int_{\{(\Psi_\lambda^{(1)})^*
 (\Psi^{(1)}_{\lambda,\bph}(\bph_\lambda))<0\}}
 (\Psi_\lambda^{(1)})^*
 (\Psi^{(1)}_{\lambda,\bph}(\bph_\lambda))\\
 &\quad \leq |\Omega|\max_{\rr\in\simcl}\Psi^{(1)}(\rr) +\int_\Omega
 (\Psi_\lambda^{(1)})^*
 (\Psi^{(1)}_{\lambda,\bph}(\bph_\lambda)),
\end{align*}
from which we deduce
\[
  \int_{\Omega}|(\Psi_\lambda^{(1)})^*
 (\Psi^{(1)}_{\lambda,\bph}(\bph_\lambda))|
 \leq2|\Omega|\max_{\rr\in\simcl}\Psi^{(1)}(\rr)
 +\int_\Omega
 (\Psi_\lambda^{(1)})^*
 (\Psi^{(1)}_{\lambda,\bph}(\bph_\lambda)).
\]
Taking this into account we have also that
\begin{align*}
  &\int_0^t\int_{Q_s}|\nabla\bmu_\lambda|^2\,\ds+
  \int_{Q_t}|\nabla\bph_\lambda|^2
  +\int_{Q_t}\Psi_\lambda^{(1)}(\bph_\lambda) +
  \int_{Q_t}|(\Psi_\lambda^{(1)})^*(\Psi^{(1)}_{\lambda,\bph}(\bph_\lambda))|\\
  &\quad \leq C + C\int_{0}^t\int_{Q_s}|(\Psi_\lambda^{(1)})^*
 (\Psi^{(1)}_{\lambda,\bph}(\bph_\lambda))|\,\ds.
\end{align*}
The Gronwall lemma readily implies that
there exists a constant $C>0$, independent of $\lambda$, such that
\begin{equation}
  \label{est2:2}
  \norma{\Psi_\lambda^{(1)}(\bph_\lambda)}_{L^1(Q)} +
  {\norma{(\Psi_\lambda^{(1)})^*(\Psi^{(1)}_{\lambda,\bph}(\bph_\lambda))}_{L^1(Q)}} \leq C.
\end{equation}
Furthermore, recalling the definition of Yosida approximation and using
classical results of convex analysis, we get
\begin{align*}
  \Psi_\lambda^{(1)}(\bph_\lambda) +
  (\Psi_\lambda^{(1)})^*(\Psi_{\lambda,\bph}^{(1)}(\bph_\lambda))
  &=\Psi_{\lambda,\bph}^{(1)}(\bph_\lambda)\cdot\bph_\lambda\\
  &=\Psi_{\lambda,\bph}^{(1)}(\bph_\lambda)\cdot
  (\bph_\lambda -\bj_\lambda(\bph_\lambda))
  +\Psi_{\lambda,\bph}^{(1)}(\bph_\lambda)\cdot\bj_\lambda(\bph_\lambda)\\
  &=\Psi_{\lambda,\bph}^{(1)}(\bph_\lambda)\cdot
  (\bph_\lambda -\bj_\lambda(\bph_\lambda))
  +\Psi_{\bph}^{(1)}(\bj_\lambda(\bph_\lambda))\cdot\bj_\lambda(\bph_\lambda)\\
  &=\lambda\norma{\Psi_{\lambda,\bph}^{(1)}(\bph_\lambda)}^2+
  \Psi^{(1)}(\bj_\lambda(\bph_\lambda))
  +(\Psi^{(1)})^*(\Psi_{\bph}(\bj_\lambda(\bph_\lambda)))\\
  &\geq\lambda\norma{\Psi_{\lambda,\bph}^{(1)}(\bph_\lambda)}^2
  -\max_{\rr\in\simcl}\Psi^{(1)}(\rr),
\end{align*}
which, thanks to \eqref{est2:2}, yields
\begin{equation}\label{est3}
\lambda^{1/2}\norma{\Psi_{\lambda,\bph}(\bph_\lambda)}_{L^2(0,T;\HH)}\leq C.
\end{equation}

\subsection{Energy estimate}\label{sub:energy}
Let us now perform the standard energy estimate on the Cahn--Hilliard system.
Namely, we test \eqref{SYS:1_app} by $\bmu_\lambda$, \eqref{SYS:2_app} by $\dt\bph_\lambda$,
add the resulting equalities, and integrate by parts to find that
\[
	\frac12\int_\Omega|\nabla\bph_\lambda(t)|^2
	+\int_\Omega \Psi_\lambda(\bph_\lambda(t))
	+ \int_{Q_t}|\nabla\bmu_\lambda|^2
	= \int_\Omega \Psi_\lambda(\bph_0)
	+\int_{Q_t} \Sph^\lambda
	(\bph_\lambda, P_\lambda, \br_\lambda) \cdot \bmu_\lambda.
\]
On the left-hand side, thanks to the Lipschitz continuity of $\Psi_\bph^{(2)}$
and the estimate \eqref{est1}, one has that
\[
  \int_\Omega \Psi_\lambda(\bph_\lambda(t))=
  \int_\Omega \Psi_\lambda^{(1)}(\bph_\lambda(t))+
  \int_\Omega \Psi^{(2)}(\bph_\lambda(t))\geq
  \int_\Omega \Psi_\lambda^{(1)}(\bph_\lambda(t)) -C.
\]
On the right-hand side, the initial value can be bounded
by assumption \ref{ass:pot} and \eqref{ass:exweak:initialdata} as
\begin{align*}
  \int_\Omega \Psi_\lambda(\bph_0)
  = \int_\Omega \Psi_\lambda(\0)
  \leq\int_\Omega \Psi(\0)\leq C.
\end{align*}
Hence, we get
\begin{equation}
	\frac12\int_\Omega|\nabla\bph_\lambda(t)|^2
	+\int_\Omega \Psi_\lambda^{(1)}(\bph_\lambda(t))
	+ \int_{Q_t}|\nabla\bmu_\lambda|^2
	\leq C
	+\int_{Q_t} \Sph^\lambda(\bph_\lambda, P_\lambda, \br_\lambda) \cdot \bmu_\lambda.
	\label{en:est}
\end{equation}
Recalling \eqref{def:sorgenti_app}--\eqref{def:Sp_app}
we have, using also \eqref{sep_PR},
\begin{align*}
  \int_{Q_t} \Sph^\lambda(\bph_\lambda, P_\lambda, \br_\lambda) \cdot \bmu_\lambda
  =\int_{Q_t}
  \begin{pmatrix}
  c_{1}P_\lambda R_{\lambda,1}\\
  c_{3}P_\lambda R_{\lambda,2}
  \end{pmatrix}
  \cdot\bmu_\lambda
  -\int_{Q_t}
  \begin{pmatrix}
  c_{2}J_\lambda^1(\bph_{\lambda})\\
  c_{4}J_\lambda^2(\bph_{\lambda})
  \end{pmatrix}
  \cdot\bmu_\lambda.
\end{align*}
By the boundedness of $h$ and $\bj_\lambda$, equation \eqref{SYS:2_app},
the Lipschitz continuity of $\Psi_{\bph}^{(2)}$, and estimate \eqref{est1}, we deduce
\begin{align*}
  \int_{Q_t}
  \begin{pmatrix}
  c_{1}P_\lambda R_{\lambda,1}\\
  c_{3}P_\lambda R_{\lambda,2}
  \end{pmatrix}
  \cdot\bmu_\lambda
  \leq
  C
  +\int_{Q_t}
  \begin{pmatrix}
 c_{1}P_\lambda R_{\lambda,1}\\
 c_{3}P_\lambda R_{\lambda,2}
  \end{pmatrix}
  \cdot\Psi_{\lambda,\bph}^{(1)}(\bph_\lambda)
\end{align*}
and
\begin{align*}
  -\int_{Q_t}
  \begin{pmatrix}
  c_{2}J_\lambda^1(\bph_{\lambda})\\
  c_{4}J_\lambda^2(\bph_{\lambda})
  \end{pmatrix}
  \cdot\bmu_\lambda
  \leq
  C
  -\int_{Q_t}
  \begin{pmatrix}
  c_{2}J_\lambda^1(\bph_{\lambda})\\
  c_{4}J_\lambda^2(\bph_{\lambda})
  \end{pmatrix}
  \cdot\Psi_{\lambda,\bph}^{(1)}(\bph_\lambda).
\end{align*}
Note that, thanks to the fact that $h$ takes values in $[0,1]$,
there exists a constant $\eps_0>0$, independent of $\lambda$, such that
\[
  \eps_0
  \begin{pmatrix}
  c_{1}P_\lambda R_{\lambda,1}\\
  c_{3}P_\lambda R_{\lambda,2}
  \end{pmatrix}
  \in\simcl
  \qquad\forall\,\lambda\in(0,1).
\]
Consequently, exploiting the Young inequality
and the dual estimate \eqref{est2:2}, we get
\begin{align*}
  &\int_{Q_t}
  \begin{pmatrix}
  c_{1}P_\lambda R_{\lambda,1}\\
  c_{3}P_\lambda R_{\lambda,2}
  \end{pmatrix}
  \cdot\Psi_{\lambda,\bph}^{(1)}(\bph_\lambda)\\
  &\quad =
  \frac1{\eps_0}\int_{Q_t}
  \eps_0
  \begin{pmatrix}
  c_{1}P_\lambda R_{\lambda,1}\\
  c_{3}P_\lambda R_{\lambda,2}
  \end{pmatrix}
  \cdot\Psi_{\lambda,\bph}^{(1)}(\bph_\lambda)\\
  &\quad \leq\frac1{\eps_0}
  \int_{Q_t}
  \Psi_\lambda^{(1)}\left(
  \eps_0
  \begin{pmatrix}
  c_{1}P_\lambda R_{\lambda,1}\\
  c_{3}P_\lambda R_{\lambda,2}
  \end{pmatrix}
  \right)
  +\frac1{\eps_0}
  \int_{Q_t}
  (\Psi_\lambda^{(1)})^*(\Psi_{\lambda,\bph}^{(1)}(\bph_\lambda))\\
  &\quad \leq\frac1{\eps_0}
  \int_{Q_t}
  \Psi^{(1)}\left(
  \eps_0
  \begin{pmatrix}
  c_{1}P_\lambda R_{\lambda,1}\\
  c_{3}P_\lambda R_{\lambda,2}
  \end{pmatrix}
  \right)
  +\frac1{\eps_0}
  \int_{Q_t}
  (\Psi_\lambda^{(1)})^*(\Psi_{\lambda,\bph}^{(1)}(\bph_\lambda))\\
  &\quad \leq \frac1{\eps_0}|Q|\max_{\rr\in\simcl}\Psi^{(1)}(\rr)
  +\frac{C}{\eps_0}.
\end{align*}
Taking the above chain of inequalities into account, we infer from \eqref{en:est} that
\[
\frac12\int_\Omega|\nabla\bph_\lambda(t)|^2
	+\int_\Omega \Psi_\lambda^{(1)}(\bph_\lambda(t))
	+\int_{Q_t}|\nabla\bmu_\lambda|^2
	+\int_{Q_t}
  \begin{pmatrix}
  c_{2}J_\lambda^1(\bph_{\lambda})\\
  c_{4}J_\lambda^2(\bph_{\lambda})
  \end{pmatrix}
  \cdot\Psi_{\lambda,\bph}^{(1)}(\bph_\lambda)
	\leq C.
\]
The last term on the left-hand side can be treated
exactly in the same way as in Subsection~\ref{ssec:dual}
(see \eqref{est3}). This gives
\[
  \int_{Q_t}
  \begin{pmatrix}
  c_{2}J_\lambda^1(\bph_{\lambda})\\
  c_{4}J_\lambda^2(\bph_{\lambda})
  \end{pmatrix}
  \cdot\Psi_{\lambda,\bph}^{(1)}(\bph_\lambda)
  \geq - |Q|(c_2+c_4)\max_{\rr\in\simcl}\Psi^{(1)}(\rr).
\]
Therefore, we eventually obtain that
\[
\frac12\int_\Omega|\nabla\bph_\lambda(t)|^2
	+\int_\Omega \Psi_\lambda^{(1)}(\bph_\lambda(t))
	+\int_{Q_t}|\nabla\bmu_\lambda|^2
	\leq C \qquad\forall\,t\in[0,T],
\]
that leads us to
\begin{equation}
 \label{est4}
 \norma{\bph_\lambda}_{L^\infty(0,T;\VV)}
 + \norma{\Psi_\lambda^{(1)}(\bph_\lambda)}_{L^\infty(0,T; \LL^1(\Omega))}
 + \norma{\nabla\bmu_\lambda}_{L^2(0,T;\HH)} \leq C.
\end{equation}
Next, comparison in \eqref{SYS:2_app}, estimates \eqref{est1} and \eqref{est3} yield
\begin{equation}
 \label{est5}
 \norma{\bph_\lambda}_{H^1(0,T;\VV^*)} \leq C.
\end{equation}
Furthermore, using the properties of the Moreau--Yosida regularization, we have
\[
  \Psi^{(1)}_\lambda(\rr) =\frac{\norma{\rr - \bj_\lambda(\rr)}^2}{2\lambda}
  + \Psi^{(1)}(\bj_\lambda(\rr))
  \geq\frac{\norma{\rr - \bj_\lambda(\rr)}^2}{2\lambda}
  =\frac\lambda2\norma{\Psi_{\lambda,\bph}^{(1)}(\rr)}^2
  \qquad\forall\,\rr\in\erre^2.
\]
Thus, it follows from \eqref{est4} that
\begin{equation}\label{est6}
  \lambda^{1/2}\norma{\Psi_{\lambda,\bph}(\bph_\lambda)}_{L^\infty(0,T;\HH)}\leq C
\end{equation}
which improves \eqref{est3}.

\subsection{Mean value properties of the phase field variable}\label{SEC:MEAN}
In this subsection, we investigate the behavior of the mean value of
$\bph_\lambda$. Recall that $\bph_\lambda(0)=\0$, so
in particular $(\ph_{\lambda,i})_\Omega(0)=0$ for $i=1,2$.
Here we show that $(\ph_{\lambda,i})_\Omega$ becomes positive for positive times, thanks to the presence of the source. In other words, the formation of complexes is instantaneous.
In particular, the pair $((\ph_{\lambda,1})_\Omega, (\ph_{\lambda,2})_\Omega)$ remains confined strictly inside the simplex $\simap$
as soon as the evolution starts. These facts will be essential in order to find an integral bound on
the chemical potential in the next section and, eventually, pass to the limit as $\lambda$ vanishes.

This strategy is possible due to an ODE argument which can be derived by considering equation \eqref{SYS:1_app}.
However, the challenge relies upon the fact that this analysis has to be performed at the level of the $\lambda$-approximating problem.
In fact, as $\Sph^\lambda(\bph_\lambda, P_\lambda, \br_\lambda)$ contains truncations (cf. \eqref{def:sorgenti_app}), some technical difficulties arise since the mean value and the truncation does not commute.
The key idea to overcome this difficulty relies first on adding and subtracting a suitable term
that allows to solve the ODE and, second, on noticing that some of the forcing terms
are dominant with respect to others as soon as $\lambda$ is chosen small enough.

We begin to observe that assumption \eqref{max:ini:Linf} ensures that
\begin{equation}\label{alpha0}
  \alpha_0:=(P_0)_\Omega \in (0,1).
\end{equation}
In addition, \eqref{max:ini:Linf} and \eqref{smallness}
imply that $P_0,R_{0,1}, R_{0,2}\geq 2c_*$ almost everywhere in $\Omega$, and that
$\alpha_0>c_*$ and $\frac12(1-\alpha_0)>c_*$.
It holds in particular that $1>\min\{\alpha_0-c_*, 1-\alpha_0-2c_*\}>0$.

Our results are collected in the following lemma.
\begin{lemma}
  \label{lem2}
  The following inequalities hold
  \begin{align*}
  &(\ph_{\lambda,1}(t))_\Omega + (\ph_{\lambda,2}(t))_\Omega \leq
  {\min\{\alpha_0-c_*, 1-\alpha_0-2c_*\}}<1 \qquad\forall\,t\in[0,T],\\
  &(\ph_{\lambda,i}(t))_\Omega \leq
  \frac{1-\alpha_0}2 -c_*<1 \quad\forall\,t\in[0,T],\qquad\text{for $i=1,2$},
  \end{align*}
  where $c_*$ is the threshold constant introduced in Lemma \ref{lem1}.
  Moreover, there exist $\lambda_0\in(0,1)$ and $c_0>0$, independent of $\lambda$,
  such that, for every $\lambda\in(0,\lambda_0)$ and $i=1,2$, it holds
  \[
  c_0(1 - e^{-c_{2i}t}) \leq
  (\ph_{\lambda,i}(t))_\Omega \leq c_{2i-1}t
  \qquad\forall\,t\in[0,T].
  \]
\end{lemma}

\begin{proof}[Proof of Lemma~\ref{lem2}]
We first recall that $(\bph_{\lambda}(0))_\Omega=\0$
and we start by showing first the upper bounds for the mean values of $\bph_\lambda$.
Testing \eqref{SYS:1_app} by $\1$, \eqref{SYS:3_app} by $1$,
and adding the resulting identites together, we get
\begin{align*}
	\dt \big( (\ph_{\lambda,1})_\Omega
	+(\ph_{\lambda,2})_\Omega
	+ (P_\lambda)_\Omega \big) = 0.
\end{align*}
Integrating in time, recalling \eqref{sep_PR} and the initial conditions \eqref{SYS:7_app}, entails
\begin{align*}
	(\ph_{\lambda,1}(t))_\Omega
	+(\ph_{\lambda,2}(t))_\Omega
	= (P_0)_\Omega -   (P_\lambda(t))_\Omega \leq \alpha_0 - c_*
	\qquad\forall\,t\in[0,T].
\end{align*}
Arguing similarly, namely, testing \eqref{SYS:1_app} by $\1$, \eqref{SYS:4_app} by $\1$, we deduce
\begin{align*}
	\dt \big( (\ph_{\lambda,1})_\Omega
	+(\ph_{\lambda,2})_\Omega
	+ (R_{\lambda,1})_\Omega
	+ (R_{\lambda,2})_\Omega \big) = 0,
\end{align*}
from which, thanks to \eqref{SYS:7_app} and \eqref{sep_PR}, we obtain
\begin{align*}
	(\ph_{\lambda,1}(t))_\Omega
	+(\ph_{\lambda,2}(t))_\Omega
	&= 1-(P_0)_\Omega -   (R_{\lambda,1}(t))_\Omega
	-(R_{\lambda,2}(t))_\Omega\\
	&\leq 1-\alpha_0 - {2c_*}
	\qquad\forall\,t\in[0,T].
\end{align*}
Testing then \eqref{SYS:1_app} and \eqref{SYS:4_app} by $(1,0)$
(and by (0,1), respectively), we find
\begin{align*}
	\dt \big( (\ph_{\lambda,i})_\Omega
	+ (R_{\lambda,i})_\Omega \big) = 0, \qquad i=1,2,
\end{align*}
so that
\begin{align*}
	(\ph_{\lambda,i}(t))_\Omega
	= \frac{1-(P_0)_\Omega}2 -   (R_{\lambda,i}(t))_\Omega
	 \leq \frac{1-\alpha_0}2 - c_*
	\qquad\forall\,t\in[0,T].
\end{align*}
Therefore, by combining the above inequalities, we obtain the bounds from above.

Let us focus now on the bounds from below.
From the two components of equation \eqref{SYS:1_app},
testing by the constant $1$ we can derive
a Cauchy problem for the mean value of
$\ph_{\lambda,i}$, for $i=1,2$.
More precisely, setting
\[
y_{\lambda,i}:= (\ph_{\lambda,i})_\Omega, \qquad
f_{\lambda,i}:=(P_\lambda R_{\lambda,i})_\Omega,
\qquad i=1,2,
\]
where we also used \eqref{sep_PR},
we deduce the Cauchy problem
\begin{align}
\label{eq_mean}
\begin{cases}
	y_{\lambda,i}' + c_{2i}(J_\lambda^i(\bph_{\lambda}))_\Omega =
	c_{2i-1}f_{\lambda,i}
	\\
	y_{\lambda,i}(0)=0
\end{cases}
\qquad i=1,2.
\end{align}
Exploiting the nonnegativity of $J_\lambda^i$, from \eqref{eq_mean}
one readily gets that
\[
  y_{\lambda,i}(t) \leq c_{2i-1}\int_0^tf_{\lambda,i}(s)\,\ds\leq c_{2i-1}t.
\]
On the other hand, the former Cauchy problem can be rewritten as follows
\begin{align*}
\begin{cases}
	y_{\lambda,i}' + c_{2i}y_{\lambda,i} =
	c_{2i-1}f_{\lambda,i} +
	c_{2i}(\ph_{\lambda,i}-J_\lambda^i(\bph_{\lambda}))_\Omega
	\\
	y_{\lambda,i}(0)=0
\end{cases}
\qquad i=1,2.
\end{align*}
Let $i\in\{1,2\}$ be fixed.
Then the variation of constants formula yields,
together with \eqref{sep_PR} and the H\"older inequality, the following
\begin{align*}
  y_{\lambda,i}(t)&\geq
  c_{2i-1}\int_0^te^{-c_{2i}(t-s)}f_{\lambda,i}(s)\,\ds
  +c_{2i}\int_0^te^{-c_{2i}(t-s)}
  (\ph_{\lambda,i}(s)-J_\lambda^i(\bph_{\lambda}(s)))_\Omega\,\ds\\
  &\geq c_{2i-1}c_*^2\int_0^te^{-c_{2i}(t-s)}\,\ds
  -\frac{c_{2i}}{|\Omega|^{1/2}}\int_0^te^{-c_{2i}(t-s)}
  \norma{\ph_{\lambda,i}(s)-J_\lambda^i(\bph_{\lambda}(s))}\,\ds.
\end{align*}
Note that, recalling the definition of Yosida approximation
and classical results of convex analysis, we have
\[
  \norma{\ph_{\lambda,i}-J_\lambda^i(\bph_{\lambda})}
  \leq\norma{\bph_{\lambda}-\bj_\lambda(\bph_{\lambda})}
  =\lambda\norma{\Psi_{\lambda,\bph}^{(1)}(\bph_\lambda)},
  \quad
  i=1,2.
\]
Thus, by the estimate \eqref{est6}, we get
\begin{align*}
  y_{\lambda,i}(t)&\geq
  c_{2i-1}c_*^2
  \int_0^te^{-c_{2i}(t-s)}\,\ds
  -\frac{c_{2i}}{|\Omega|^{1/2}}\lambda \int_0^te^{-c_{2i}(t-s)}
  \norma{\Psi_{\lambda,\bph}^{(1)}(\bph_\lambda(s))}\,\ds\\
  &\geq
  c_{2i-1}c_*^2
  \int_0^te^{-c_{2i}(t-s)}\,\ds
  -\frac{c_{2i}}{|\Omega|^{1/2}}C\lambda^{1/2} \int_0^te^{-c_{2i}(t-s)}\,\ds\\
  &=
  \left[\frac{c_{2i-1}}{c_{2i}}c_*^2
  -\frac{C\lambda^{1/2}}{|\Omega|^{1/2}}\right]
  (1-e^{-c_{2i}t}).
\end{align*}
Hence, there exists $\lambda_0\in(0,1)$ and a constant $c_0>0$ independent of $\lambda$ such that
\[
   y_{\lambda,i}(t) \geq c_0(1-e^{-c_{2i}t}) \qquad\forall\,t\in[0,T],
   \quad\forall\,\lambda\in(0,\lambda_0).
\]
This concludes the proof.
\end{proof}

\subsection{Estimating the chemical potential}\label{sec:chem}
This subsection is devoted to establishing a uniform bound on the chemical potential $\bmu$. We recall that we only have a bound for its gradient but this
is not enough to recover the full $\L2 \VV$-norm because of the no-flux boundary conditions. As we already pointed out, we cannot argue in the usual way because of the null initial condition
for $\bph_\lambda$. The main idea is to divide the analysis into two cases:
small times (i.e., in a right neighborhood of $t=0$) and large times (i.e., outside a right neighborhood of $t=0$).

\noindent
{\sc Large times estimate.}
By virtue of Lemma~\ref{lem2}, for every $\sigma\in(0,T)$ there exists a
constant $c_\sigma>0$, independent of $\lambda$,
such that, for all $t\in[\sigma,T],$
\begin{align}
  & 0<c_\sigma \leq \min \{(\varphi_{\lambda,1}(t))_\Omega, (\varphi_{\lambda,2}(t))_\Omega\}
  <(\varphi_{\lambda,1}(t))_\Omega + (\varphi_{\lambda,2}(t))_\Omega
  \leq 1-\alpha_0.
  \label{est_largetimes}
\end{align}
Then, we deduce that there exists a compact subset $K_\sigma\subset\simap$, independent of $\lambda$,
such that $(\bph_\lambda(t))_\Omega\in K_\sigma$ for every $t\in[\sigma,T]$. Moreover, thanks to
the definition of resolvent and Yosida approximation combined with \eqref{est6}, for $i=1,2$, we get
\begin{equation}\label{diff_mean}
  |(\varphi_{\lambda,i})_\Omega - (J_\lambda^i(\bph_\lambda))_\Omega|
  \leq |\Omega|^{1/2}\|\bph_\lambda - \bj_\lambda(\bph_\lambda)\|
  \leq |\Omega|^{1/2} \lambda\|\Psi^{(1)}_{\lambda,\bph}(\bph_\lambda)\|
  \leq C\lambda^{1/2},
\end{equation}
where $C$ is independent of $\lambda$.
Consequently, there exists $\lambda_0\in(0,1)$ sufficiently small
and a compact subset $\tilde K_\sigma\subset \simap$, independent of $\lambda$,
such that  $K_\sigma\subset \tilde K_\sigma\subset \simap$ and
$(\bj_\lambda(\bph_\lambda(t)))_\Omega \in\tilde K_\sigma$
for every $\lambda\in(0,\lambda_0)$ and $t\in[\sigma,T]$.
Hence, exploiting Property~\ref{MZ0}
we infer the existence of constants $c_{\tilde K_\sigma}, C_{\tilde K_\sigma}>0$
independent of $\lambda$, such that
\[
  c_{\tilde K_\sigma}\int_\Omega|\Psi^{(1)}_\bph(\bj_\lambda(\bph_\lambda(t)))| \leq
  \int_\Omega\Psi^{(1)}_\bph(\bj_\lambda(\bph_\lambda(t)))\cdot
  (\bj_\lambda(\bph_\lambda(t)) - (\bj_\lambda(\bph_\lambda(t)))_\Omega) + C_{\tilde K_\sigma}
\]
for every $t\in[\sigma,T]$. By the properties of the Yosida approximation, this amounts to writing
\[
  c_{\tilde K_\sigma}\int_\Omega|\Psi^{(1)}_{\lambda,\bph}(\bph_\lambda(t))| \leq
  \int_\Omega\Psi^{(1)}_{\lambda,\bph}(\bph_\lambda(t))\cdot
  (\bj_\lambda(\bph_\lambda(t)) - (\bj_\lambda(\bph_\lambda(t)))_\Omega) + C_{\tilde K_\sigma}{.}
\]
Now, comparing this inequality with equation \eqref{SYS:2_app} tested by
$\bj_\lambda(\bph_\lambda) - (\bj_\lambda(\bph_\lambda))_\Omega$, one deduces that
\begin{align*}
  &c_{\tilde K_\sigma}\int_\Omega|\Psi^{(1)}_{\lambda,\bph}(\bph_\lambda)|
  +\int_\Omega(-\Delta\bph_\lambda)\cdot \bj_\lambda(\bph_\lambda)\\
  &\quad \leq C_{\tilde K_\sigma} +
  \int_\Omega \bmu_\lambda\cdot (\bj_\lambda(\bph_\lambda) - (\bj_\lambda(\bph_\lambda))_\Omega)
  -\int_\Omega \Psi^{(2)}_{\bph}(\bph_\lambda)\cdot (\bj_\lambda(\bph_\lambda) - (\bj_\lambda(\bph_\lambda))_\Omega).
\end{align*}
Then, concerning the second term on the \lhs, recalling the identity $\bj_\lambda(\bph_\lambda) = \bph_\lambda - \lambda \Psi_\bph^{(1)}(\bph_\lambda)$, we have that
\begin{align*}
 \int_\Omega(-\Delta\bph_\lambda)\cdot \bj_\lambda(\bph_\lambda)
 &
 = \int_\Omega|\nabla\bph_\lambda|^2
  -\lambda\int_\Omega(-\Delta\bph_\lambda)\cdot \Psi^{(1)}_{\lambda,\bph}(\bph_\lambda),
\end{align*}
where the first term on the \rhs\ is nonnegative, whereas the second one has to be estimated.
Furthermore, combining the preliminary estimate \eqref{est1} with \eqref{est6}, we infer that
\begin{align}\label{rev:J}
	-\lambda\int_\Omega(-\Delta\bph_\lambda)\cdot \Psi^{(1)}_{\lambda,\bph}(\bph_\lambda)
	\leq
	C \lambda^{1/2} \norma{\Delta \bph_\lambda}
	\leq C.
\end{align}
Taking this into account, exploiting the Poincar\'e--Wirtinger inequality,
the Lipschitz continuity of $\Psi^{(2)}_\bph$,
and the nonexpansivity of $\bj_\lambda$ and \eqref{rev:J}, we deduce that
\begin{align}
  \nonumber
  &c_{\tilde K_\sigma}\int_\Omega|\Psi^{(1)}_{\lambda,\bph}(\bph_\lambda)|\\
  \nonumber
  &\quad\leq C_{\tilde K_\sigma} +
  \int_\Omega \bmu_\lambda\cdot (\bj_\lambda(\bph_\lambda) - (\bj_\lambda(\bph_\lambda))_\Omega)
  -\int_\Omega \Psi^{(2)}_{\bph}(\bph_\lambda)\cdot
  (\bj_\lambda(\bph_\lambda) - (\bj_\lambda(\bph_\lambda))_\Omega)
  \\ & \qquad \non
   -\lambda\int_\Omega \Delta\bph_\lambda \cdot \Psi^{(1)}_{\lambda,\bph}(\bph_\lambda)
  \\
  \nonumber
  &\quad= C_{\tilde K_\sigma} +
  \int_\Omega (\bmu_\lambda-(\bmu_\lambda)_\Omega{)}\cdot (\bj_\lambda(\bph_\lambda) - (\bj_\lambda(\bph_\lambda))_\Omega)
  \\ \non & \qquad
  -\int_\Omega \Psi^{(2)}_{\bph}(\bph_\lambda)\cdot
  (\bj_\lambda(\bph_\lambda) - (\bj_\lambda(\bph_\lambda))_\Omega)
  +C \|\Delta\bph_\lambda\|
  \\
  \label{est7_prequel}
  &\quad\leq C_{\tilde K_\sigma} + C\|\nabla\mu_\lambda\|(1+\|\bph_\lambda\|)
  +C(1+\|\bph_\lambda\|^2 +\|\Delta \bph_\lambda\|).
\end{align}
Thus, estimate \eqref{est1} and \eqref{est4} give
\begin{equation}\label{est7}
  \forall\,\sigma\in(0,T),\quad\exists\,C_\sigma>0:\quad
  \norma{\Psi_{\lambda,\bph}(\bph_\lambda)}_{L^2(\sigma,T; \LL^1(\Omega))} \leq C_\sigma,
\end{equation}
where the constants $\{C_\sigma\}_{\sigma\in(0,T)}$ are independent of $\lambda$.

\noindent
{\sc Small times estimate.}
Let us now study the behavior close to the initial time.
By virtue of Lemma~\ref{lem2}, there exists $T_0\in(0,T)$,
independent of $\lambda$, and a constant $c_0'\in(0,1)$ such that, for all $t\in(0,T_0)$,
\begin{equation}\label{est_smalltimes}
  0<c_0't\leq \min\{(\varphi_{\lambda,1}(t))_\Omega, (\varphi_{\lambda,2}(t))_\Omega\}
  <(\varphi_{\lambda,1}(t))_\Omega+ (\varphi_{\lambda,2}(t))_\Omega \leq
  (c_1+c_3)t \leq \frac{R}2,
\end{equation}
where $R$ is given by Property~\ref{MZ}.
Taking \eqref{diff_mean} into account, we get then, for all $t\in(0,T_0)$ and $\lambda\in(0,\lambda_0)$,
\begin{align*}
  c_0't - C\lambda^{1/2}&\leq \min\{(J_\lambda^1(\bph_\lambda(t)))_\Omega,
  (J_\lambda^2(\bph_\lambda(t)))_\Omega\}\\
  &\leq(J_\lambda^1(\bph_\lambda(t)))_\Omega +
  (J_\lambda^2(\bph_\lambda(t)))_\Omega \\
  &\leq
  (c_1+c_3)t + 2C\lambda^{1/2} \leq \frac R2 + 2C\lambda^{1/2}.
\end{align*}
Now, since $C$ is independent of $\lambda$, without loss of generality
we can suppose that $\lambda_0$ satisfies the inequality
\[
\lambda_0^{1/2}<\frac1{2C}\min\left\{\frac R2,c_0' T_0\right\}.
\]
With this choice, for every $\lambda\in(0,\lambda_0)$, it holds that
\begin{align}
	\label{Tlam}
  T_\lambda:= \frac{2C}{c_0'}\lambda^{1/2} \in (0,T_0), \qquad
  \frac R2 + 2C\lambda^{1/2} < R.
\end{align}
Let us incidentally observe that $ T_\lambda \to 0$ as $\lambda \to 0$.
Consequently, we deduce for every $\lambda\in(0,\lambda_0)$
and for every $t\in(T_\lambda, T_0)$  that
\begin{align*}
  0<c_0't - C\lambda^{1/2}&\leq \min\{(J_\lambda^1(\bph_\lambda(t)))_\Omega,
  (J_\lambda^2(\bph_\lambda(t)))_\Omega\}\\
  &\leq(J_\lambda^1(\bph_\lambda(t)))_\Omega +
  (J_\lambda^2(\bph_\lambda(t)))_\Omega \\
  &\leq
  (c_1+c_3)t + 2C\lambda^{1/2}  {<} R.
\end{align*}
Hence, we can exploit Property~\ref{MZ} with the choice $\bphi=\bj_\lambda(\bph_\lambda(t))$
for every $t\in(T_\lambda,T_0)$. Taking into account the monotonicity of $F$ we infer that,
for every $t\in(T_\lambda,T_0)$,
\begin{align*}
  c_\Psi(c_0't - C\lambda^{1/2})\int_\Omega|\Psi^{(1)}_\bph(\bj_\lambda(\bph_\lambda(t)))|
  &\leq \int_\Omega\Psi^{(1)}_\bph(\bj_\lambda(\bph_\lambda(t)))\cdot
  (\bj_\lambda(\bph_\lambda(t)) - (\bj_\lambda(\bph_\lambda(t)))_\Omega)\\
  &\quad +C_\Psi[(c_1+c_3)t + 2C\lambda^{1/2}]\left(1+ F(c_0't - C\lambda^{1/2})\right)
\end{align*}
and, by the properties of the Yosida approximation this yields, for all $t\in(T_\lambda, T_0)$,
\begin{align*}
  c_\Psi(c_0't - C\lambda^{1/2})\int_\Omega|\Psi^{(1)}_{\lambda,\bph}(\bph_\lambda(t))|
  &\leq \int_\Omega\Psi^{(1)}_{\lambda,\bph}(\bph_\lambda(t))\cdot
  (\bj_\lambda(\bph_\lambda(t)) - (\bj_\lambda(\bph_\lambda(t)))_\Omega)\\
  &\quad +C_\Psi[(c_1+c_3)t + 2C\lambda^{1/2}]\left(1+ F(c_0't - C\lambda^{1/2})\right).
\end{align*}
Arguing now as in the former case, we can compare this inequality with
equation \eqref{SYS:2_app} tested by
$\bj_\lambda(\bph_\lambda) - (\bj_\lambda(\bph_\lambda))_\Omega$.
We can now handle the Laplacian arising from the first term on the left-hand side as in \eqref{rev:J} but keeping the dependence on $\lambda$ explicit. This gives
\begin{align*}
  & c_\Psi(c_0't - C\lambda^{1/2})\int_\Omega|\Psi^{(1)}_{\lambda,\bph}(\bph_\lambda(t))|
  \leq C_\Psi[(c_1+c_3)t + 2C\lambda^{1/2}]\left(1+ F(c_0't - C\lambda^{1/2})\right)\\
  &\quad +\int_\Omega (\bmu_\lambda(t)-(\bmu_\lambda(t))_\Omega {)}\cdot (\bj_\lambda(\bph_\lambda(t)) - (\bj_\lambda(\bph_\lambda(t)))_\Omega)\\
  &\quad -\int_\Omega \Psi^{(2)}_{\bph}(\bph_\lambda(t))\cdot
  (\bj_\lambda(\bph_\lambda(t)) - (\bj_\lambda(\bph_\lambda(t)))_\Omega)
  	+C \lambda^{1/2} \norma{\Delta \bph_\lambda(t)}
\end{align*}
for almost every $t\in(T_\lambda,T_0)$.
We now exploit
the Poincar\'e--Wirtinger inequality,
the Lipschitz continuity of $\Psi^{(2)}_\bph$, \eqref{est1}, and
\eqref{diff_mean},
to infer that, for almost every $t\in(T_\lambda,T_0)$ and for every $\alpha\in(0,1)$,
\begin{align*}
  c_\Psi(c_0't - C\lambda^{1/2})\int_\Omega|\Psi^{(1)}_{\lambda,\bph}(\bph_\lambda(t))|
  &\leq C_\Psi[(c_1+c_3)t + 2C\lambda^{1/2}]\left(1+ F(c_0't - C\lambda^{1/2})\right)\\
  &\quad
  +C_\alpha\left(1+\|\nabla\bmu_\lambda(t)\| +\|\bph_\lambda(t)\|_{\WW}\right)
  \left(t^\alpha + \lambda^{1/2}\right).
\end{align*}
Equivalently, this can be written as follows
\begin{align*}
  c_\Psi \int_\Omega|\Psi^{(1)}_{\lambda,\bph}(\bph_\lambda(t))|
  &\leq C_\Psi \frac{(c_1+c_3)t + 2C\lambda^{1/2}}{c_0't - C\lambda^{1/2}}\left(1+ F(c_0't - C\lambda^{1/2})\right)\\
  &\quad +C_\alpha\frac{t^\alpha + \lambda^{1/2}}{c_0't - C\lambda^{1/2}}
  \left(1+\|\nabla\bmu_\lambda(t)\| +\|\bph_\lambda(t)\|_{\WW}\right)
\end{align*}
for almost every $t\in(T_\lambda, T_0)${.}
Recalling the definition of $T_\lambda$ in \eqref{Tlam}, for every $t\in(T_\lambda,T_0)$ we have that
$c_0't - C\lambda^{1/2} \geq C\lambda^{1/2}$, so that
\[
  \sup_{t\in(T_\lambda,T_0)}\frac{\lambda^{1/2}}{c_0't - C\lambda^{1/2}} \leq \frac1C.
\]
Consequently, for $t\in(T_\lambda,T_0)$ we also have that
\[
  \frac{t}{c_0't - C\lambda^{1/2}} =
  \frac1{c_0'} + \frac{C\lambda^{1/2}}{c_0'(c_0't - C\lambda^{1/2})} \leq
  \frac2{c_0'} \qquad \forall\,t\in(T_\lambda,T_0).
\]
Similarly, noting that $C\lambda^{1/2}\leq \frac{c_0'}{2}t$ for all $t\in(T_\lambda,T_0)$,
we deduce that
\[
  \frac{t^\alpha}{c_0't - C\lambda^{1/2}}\leq \frac2{c_0'}\frac1{t^{1-\alpha}} \qquad \forall\,t\in(T_\lambda,T_0).
\]
Therefore, updating the values of the constants
independent of $\lambda$, it follows that, for almost every $t\in(T_\lambda,T_0)$,
\begin{align}
  \nonumber
  \int_\Omega|\Psi^{(1)}_{\lambda,\bph}(\bph_\lambda(t))|
  &\leq C\left(1+ F(c_0't - C\lambda^{1/2})\right)\\
  \label{est8_prequel}
  &\quad +C_\alpha\left(1+\frac{1}{t^{1-\alpha}}\right)
  \left(1+\|\nabla\bmu_\lambda(t)\| +\|\bph_\lambda(t)\|_{\WW}\right).
\end{align}
Our choices of $T_0$, $\lambda_0$, and $T_\lambda$
imply that $c_0'T_\lambda - C\lambda^{1/2}>0$ and $c_0'T_0 - C\lambda^{1/2}<R$. Thus,
by substitution, one has that
\begin{align*}
  \int_{T_\lambda}^{T_0}|F(c_0's - C\lambda^{1/2})|^q\,\ds
  =\frac1{c_0'}\int_{c_0'T_\lambda - C\lambda^{1/2}}^{c_0'T_0 - C\lambda^{1/2}}
  |F(s)|^q\,\ds\leq \norma{F}_{L^q(0,R)}^q.
\end{align*}
We thus infer that there exists a constant $C>0$, independent of $\lambda$, such that
\[
  \norma{t\mapsto F(c_0't-C\lambda^{1/2})}_{L^q(T_\lambda, T_0)} \leq C.
\]
Moreover, we note that
for every $\ell\in(1,+\infty)$ there exists $\alpha\in(0,1)$ such that
$(1-\alpha)\ell<1$: hence, there exists also a constant $C_\ell>0$, independent of $\lambda$, such that
\[
  \norma{t\mapsto \frac{1}{t^{1-\alpha}}}_{L^\ell(T_\lambda,T_0)}\leq C_\ell.
\]
Therefore, recalling \eqref{est1} and \eqref{est4}, by the H\"older inequality we get
\begin{equation}\label{est8}
  \forall\,p\in(1,2), \quad\exists\,C_p>0:\quad
  \norma{\Psi_{\lambda,\bph}^{(1)}(\bph_\lambda)}_{L^p(T_\lambda,T_0; \LL^1(\Omega)) }\leq C_p,
\end{equation}
where the constants $\{C_p\}_{p\in(1,2)}$ are independent of $\lambda$.
Equivalently, letting
\[
  \chi_\lambda:[0,T]\to[0,1]\,, \qquad
  \chi_\lambda(t):=
  \begin{cases}
  0 \quad&\text{if } t\in[0,T_\lambda],\\
  1 \quad&\text{if } t\in(T_\lambda,T],
  \end{cases}
\]
we have that
\begin{equation}\label{est9}
  \forall\,p\in(1,2), \quad\exists\,C_p>0:\quad
  \norma{\chi_\lambda\Psi_{\lambda,\bph}^{(1)}(\bph_\lambda)}_{L^p(0,T_0;\LL^1(\Omega))} \leq C_p.
\end{equation}

\noindent
{\sc Global estimates.}
It is now immediate, taking \eqref{est7} and \eqref{est9} into account,
to infer, by comparison in equation \eqref{SYS:2_app}, that
\begin{align}
  \label{est10}
  \forall\,\sigma\in(0,T),\quad\exists\,C_\sigma>0:\quad
  &\norma{(\bmu_\lambda)_\Omega}_{L^2(\sigma,T)} \leq C_\sigma,\\
  \label{est11}
  \forall\,p\in(1,2), \quad\exists\,C_p>0:\quad
  &\norma{\chi_\lambda (\bmu_\lambda)_\Omega}_{L^p(0,T_0)} \leq C_p.
\end{align}
The above estimates combined with \eqref{est4} imply, on account of the Poincar\'e--Wirtinger's inequality, that
\begin{align}
  \label{est12}
  \forall\,\sigma\in(0,T),\quad\exists\,C_\sigma>0:\quad
  &\norma{\bmu_\lambda}_{L^2(\sigma,T; \VV)} \leq C_\sigma,\\
  \label{est13}
  \forall\,p\in(1,2), \quad\exists\,C_p>0:\quad
  &\norma{\chi_\lambda \bmu_\lambda}_{L^p(0,T_0; \VV)} \leq C_p,
\end{align}
which in turn yield, by comparison in \eqref{SYS:2_app} and by estimate \eqref{est1},
\begin{align}
  \label{est14}
  \forall\,\sigma\in(0,T),\quad\exists\,C_\sigma>0:\quad
  &\norma{\Psi_{\lambda,\bph}^{(1)}(\bph_\lambda)}_{L^2(\sigma,T; \HH)} \leq C_\sigma,\\
  \label{est15}
  \forall\,p\in(1,2), \quad\exists\,C_p>0:\quad
  &\norma{\chi_\lambda\Psi_{\lambda,\bph}^{(1)}(\bph_\lambda)}_{L^p(0,T_0; \HH)} \leq C_p{.}
\end{align}
Eventually, arguing along the lines of \cite[Lemmas 7.3, and 7.4]{GGW} and using interpolation, we can
regard \eqref{SYS:2_app} as a one-parameter family of time-dependent systems of two elliptic equations with maximal monotone nonlinearity and infer that
\begin{align}
\label{est16}
\forall\,\sigma\in(0,T),\quad\exists\,C_\sigma>0:\quad
  &\norma{\bph_\lambda}_{L^4(\sigma,T;\HH^2(\Omega)) \cap L^2(\sigma,T; \WW^{2,r}(\Omega){)}}
	\leq C_\sigma,\\
\label{est17}
  \forall\,p\in(1,2), \quad\exists\,C_p>0:\quad
  &\norma{\chi_\lambda\bph_\lambda}_{L^{2p}(0,T_0;\HH^2(\Omega)) \cap L^p(0,T_0; \WW^{2,r}(\Omega){)}}
	\leq C_p,
\end{align}
where $r=6$ if $d=3$ and $r$ is arbitrary in $(1,+\infty)$ if $d=2$.

\subsection{Passing to the limit}
\label{SEC:PASSTOLIMIT}

We are now ready to take the final step, namely, letting $\lambda$ go to $0$. More precisely, $\lambda \to 0$ must be
understood in the sequel as ``$\lambda \to 0$ along a suitable subsequence''.

Taking the estimates \eqref{est1}{--}\eqref{est2} and \eqref{est4}{--}\eqref{est5} into account,
recalling Lemma~\ref{lem1}, and
exploiting the classical Aubin--Lions compactness theorems (see, e.g., \cite{Simon}),
we infer the existence of
\begin{align*}
  \bph&\in H^1(0,T; \VV^*)\cap L^\infty(0,T;\VV)\cap L^2(0,T;\WW),\\
  P&\in H^1(0,T;V^*)\cap L^2(0,T; V)\cap L^\infty(Q),\\
  \br&\in H^1(0,T;\VV^*)\cap L^2(0,T; \VV)\cap \LL^\infty(Q),
\end{align*}
such that
\[
  0<c_*\leq P\leq 1, \quad 0<c_*\leq R_i\leq 1 \quad\text{a.e.~in } Q, \quad i=1,2,
\]
and, as $\lambda\to0$,
\begin{align}
  \bph_\lambda\wstarto\bph \quad&\text{in } H^1(0,T; \VV^*)\cap L^\infty(0,T;\VV)\cap L^2(0,T;\WW),\\
  \label{strong1}
  \bph_\lambda\to\bph\quad&\text{in } C^0([0,T];\HH)\cap L^2(0,T;\VV),\\
  P_\lambda\wstarto P \quad&\text{in } H^1(0,T; V^*)\cap L^2(0,T;V)\cap L^\infty(Q),\\
  \label{strong2}
  P_\lambda\to P \quad&\text{in } C^0([0,T]; V^*)\cap L^2(0,T;H),\\
  \br_\lambda\wstarto \br \quad&\text{in } H^1(0,T; \VV^*)\cap L^2(0,T;\VV)\cap \LL^\infty(Q),\\
  \label{strong3}
  \br_\lambda\to \br \quad&\text{in } C^0([0,T]; \VV^*)\cap L^2(0,T;\HH).
\end{align}
It is clear from the strong convergences \eqref{strong1}, \eqref{strong2}, and \eqref{strong3}  that
the following initial conditions are satisfied
\[
\bph(0)=\0, \quad P(0)=P_0,  \quad \br(0)= \Big( \frac {1-P_0}2\Big)\1
\]
in $\VV$, $H$, and $\HH$, respectively. Moreover, \eqref{strong1} and
the Lipschitz continuity of $\Psi^{(2)}_\bph$ readily imply that
\[
  \Psi^{(2)}_\bph(\bph_\lambda) \to \Psi^{(2)}_\bph(\bph) \quad\text{in } C^0([0,T]; \HH).
\]
Also, \eqref{strong1} and \eqref{est6}
yield, by definition of the Yosida approximation, that
\[
  \bj_\lambda(\bph_\lambda)\to\bph \quad\text{in } C^0([0,T]; \HH).
\]
Consequently, it follows from the definitions {\eqref{def:sorgenti_app}}--\eqref{def:Sp_app}
and the dominated convergence theorem that, as $\lambda\to 0$,
\begin{alignat*}{2}
  \Sph^\lambda(\bph_\lambda, P_\lambda,\br_\lambda)&  \to \Sph(\bph, P,\br)
  \qquad&&\text{in } L^2(0,T; \HH),\\
  \SP^\lambda(\bph_\lambda, P_\lambda,\br_\lambda)& \to \SP(\bph, P,\br)
  \qquad&&\text{in } L^2(0,T; H),\\
  \SR^\lambda(\bph_\lambda, P_\lambda,\br_\lambda)& \to \SR(\bph, P,\br)
  \qquad&&\text{in } L^2(0,T; \HH).
\end{alignat*}
Furthermore, noting that $\chi_\lambda\to1$ in $L^s(0,T)$ for every $s\in(1,+\infty)$, as $\lambda \to 0$,
we infer from \eqref{est4} and \eqref{est12}--\eqref{est13} that there exists
\[
  \bmu\in \bigcap_{p\in(1,2), \,\sigma\in(0,T)} L^p(0,T_0;\VV) \cap L^2(\sigma,T; \VV),
  \qquad\nabla\bmu\in L^2(0,T;\HH),
\]
and $T_0\in(0,T)$, independent of $\lambda$, such that, as $\lambda\to0$,
\begin{alignat*}{2}
  \chi_\lambda\bmu_\lambda & \wto\bmu \quad&&\text{in } L^p(0,T_0;\VV)\qquad\forall\,p\in(1,2),\\
  \chi_\lambda\bmu_\lambda & \wto \bmu \quad&&\text{in } L^2(\sigma,T;\VV)\qquad\forall\,\sigma\in(0,T),\\
  \nabla\bmu_\lambda &\wto \nabla\bmu \quad&&\text{in } L^2(0,T;\HH).
\end{alignat*}
Similarly, \eqref{est14}--\eqref{est15} yield the existence of
\[
  \boldsymbol{\xi}\in \bigcap_{p\in(1,2), \,\sigma\in(0,T)} L^p(0,T;\HH) \cap L^2(\sigma,T; \HH)
\]
such that, as $\lambda\to0$,
\begin{align*}
  \chi_\lambda\Psi^{(1)}_{\lambda,\bph}(\bph_\lambda)\wto\boldsymbol{\xi}
  \quad&\text{in } L^p(0,T_0;\HH)\qquad\forall\,p\in(1,2),\\
  \chi_\lambda\Psi^{(1)}_{\lambda,\bph}(\bph_\lambda)\wto\boldsymbol{\xi}
   \quad&\text{in } L^2(\sigma,T;\HH)\qquad\forall\,\sigma\in(0,T).
\end{align*}
By continuity of $\Psi^{(1)}_\bph$ and the fact that
$\bj_\lambda(\bph_\lambda)\to\bph$ almost everywhere in $Q$
and $\chi_\lambda\to1$ almost everywhere in $(0,T)$, we find
\[
  \chi_\lambda\Psi^{(1)}_{\lambda,\bph}(\bph_\lambda) =
  \chi_\lambda\Psi^{(1)}_{\bph}(\bj_\lambda(\bph_\lambda))
  \to \Psi^{(1)}_{\bph}(\bph) \quad\text{a.e.~in } Q,
\]
which readily implies the identification
\[
  \boldsymbol{\xi} = \Psi^{(1)}_{\bph}(\bph) \qquad\text{a.e.~in } Q.
\]
Eventually,
taking into account that $\chi_\lambda\to 1$ almost everywhere in $(0,T)$
and $\bph_\lambda\to\bph$ almost everywhere in $Q$,
from \eqref{est16}--\eqref{est17} we infer that
\begin{align*}
  \bph \in \!\!\bigcap_{p\in(1,2)}\!\!
  L^{2p}(0,T;\HH^2(\Omega)) \cap L^p(0,T; \WW^{2,r}(\Omega))
  \cap \!\! \bigcap_{\sigma\in(0,T)}\!\!
  L^4(\sigma,T;\HH^2(\Omega)) \cap L^2(\sigma,T; \WW^{2,r}(\Omega)),
\end{align*}
where $r=6$ if $d=3$ and $r$ is arbitrary in $(1,+\infty)$ if $d=2$.

Collecting all the information obtained so far, we can pass to the limit as $\lambda \to 0$ in the variational formulation
of \eqref{SYS:1_app}, \eqref{SYS:3_app}--\eqref{SYS:5_app} and infer that the limits fulfill
\begin{align*}
	& \<\dt \bph ,\bv>_{\VV}
	+ \iO \nabla \bmu : \nabla \bv
	=
	\iO \Sph (\bph, P, \br)\cdot \bv,
	\\
	& \<\dt P,v> _V +\iO \nabla  P \cdot \nabla v = \iO \SP(\bph, P, \br) v,
	\\
	& \<\dt \br,\bv>_{\VV}  +\iO \nabla \br : \nabla \bv= \iO \SR(\bph, P, \br)\cdot \bv ,
\end{align*}
for every $\bv\in \VV$ and $v\in V$, almost everywhere in $(0,T)$. Furthermore, multiplying \eqref{SYS:2_app}
by $\chi_\lambda$ yields
\[
  \chi_\lambda\bmu_\lambda = \chi_\lambda(-\Delta\bph_\lambda) +
  \chi_\lambda\Psi^{(1)}_{\lambda,\bph}(\bph_\lambda)
  +\chi_\lambda\Psi^{(2)}_\bph(\bph_\lambda) \quad\text{a.e.~in } Q.
\]
Then, testing by arbitrary $\bv\in\VV$ and integrating by parts (recall that $\chi_\lambda$ is just a function of time), we get
\[
  \int_\Omega\chi_\lambda\bmu_\lambda\cdot\bv =
  \int_\Omega\chi_\lambda\nabla\bph_\lambda:\nabla\bv
  +\int_\Omega\chi_\lambda\Psi^{(1)}_{\lambda,\bph}(\bph_\lambda)\cdot\bv
  +\int_\Omega\chi_\lambda\Psi^{(2)}_\bph(\bph_\lambda)\cdot\bv \quad\text{a.e.~in } (0,T),
\]
from which, thanks to the convergences above, letting $\lambda\to0$ we find
\[
  \int_\Omega\bmu\cdot\bv =
  \int_\Omega\nabla\bph:\nabla\bv
  +\int_\Omega\Psi^{(1)}_\bph(\bph)\cdot\bv
  +\int_\Omega\Psi^{(2)}_\bph(\bph)\cdot\bv \quad\text{a.e.~in } (0,T).
\]
Hence, the regularity of $\bph$
{and the arbitrariness of $\bv\in\VV$ imply the boundary condition $\dn \bph=\0$ almost everywhere on $\Sigma$,} which in turn allows us to infer
\[
  \bmu = -\Delta\bph + \Psi_{\bph}(\bph) \quad\text{a.e.~in } Q.
\]
This completes the proof of Theorem~\ref{THM:EX:WEAK}.

\section{Regularity results}\label{SEC:PAR3d}

\subsection{Proof of Theorem~\ref{THM:REGULARITY}}

Here we show some regularity results for the variables $P$ and $\br$ which directly follow from the parabolic structure of subsystem \eqref{SYS:3}--\eqref{SYS:4}, \eqref{SYS:6}--\eqref{SYS:8}.
For convenience, we rewrite it as
\begin{align*}
	\begin{cases}
	\dt P -\Delta P
	=
	 f_P
	 :=
	c_2 \ph_1  + c_4 \ph_2 - (c_1 R_1 +c_3 R_2)P &\qquad \text{in $Q$},
	\\
	 \dt \br-\Delta \br 	=
	{\bf f}_R
	:=
	-(c_1 P R_1 , c_3 PR_2)
	 + (c_2 \ph_1  ,c_4 \ph_2)
	&\qquad \text{in $Q$},
	\\
	\dn P = 0, \quad \dn \br = \0&\qquad \text{on $\Sigma$},
	\\
	P(0)=P_0 ,
	\quad \br(0)=\br_0=\big(\frac {1-P_0}2 \big)\1 &\qquad \text{in $\Omega$}.
	\end{cases}
\end{align*}
Due to Theorem \ref{THM:EX:WEAK}, we infer that
$ f_P \in \L\infty {H}, $ and ${\bf f}_R\in \L\infty \HH.$
Therefore, by virtue of the assumption on the initial data \eqref{P:regpar}, standard parabolic regularity theory ensure that the regularities in \eqref{reg:P}--\eqref{reg:R} are fulfilled concluding the proof.

\subsection{Proof of Theorem \ref{THM:REGULARITY:2d}}
The approximation scheme presented in Subsection~\ref{ssec:approx}
does not ensure enough regularity on the approximating solutions
allowing us to get \ref{reg:extra'} in a rigorous way.
We show then here how to suitably modify the approximating problem
\eqref{SYS:1_app}--\eqref{SYS:7_app} so that the corresponding
solution $(\bph_\lambda, \bmu_\lambda, P_\lambda, \br_\lambda)$
inherits enough regularity to perform the wanted estimates.

Given $\lambda>0$, we first regularize the initial datum $P_0$
through the resolvent of the Laplace operator with Neumann conditions
(namely, the operator $\RR$ defined in Subsection~\ref{SEC:NOT}) by setting
\[
  P_{0,\lambda}:=(I_H+\lambda\RR)^{-1}P_0 \in W\,, \qquad
  \br_{0,\lambda}:=\Big(\frac{1-P_{0,\lambda}}2 \Big)\1 \in \WW.
\]
Clearly, as $\lambda\searrow0$ it holds that $P_{0,\lambda}\to P_0$ in $H$
and $\br_{0,\lambda}\to  \br_{0}$ in $\HH$. Moreover,
noting that $(I_H+\lambda\RR)^{-1}(1/2)=1/2$ thanks to the Neumann boundary conditions,
the non-expansivity of $(I_H+\lambda\RR)^{-1}$ in $L^\infty(\Omega)$ yields
\[
  \norma{P_{0,\lambda}-\frac12}_{L^\infty(\Omega)} \leq \norma{P_{0}-\frac12}_{L^\infty(\Omega)}.
\]
Hence, the assumptions \eqref{ass:exweak:initialdata}--\eqref{max:ini:Linf}
are still satisfied by $(P_{0,\lambda}, \br_{0,\lambda})$, uniformly in $\lambda$.

We consider here the following modified approximated problem:
\begin{alignat}{2}
	\label{SYS:1_app_mod}
	& \dt \bph_\lambda - \Delta\bmu_\lambda
	= \Sph^\lambda(\bph_\lambda, P_\lambda, \br_\lambda)
		&&\qquad \text{in $Q$},
		\\	\label{SYS:2_app_mod}
	& \bmu_\lambda = \lambda\partial_t\bph_\lambda
	-\Delta \bph_\lambda + \Psi_{\lambda,\bph}(\bph_\lambda)
		&&\qquad \text{in $Q$},
		\\
	\label{SYS:3_app_mod}
	&\dt P_\lambda - \Delta P_\lambda
	=  \SP^\lambda(\bph_\lambda, P_\lambda, \br_\lambda)
		&&\qquad \text{in $Q$},\\
	\label{SYS:4_app_mod}
	&\dt \br_\lambda - \Delta \br_\lambda
	=  \SR^\lambda(\bph_\lambda, P_\lambda, \br_\lambda)
		&&\qquad \text{in $Q$},\\
	\label{SYS:5_app_mod}
	& \dn\bph_\lambda = \dn\bmu_\lambda  =\dn\br_\lambda
		= \0, \quad \dn P_\lambda= 0
		&&\qquad \text{on $\Sigma$},\\
	\label{SYS:7_app_mod}
	&\bph_\lambda(0)
	= \0, \quad
	\br_\lambda(0)
	= \br_{0,\lambda} , \quad
	P_\lambda(0)
	= P_{0,\lambda}
	&&\qquad \text{in $\Omega$}.
	\end{alignat}
The main differences
with respect to \eqref{SYS:1_app}--\eqref{SYS:7_app}
concern the regularized initial data and the (vanishing-in-$\lambda$) viscosity term $\lambda\partial_t\bph_\lambda$
in the Cahn--Hilliard equation. Proceeding exactly as in Subsection~\ref{ssec:approx},
taking into account the more regular initial data in the parabolic subsystem and
the viscosity contribution in the Cahn--Hilliard equation, it is a standard matter to
infer that \eqref{SYS:1_app_mod}--\eqref{SYS:7_app_mod} admits a unique global solution
\begin{align*}
  \bph_\lambda &\in H^1(0,T; \HH)\cap L^\infty(0,T; \VV)\cap L^2(0,T; \HH^3(\Omega)),\\
  \bmu_\lambda &\in L^2(0,T; \WW),\\
  P_\lambda &\in H^1(0,T; H)\cap L^2(0,T; W), \\
  \br_\lambda &\in H^1(0,T; \HH)\cap L^2(0,T;\WW).
\end{align*}
Actually, we can say something more. Indeed,
from the regularity of $P_\lambda$ and $\br_\lambda$ it readily follows that
$P_\lambda\in C^0([0,T]; V)$ and $\br_\lambda\in C^0([0,T]; \VV)$.
By the Lipschitz continuity and boundedness of $h$, this implies that
$h(P_\lambda)h(R_{\lambda,i})\in H^1(0,T; H)\cap L^\infty(0,T; V)$, $i=1,2$.
Since also $\bph\in H^1(0,T;\HH)\cap L^2(0,T;\WW)\subset C^0([0,T]; \VV)$,
we have that $\SP^\lambda(\bph_\lambda, P_\lambda, \br_\lambda)\in H^1(0,T; H)\cap L^\infty(0,T;V)$
and $\SR^\lambda(\bph_\lambda, P_\lambda, \br_\lambda)\in H^1(0,T; \HH)\cap L^\infty(0,T;\VV)$.
Hence,
by parabolic regularity in the reaction-diffusion subsystem
\eqref{SYS:3_app_mod}--\eqref{SYS:4_app_mod} we have then
\begin{align*}
P_\lambda &\in H^2(0,T; V^*)\cap C^1([0,T]; H)\cap H^1(0,T; V)\cap
C^0([0,T]; W)\cap L^2(0,T; H^3(\Omega)),\\
\br_\lambda &\in H^2(0,T; \VV^*)\cap C^1([0,T]; \HH)\cap H^1(0,T; \VV)\cap
C^0([0,T]; \WW)\cap L^2(0,T;\HH^3(\Omega)).
\end{align*}
Analogously, again by the boundedness and Lipschitz continuity of $h$
as well as {\bf A2}, we get that $\Sph^\lambda(\bph_\lambda, P_\lambda, \br_\lambda)
\in H^1(0,T;\HH)\cap L^\infty(0,T; \VV)$. Furthermore,
noting that $\Psi_{\lambda,\bph}\in C^1(\erre^2;\erre^2)$
with bounded derivatives thanks to {\bf A1}, and recalling that
$\bph_\lambda\in H^1(0,T;\HH)\cap C^0([0,T]; \VV)$, one can easily check that
$\Psi_{\lambda,\bph}(\bph_\lambda) \in H^1(0,T;\HH)\cap C^0([0,T]; \VV)$ by the dominated convergence theorem.
It follows that $\bRR\Psi_{\lambda,\bph}(\bph_\lambda)\in H^1(0,T;\WW^*)\cap C^0([0,T]; \VV^*)$.
Hence, one can rewrite the Cahn--Hilliard equation \eqref{SYS:1_app_mod}--\eqref{SYS:2_app_mod}
as
\[
  (I_\HH + \lambda\bRR)\partial_t\bph_\lambda + \bRR^2\bph_\lambda =
  \Sph^\lambda(\bph_\lambda, P_\lambda, \br_\lambda) - \bRR\Psi_{\lambda,\bph}(\bph_\lambda).
\]
Since we have shown that the right-hand side belongs to $H^1(0,T;\WW^*)\cap C^0([0,T];\VV^*)$
and $\bph(0)=\0$, classical parabolic regularity (see, e.g., \cite{Colli-Vis90}) implies that
$\bph_\lambda\in W^{1,\infty}([0,T]; \VV)\cap H^1(0,T; \WW)$.
Hence, comparison in the equation gives
$(I_\HH + \lambda\bRR)\partial_t\bph_\lambda \in H^1(0,T; \WW^*)$,
which implies that $\bph_\lambda \in H^2(0,T;\HH)$.
It follows that $\bph_\lambda \in H^2(0,T;\HH)\cap H^1(0,T;\WW)\subset C^1([0,T]; \VV)$,
and again by comparison in the equation this implies
$\bRR^2\bph_\lambda\in C^0([0,T]; \VV^*)$, yielding
$\bph_\lambda\in C^0([0,T]; \HH^3(\Omega))$ by elliptic regularity.
Putting this information together
and recalling the definition of $\bmu_\lambda$,
we deduce that the approximated
solution $(\bph_\lambda, \bmu_\lambda, P_\lambda, \br_\lambda)$ of the system
\eqref{SYS:1_app_mod}--\eqref{SYS:7_app_mod} satisfies
\begin{align*}
  \bph_\lambda &\in H^2(0,T;\HH)\cap C^1([0,T];\VV)\cap
  H^1(0,T; \WW)\cap C^0([0,T];  \HH^3(\Omega)),\\
  \bmu_\lambda &\in H^1(0,T;\HH)\cap C^0([0,T];\VV)\cap L^2(0,T; \HH^3(\Omega)),\\
  P_\lambda &\in H^2(0,T; V^*)\cap C^1([0,T]; H)\cap H^1(0,T; V)\cap
C^0([0,T]; W)\cap L^2(0,T; H^3(\Omega)),\\
\br_\lambda &\in H^2(0,T; \VV^*)\cap C^1([0,T]; \HH)\cap H^1(0,T; \VV)\cap
C^0([0,T]; \WW)\cap L^2(0,T;\HH^3(\Omega)).
\end{align*}
In particular, the variational formulations of the
approximated equations \eqref{SYS:1_app_mod}--\eqref{SYS:2_app_mod}
hold for {\em every} $t\in[0,T]$, and not just almost everywhere: this will be essential
to give rigorous meaning to the estimates to follow.

Taking these considerations into account, for brevity of notation
we proceed now in a formal way,
by arguing directly on the limiting variables $(\bph,\bmu,P,\br)$ and on system \eqref{SYS:1}--\eqref{SYS:4}, \eqref{SYS:5}--\eqref{SYS:8}, bearing in mind  that
the argument can be made rigorous by working with the aforementioned approximation
and passing to the limit as $\lambda \to 0$ in the spirit of Subsection \ref{SEC:PASSTOLIMIT}.

\noindent
{The vanishing initial condition $\bph(0)= \0$ prevents us to obtain
higher-order estimates starting from the initial time: for this reason, we make use instead of
suitable time-averaged estimates.
We test \eqref{SYS:1} by $t\dt \bmu$, the time-derivative of \eqref{SYS:2} by $-t\dt\bph$, integrate in time, and sum.
Note that thanks to the enhanced regularity of the approximated solutions
to the system \eqref{SYS:1_app_mod}--\eqref{SYS:7_app_mod} proved above,
these testings are actually rigorous (at $\lambda>0$) as scalar products
in $L^2(0,T;\HH)$ for example.}

{We deduce that
\begin{align*}
  &\frac{t}2\norma{\nabla\bmu(t)}^2 - t(\Sph(t), \bmu(t))
  +\int_0^ts\left(\norma{\nabla\partial_t\bph(s)}^2
  +(\Psi_{\bph\bph}(\bph(s))\partial_t\bph(s), \partial_t\bph(s))
  \right)\,\ds\\
  &\quad\leq \frac12\int_0^t\norma{\nabla\bmu(s)}^2\,\ds
  -\int_0^t\left(\Sph(s),\bmu(s)\right)\,\ds
   -\int_0^ts(\partial_t(\Sph(s)), \bmu(s))\,\ds.
\end{align*}
Now, the first two terms on the right-hand side can be controlled by
the estimates \eqref{est12}--\eqref{est13} and the boundedness of $\Sph$ (which is replaced by $\Sph^\lambda$ in the rigorous framework).
Using \ref{ass:pot} along with the compactness inequality \eqref{comp:ineq2},
we infer that
\[
	(\dt \bph, \Psi_{\bph\bph}(\bph)\dt \bph)
	\geq
	- C \norma{\dt \bph}^2
	\geq
	-\frac12 \norma{\nabla \dt\bph}^2
	- C\norma{\dt \bph}^2_{\bstar}.
\]
Next, let us notice that due to \ref{ass:sources}, using once again \eqref{comp:ineq2}
and comparison in \eqref{SYS:1} we obtain
\begin{align*}
	-s(\dt(\Sph(s)) ,  \bmu(s) )
	& \leq s\left(\norma{P(s)}_{L^\infty(\Omega)}\norma{\dt \br(s)}
	+\norma{\partial_t P(s)}\norma{\br(s)}_{\LL^\infty(\Omega)}\right)\norma{\bmu(s)}
	\\ & \qquad + s\norma{\partial_t\bph(s)}_{\VV^*}\norma{\bmu(s)}_\VV
	\\ & \leq
	\norma{\dt \br(s)}^2
	+\norma{\partial_t P(s)}^2 + \frac{s^2}2\norma{\bmu(s)}^2
	+ Cs\left(1+\norma{\nabla\bmu(s)}\right)\norma{\bmu(s)}_\VV
	\\ & \leq
	\norma{\dt \br(s)}^2 +\norma{\partial_t P(s)}^2
	+ Cs^2\norma{\bmu(s)}_\VV^2 + C\left(1+\norma{\nabla\bmu(s)}^2\right),
\end{align*}
where we have updated the value of $C$, as usual.
The regularity of $(\bph,\bmu,P,\br)$ given by
Theorem~\ref{THM:EX:WEAK} and Theorem~\ref{THM:REGULARITY}
yields then
\[
  -\int_0^ts(\dt(\Sph(s)) ,  \bmu(s) )\,\ds
  \leq C\left(1+\int_0^ts^2\norma{\bmu(s)}_\VV^2\,\ds \right).
\]
Moreover, by the Young and Poincar\'e--Wirtinger inequalities
we have, exploiting the regularity of  $(\bph,\bmu,P,\br)$ given by
Theorem~\ref{THM:EX:WEAK} and the form of $\Sph$, that for every $\delta>0$
\begin{align*}
  t\big(\Sph(t) ,  \bmu(t) \big)&\leq
  \delta t \norma{\nabla\bmu(t)}^2 + \delta t^2|(\bmu(t))_\Omega|^2
  \\ & \quad +
  C_\delta (t+1)(\norma{\bph(t)}^2+\norma{\br(t)}^2+\norma{P(t)}^2+1)\\
  &\leq
  \delta t\norma{\nabla\bmu(t)}^2 + \delta t^2|(\bmu(t))_\Omega|^2+
  C_{\delta,T}.
\end{align*}
Let now $T_0$ be as defined in Section~\ref{sec:chem} (recall that
$T_0$ is independent of $\lambda$).
Taking into account  the
(corresponding limit version of) inequality \eqref{est8_prequel},
thanks to the regularity of $\bph$ we have, for every $t\in(0,T_0)$, that
\begin{align*}
  t^2|(\bmu(t))_\Omega|^2
  &\leq Ct^2\left(1+ |F(c_0't)|^2\right)+
  C_\alpha t^{2\alpha}
  \left(1+\|\nabla\bmu(t)\|^2\right),
\end{align*}
which in turn yields by assumption \eqref{ip_F} that
\begin{align*}
  t^{3-2\alpha}|(\bmu(t))_\Omega|^2
  &\leq C_\alpha + C_\alpha t
  \left(1+\|\nabla\bmu(t)\|^2\right).
\end{align*}
Hence, recalling that $\alpha\in(0,1)$ and rearranging the terms, we get in particular that
\[
  t\big(\Sph(t) ,  \bmu(t) \big) \leq \delta t\norma{\nabla\bmu(t)}^2 + C_{\delta}
  \quad\forall\,t\in(0,T_0).
\]
Similarly, exploiting \eqref{est7_prequel} instead, one easily obtains also
the complementary case
\[
  t\big(\Sph(t) ,  \bmu(t) \big) \leq \delta t\norma{\nabla\bmu(t)}^2 + C_{\delta}
  \quad\forall\,t\in(T_0,T).
\]
The same computations also imply, in particular, that
\[
  \int_0^ts^2\norma{\bmu(s)}_\VV^2\,\ds \leq
  C_\alpha\left(1+ \int_0^ts\norma{\nabla\bmu(s)}^2\,\ds\right).
\]
Choosing then $\delta$ small enough and rearranging the terms yield
\begin{align*}
  &t^{3-2\alpha}|(\bmu(t))_\Omega|^2 +
  t\norma{\nabla\bmu(t)}^2 +\int_0^ts\norma{\nabla\partial_t\bph(s)}^2\,\ds
  \leq C_\alpha\left(1+\int_0^ts\norma{\nabla\bmu(s)}^2\,\ds \right).
\end{align*}
Hence the Gronwall lemma, along with the Poincar\'e--Wirtinger inequality, implies} that
for every $\alpha\in(0,1)$ there exists $C_\alpha>0$ such that
\begin{align}
  \label{est_reg1}
	\norma{t^{\frac 12}\nabla\dt\bph}_{L^2(0,T; {\HH})}
	+\norma{t^{\frac32-\alpha}\bmu}_{L^\infty(0,T; {\VV})}
	+\norma{t^{\frac 12}\nabla\bmu}_{L^\infty(0,T;\HH)}
	\leq C_\alpha.
\end{align}
It is then a standard matter to derive from a comparison argument in equation \eqref{SYS:1} that
\begin{align}
  \label{est_reg2}
		\norma{t^{\frac 12}\dt \bph}_{L^\infty(0,T; {\VVp})}
		\leq C_\alpha.
\end{align}
This ends the proof of Theorem \ref{THM:REGULARITY:2d}.

\subsection{Proof of Theorem~\ref{THM:2d:separation}}

We recall that, contrary to the existing literature, here we start from a pure phase $\bph_0 = \0$ as dictated by the model.
This entails, in particular, that $\widetilde{\bmu} \in \Ls 2\VV$ holds just for all $\sigma\in (0,T)$ and not in $\sigma= 0$ as usual.

The proof is split into three steps. The final and crucial estimate that will allow us to deduce \eqref{sep}
is the uniform bound $\widetilde{\bmu} \in L^\infty (\sigma,T; \HH^2(\Omega))$ for all $\sigma\in(0,T)$.
For the sake of convenience, we drop the $\sim$ superscript but on the approximated double-well potential $\widetilde \Psi$.
We can refer to the previous theorem for a rigorous approximation scheme.

\noindent
{\sc First estimate.}
Let us rewrite \eqref{SYS:2} as a one-parameter family of time-dependent elliptic system with maximal monotone perturbation:
\begin{align*}
	\begin{cases}
	- \Delta \bph + \widetilde \Psi_\bph^{(1)}(\bph) = \bmu - \widetilde \Psi^{(2)}_{\bph}(\bph)=:{\bf f}_{\bmu} \quad  & \hbox{in $\Omega\times (0,T)$, }
	\\
	\dn \bph = \0 \quad & \hbox{on $\Gamma\times (0,T)$.}
	\end{cases}
\end{align*}
Owing to \eqref{est_reg1}, the inclusion
$\VV\emb \LL^r(\Omega)$ for all $r >2$ (recall that now $d=2$),
and elliptic regularity theory, we argue componentwise to infer that (see \cite[Cor.~4.1]{GGM})
\begin{align}
\label{est_reg3}
	\norma{ t^{\frac32-\alpha}\bph}_{L^\infty(0,T; \WW^{2,r}(\Omega)) }
	+\norma{t^{\frac32-\alpha}\widetilde \Psi^{(1)}_\bph(\bph)}_{L^\infty(0,T; \LL^r(\Omega))}
	 \leq C_{\alpha,r},
\end{align}
where we underline that the positive constant $C_{\alpha,r}$ may depend on $\alpha$ and $r$, as the notation suggests.
Furthermore, on account of \eqref{growth}, we can argue componentwise as in \cite[Lemma~5.1]{GGM} to get
\begin{equation}
  \label{est_reg4}
  \norma{e^{-C_\alpha t^{\alpha-\frac32}}\widetilde \Psi^{(1)}_{\bph\bph}(\bph)}_{L^\infty(0,T;\LL^r(\Omega))} \leq C_{\alpha,r}.
\end{equation}

\noindent
{\sc Second estimate.}
Taking the time derivatives of \eqref{SYS:1}--\eqref{SYS:2}, one has that
\begin{align}\label{time1}
  &\<\dt^2\bph,\bv>_\VV - (\dt \bmu, \Delta \bv) =  \big( \dt\Sph , \bv \big)
	\quad \forall\,\bv \in \WW, \quad\text{a.e.~in } (0,T),\\
\label{time2}
  &\dt \bmu = - \Delta  \dt \bph + \dt(\widetilde \Psi_{\bph}(\bph))
	 \quad\text{a.e.~in } (0,T).
\end{align}
Given $C_\alpha>0$ as in \eqref{est_reg3},
we test equation \eqref{time1} by $t\exp(-C_\alpha t^{\alpha-\frac32})\partial_t\bph$,
equation \eqref{time2} by $-t\exp(-C_\alpha t^{\alpha-\frac32})\Delta\partial_t\bph$, add the resulting equations together, and integrate with respect to time.
Again, we point out that these computations can be made rigorous as testings in $L^2(0,T;\HH)$
at the approximation level $\lambda>0$, thanks to the enhanced
regularity of the solutions $\bph_\lambda$ and $\bmu_\lambda$ proved above.

Noting that two terms cancel out, we obtain
\begin{align*}
  &\frac{t e^{-C_\alpha t^{\alpha-\frac32}}}2\norma{\dt\bph(t)}^2
  +\int_0^ts e^{-C_\alpha s^{\alpha-\frac32}}\norma{\Delta\dt\bph(s)}^2\,\ds\\
  &\quad\leq\frac12\int_0^te^{-C_\alpha s^{\alpha-\frac32}}
  \left(1 +C_\alpha\left(\frac32-\alpha\right) s^{\alpha-\frac32}\right)
  \norma{\dt\bph(s)}^2\,\ds\\
  &\qquad+\int_0^ts e^{-C_\alpha s^{\alpha-\frac32}}\big(\dt\Sph(s), \dt\bph(s)\big)\,\ds\\
  &\qquad
  +\int_0^ts e^{-C_\alpha s^{\alpha-\frac32}}\big(\dt(\widetilde \Psi_{\bph}(\bph(s))), \Delta\dt\bph(s) \big)\,\ds.
\end{align*}
The first term on the right-hand side is controlled by the estimates \eqref{est_reg1}--\eqref{est_reg2} since
an elementary computation shows that there exists $\widetilde C_\alpha>0$ such that
\[
  e^{-C_\alpha s^{\alpha-\frac32}}
  \left(1 +C_\alpha\left(\frac32-\alpha\right) s^{\alpha-\frac32}\right) \leq \widetilde C_\alpha s
  \quad\forall\,s\in[0,T].
\]
Similarly, by analogous computations, the second term on the right-hand side
is also controlled by the estimates \eqref{est_reg1}--\eqref{est_reg2},
since by {\bf A2} we have
\begin{align*}
  \norma{\dt\Sph} &\leq C\left(\norma{\dt\bph}
  +\norma{P}_{L^\infty(\Omega)}\norma{\dt\br} + \norma{\br}_{\LL^\infty(\Omega)}\norma{\dt P}\right)\\
  &\leq C\left(\norma{\dt\bph}
  +\norma{\dt\br} + \norma{\dt P}\right).
\end{align*}
Hence, possibly updating the value of $C_\alpha$ as usual,
the H\"older and Young inequalities
together with the continuous embedding $\VV\emb \LL^6(\Omega)$ entail that
\begin{align*}
  &\frac{t e^{-C_\alpha t^{\alpha-\frac32}}}2\norma{\dt\bph(t)}^2
  +\frac12\int_0^ts e^{-C_\alpha s^{\alpha-\frac32}}\norma{\Delta\dt\bph(s)}^2\,\ds\\
  &\quad\leq C_\alpha
  +\frac12\int_0^t
  s e^{-C_\alpha s^{\alpha-\frac32}}
  \norma{\widetilde \Psi_{\bph\bph}(\bph(s))}_{\LL^3(\Omega)}^2\norma{\dt\bph(s)}_{\LL^6(\Omega)}^2\,\ds\\
  &\quad\leq C_\alpha +
  C_\alpha\norma{r\mapsto e^{-C_\alpha r^{\alpha-\frac32}}\widetilde \Psi_{\bph\bph}(\bph(r))}_{L^\infty(0,T;\LL^3(\Omega))}^2
  \int_0^Ts\norma{\dt\bph(s)}_{\VV}^2\,\ds.
\end{align*}
Renominating $C_\alpha$ once more, it follows from the estimates
\eqref{est_reg1}--\eqref{est_reg2} and \eqref{est_reg4},
and from elliptic regularity theory by comparison in \eqref{SYS:1}, that
\begin{align}
  \label{est_reg5}
 \norma{t^{\frac 12}e^{-C_\alpha t^{\alpha-\frac32}}\dt\bph}_{L^\infty(0,T; \HH)\cap L^2(0,T;\WW)}
 +\norma{t^{\frac 12}e^{-C_\alpha t^{\alpha-\frac32}}\bmu}_{L^\infty(0,T;\WW)} \leq C_\alpha.
\end{align}

\noindent
{\sc Separation property.}
First, due to the Sobolev embedding $\HH^2(\Omega) \emb \LL^\infty(\Omega)$,
from the estimate \eqref{est_reg5} we have
\begin{align*}
	 \norma{t^{\frac 12}e^{-C_\alpha t^{\alpha-\frac32}}\bmu}_{L^\infty(0,T; \LL^\infty(\Omega))}	
	 \leq C_\alpha.
\end{align*}
Next, arguing as in the proof of \cite[Thm.~5.1]{GGM} componentwise, we get
\begin{align}
\label{sepprop}
	 \norma{t^{\frac 12}e^{-C_\alpha t^{\alpha-\frac32}}\widetilde \Psi_{\bph}^{(1)} (\bph)}_{L^\infty(0,T; \LL^\infty(\Omega))}	
	 \leq C_\alpha.
\end{align}
Then, recalling \ref{growth2}, we have that, for every $\sigma\in(0,T)$,
\begin{align*}
	 \norma{\widetilde \Psi_{\bph}^{(1)} (\bph)}_{L^\infty(\sigma,T; \LL^\infty(\Omega))}	
	 \leq C_\sigma,
\end{align*}
for some $C_\sigma>0$, and \eqref{sep} follows.

Finally, using \eqref{SYS:2} it is easy to recover the last regularity \eqref{reg:extra} in the theorem from \eqref{est_reg5}, the separation property, and elliptic regularity theory.

\begin{remark}
We recall that in order to prove \eqref{sepprop} an essential ingredient is \eqref{est_reg4}.
The argument used in \cite{GGM} (see also \cite{MZ}) to get \eqref{est_reg4} is based on the two-dimensional Moser--Trudinger inequality.
A recent alternative (and simplified) approach which does not require such an inequality can be found in \cite[Sec.~3]{GGG} though it is still two dimensional. What happens in dimension three is an open issue.
\end{remark}

\color{black}

\section{Proof of Theorem \ref{THM:UNIQ:2d}}
\label{SEC:UQ}
This last section is devoted to the unique continuation property of solutions obtained from Theorem~\ref{THM:2d:separation} in the two-dimensional case.
As already mentioned in the Introduction, the separation property \eqref{sep}
plays a fundamental role in the derivation of such a result as it guarantees that the nonlinear potential 
$\widetilde \Psi$, along with its higher-order derivatives, is a globally \Lip\ continuous function.

Let us set the following notation for the difference of two given solutions (the $\sim$ superscript is dropped for the sake of simplicity)
\begin{align*}
	\bph & :=  \bph^1 - \bph^2,
	\quad
	\bmu:= \bmu^1 - \bmu^2,
	\quad
	P:= P^1 - P^2,
	\quad
	\br:= \br^1 - \br^2,
	\\
	\Sph^i& := \Sph(\bph^i, P^i, \br^i),
	\quad
	\SP^i := \SP(\bph^i, P^i, \br^i),
	\quad
	\SR^i := \SR(\bph^i, P^i, \br^i)
	\quad \text{for $i=1,2$},
	\\
	\Sph& := \Sph^1 - \Sph^2,
	\quad 	
	\SP := \SP^1 - \SP^2,
	\quad
	\SR := \SR^1 - \SR^2,
\end{align*}
where, again, the lower subscripts denote the components of the occurring vector, for instance, for the phase field variable $\bph$ we have
\begin{align*}
	\bph = (\ph_1, \ph_2)= (\ph_1^1-\ph_1^2, \ph_2^1-\ph_2^2).
\end{align*}
Due to the definitions of the source terms \eqref{def:sorgenti} and \eqref{def:Sp}, we have
\begin{align}
	\label{Sph:diff}
	\Sph & = - \SR= (c_1 (P^1R_1 + P R_1^2) - c_2 \ph_1, c_3 (P^1R_2 + P R_2^2) - c_4 \ph_2),
	\\
	\label{SP:diff}
	\SP & = -c_1 (P^1R_1 + P R_1^2) + c_2 \ph_1 -  c_3 (P^1R_2 + P R_2^2) + c_4 \ph_2.
\end{align}
Next, we consider problem \eqref{SYS:1}--\eqref{SYS:4}, \eqref{SYS:5}--\eqref{SYS:8} written for the difference of the two given solutions which (formally) reads as follows
\begin{alignat}{2}
	\label{SYS:CD:1:2d}
	&\dt \bph  -\Delta \bmu
	=  \Sph
		&&\qquad \text{in $Q$},
		\\	\label{SYS:CD:2:2d}
	 &\bmu= -\Delta \bph + (\widetilde \Psi_{\bph}(\bph^1) - \widetilde \Psi_{\bph}(\bph^2) )
		&&\qquad \text{in $Q$},
		\\
	\label{SYS:CD:3:2d}
	&\dt P - \Delta P
	=\SP
		&&\qquad \text{in $Q$},\\
	\label{SYS:CD:4:2d}
	&\dt \br - \Delta \br
	= \SR
		&&\qquad \text{in $Q$},
		\\
	\label{SYS:CD:5:2d}
	& \dn \bph= \dn \br = \0, \quad \dn P= 0
		&&\qquad \text{on $\Sigma$},
	\\
	\label{SYS:CD:6:2d}
	& \bph(0) =\0, \quad P_0 =  P_0^1 - P_0^2, \quad \br (0)=\br_0^1-\br_0^2
		&&\qquad \text{in $\Omega$}.
	\end{alignat}

\noindent
On account of Theorem \ref{THM:REGULARITY:2d}, we multiply \eqref{SYS:CD:1:2d} by $\bph$, \eqref{SYS:CD:2:2d} by $-\Delta \bph$, \eqref{SYS:CD:3:2d} by $P$, \eqref{SYS:CD:4:2d} by $\br$, integrate in space,
add the resulting equalities and integrate by parts. This gives
\begin{align}
	\non
	& \frac 12 \frac d {dt} \Big(\norma{\bph}^2 + \norma{P}^2
	+ \norma{\br}^2 \Big)
	+ \norma{\Delta \bph}^2
	+ \norma{\nabla P}^2
	+ \norma{\nabla \br}^2
	\\ & \quad
	= (\Sph, \bph)
	+ (\widetilde \Psi_{\bph}(\bph^1) - \widetilde \Psi_{\bph}(\bph^2) , \Delta \bph)
	+(\SP, P)
	+(\SR, \br).
	\label{uq:proof:1}
\end{align}
The last two terms on the \rhs\ involving $\SP$ and $\SR$ can be handled by using \eqref{Sph:diff} and \eqref{SP:diff}. This gives
\begin{align*}
	(\SP, P)
	+(\SR, \br)
	\leq  C (\norma{\bph}^2+\norma{P}^2+\norma{\br}^2)
\end{align*}
as $(\bph^i, \bmu^i, P^i, \br^i)_i$, for $i=1,2$ fulfill the bounds in Theorem \ref{THM:EX:WEAK}.
Moreover, on account of \eqref{sep}, for all fixed $\s \in (0,T)$ and
$s\in(\s,T)$, we have
\[
\norma{\widetilde \Psi_{\bph}(\bph^1(s)) - \widetilde \Psi_{\bph}(\bph^2(s))}
\leq C_\sigma\norma{\bph^1(s) - \bph^2(s)}.
\]
Therefore, for the other terms on the \rhs\ of \eqref{uq:proof:1}, upon integrating over $(\s,t)$, for an arbitrary $t \in (\s,T)$, we get
\begin{align*}
	& \int_\sigma^t (\Sph(s), \bph(s)) \,\ds
	+	\int_\sigma^t \big(\widetilde \Psi_{\bph}(\bph^1(s)) - 
	\widetilde \Psi_{\bph}(\bph^2(s)) , \Delta \bph(s) \big)\,\ds
	\\ & \quad
	\leq
	\frac 12 	\int_\sigma^t \norma{\Delta \bph(s)}^2\,\ds
	+ C 	\int_\sigma^t(\norma{\bph(s)}^2+\norma{P(s)}^2+\norma{\br(s)}^2) \, \ds,
\end{align*}
where we also owe to \eqref{Sph:diff} and Young's inequality.
Thus, we integrate \eqref{uq:proof:1} in time as above and obtain that
\begin{align*}
	& \norma{\bph(t)}^2
	+ \norma{P(t)}^2+\norma{\br(t)}^2
	 + \frac 12 \int_{\s}^{t} \norma{\Delta \bph(s)}^2\ds
	+ \int_{\s}^{t}( \norma{\nabla \br(s)}^2+ \norma{\nabla P(s)}^2)\,\ds
	\\ & \quad
	 \leq
	 C ({\norma{\bph(\s)}^2}+ \norma{P(\s)}^2+\norma{\br(\s)}^2)
	 + C \int_{\s}^{t} (\norma{\bph(s)}^2 + \norma{P(s)}^2	+ \norma{\br(s)}^2)\,\ds.
\end{align*}
Hence Gronwall's lemma yields \eqref{est:cd:2d}.

\begin{remark} \label{REM:UQ}
We are unable to prove a similar result for weak solutions in dimension three.
As above, let $(\bph^i, \bmu^i, P^i, \br^i)$, $i=1,2$, denote two solutions obtained from Theorem \ref{THM:EX:WEAK} associated to initial data $P_0^i$ and $\br_0^i$.
On account of \cite[Proof of Thm.~3.1, p.~2493]{GGM}, using the notation for the differences introduced in the proof of Theorem \ref{THM:UNIQ:2d}, let us claim that
the following inequality can be obtained:
\begin{align*}
	& \frac 12\frac d {dt} \Big(\norma{\bph}_{-\1}^2+ \norma{P}^2+\norma{\br}^2\Big)
	 + \norma{\nabla \bph}^2 + \norma{\nabla \br}^2 + \norma{\nabla P}^2
	\\ & \quad
	 \leq
	C \Big(\norma{\bph}_{-\1}^2
	 + \norma{P}^2+\norma{\br}^2
	 + (\Lambda+ \Pi)|\bph_\Omega|\Big) .
\end{align*}
The functions $\Lambda$ and $\Pi$ appearing on the \rhs\ are defined and fulfill (due to Theorem \ref{THM:EX:WEAK}), for every $\s \in (0,T)$,
\begin{align*}
	 t \mapsto \Lambda(t) & :=\big(\norma{\Psi_{\bph}^{(1)}(\bph^1(t)) }_{\LL^1(\Omega)}
	 +\norma{\Psi_{\bph}^{(1)}(\bph^2(t)) }_{\LL^1(\Omega)}\big) \in L^1(\sigma,T),
	\\
	 t \mapsto \Pi(t) & : =\max \{|(P^1(t)R_1(t) + P(t) R_1^2(t))_\Omega|,
	 | (P^1(t)R_2 (t)+ P(t) R_2(t)^2)_\Omega|\} \in  L^1(\sigma,T).
\end{align*}
Then, integration over time, \eqref{ass:exweak:initialdata}, the Gronwall lemma,
along with the equivalence of the norms $\norma{\cdot }_{-\1}$ and $\norma{\cdot}_\VVp$, yield that
\begin{align*}
	& \norma{\bph}_{L^\infty (\sigma, T ; \VVp) \cap L^2(\s, T; \VV)}
	+ \norma{P}_{L^\infty (\sigma, T ; H) \cap L^2( \s, T; V)}
	+ \norma{\br}_{L^\infty (\sigma, T ; \HH) \cap L^2( \s, T; \VV)}
	\\ & \quad
	\leq C  (\norma{\bph (\sigma)}_{\VVp}^2+ \norma{P(\sigma)}^2)
	+ C \norma{\bph_\Omega}^{1/2}_{L^\infty(\sigma,T)}.
\end{align*}
The crucial point is that, now, in contrast to \cite{GGM}, the last term cannot
be controlled independently of the original system (cf. \eqref{Sph:diff} and equation \eqref{SYS:CD:1:2d}).
In fact, using \eqref{oss:mean} we find, for $i=1,2$,
\begin{align}
	\non
	& \norma{(\ph_i^1-\ph_i^2)_\Omega}_{L^\infty(\s,T)}
	\non
	\leq C \norma{(P^1R_i + P R_i^2)_\Omega}_{L^1(\s,T)}
	\leq C \norma{P^1R_i + P R_i^2}_{\Ls 1 \Vp}
\end{align}
which does not suffice to express the last term in the above estimate in term of the initial data.
A different choice like the one done in the above proof (i.e., taking $\bph$ as test function in \eqref{SYS:CD:1:2d}) requires
the separation property whose validity is unknown in dimension three.
\end{remark}

\section{Cahn--Hilliard equations with sources}
\label{SEC:CONC}

The techniques developed in this paper allow us to extend some previous well-posedness results known for Cahn--Hilliard equations with source terms
(see, for instance,  \cite{Lam, Mir, FMN, IM, S}).
In particular, our approach is more general compared to the current state of arts as it allows the initial datum $\ph_0$ to be also the pure phase , say, $0$, and this does not comply with the standard requirement $(\ph_0)_\Omega \in (0,1)$.
Let us incidentally comment that in the case of scalar order parameters pure phases are often taken to be $-1$ are $1$ (leading to the corresponding condition $(\ph_0)_\Omega \in (-1,1)$).

We need to simplify our scenario and rephrase the problem  in a form which can be useful elsewhere.
This extension is not just a mathematical improvement: despite the
fact that condition $(\ph_0)_\Omega \in (-1,1)$ is perfectly natural for the classical Cahn--Hilliard equation, where pure phases are equilibrium configurations due to the mass conservation property, this may not be the case if a source term is present because in this case no mass conservation is expected. A meaningful example is just given by the phenomenon whose mathematical model is analyzed in this paper.

We use the same notation as before, namely, $\ph$ stands for the order parameter,
$\mu$ for the corresponding chemical potential,
$\ph_0$ is the initial datum, and $F'$ denotes the derivative of a singular, while regular, potential $F$, the prototype being the logarithmic potential, whose convex part is defined by
\begin{align}
	\label{Flog:mono}
	F_{\rm log,1 }(r) := \Psi^{(1)}\Big(0,\frac{r+1}2\Big)=\Psi^{(1)}\Big(\frac{r+1}2,0\Big), \quad r \in [-1,1],
\end{align}
{and $\Psi^{(1)}$ is given as in \eqref{Flog}.}
Moreover, let $S(\ph) $ and $g$ be given functions acting as source terms.
Then, we consider the initial and boundary value problem
\begin{alignat}{2}	
	\label{sys:1:abstract}
	&  \dt \ph - \Delta \mu  = S(\ph) \qquad && \text{in $Q$,}
	 \\
	\label{sys:2:abstract}
	 & \mu  =  -\Delta \ph + F'(\ph) + g  \qquad && \text{in $Q$,}
	 \\
	\label{sys:3:abstract}
	&  \dn \ph
	 =  \dn\mu
	= 0  \qquad && \text{on $\Sigma$,}
	 \\	
	 \label{sys:4:abstract}
	&  \ph(0) = \ph_0  \qquad && \text{in $\Omega$.}
\end{alignat}
A key point to handle Cahn--Hilliard-type systems with singular potential is the mass conservation property arising from \eqref{sys:1:abstract} when $S(\ph)=0$.
In those problems, the mass is conserved, and therefore, to fulfill the physical requirement $(\ph(t))_\Omega \in (-1,1)$ for all $t \in [0,T]$ (necessary to handle the singular term $F'(\ph)$ in \eqref{sys:2:abstract}), it suffices to choose the initial datum $\ph_0$ such that $(\ph_0)_\Omega \in (-1,1)$.
For this reason, the main difficulty of \eqref{sys:1:abstract}--\eqref{sys:4:abstract}, when $S(\ph)\not=0$, is the interplay between the singularity of $F'$ and the evolution of total mass of the order parameter, which is ruled by $S(\ph)$. In fact, testing equation \eqref{sys:1:abstract} by $\frac 1{|\Omega|}$ and setting $y:= \ph_\Omega$, we find
\begin{equation}
\label{ODEmean}
	y' (t)= \frac 1 {|\Omega|}\iO S(\ph(t)) \quad \text{for all $t \in [0,T],$}
\end{equation}
meaning that the physical condition $y(t)=\ph_\Omega (t)\in (-1,1)$ is fulfilled for all $t \in [0,T]$ if and only if $S$ has a suitable structure.

As a direct consequence of Theorem \ref{THM:EX:WEAK}, we can state an existence result for  \eqref{sys:1:abstract}--\eqref{sys:4:abstract}.

\noindent
Suppose that the potential $F$ enjoys the splitting $F=F_1 + F_2$ with
\begin{align}
	\label{F:abs:1}
	& F_1 \in C^0([-1,1]) \cap C^2((-1,1)), \quad F_2 \in C^1(\erre),
	\\
	\label{F:abs:2}
	& F_1 :\erre \to [0, +\infty]\quad\text{is proper and convex},
	\\
	\label{F:abs:3}
	& F_2':\erre \to \erre \quad\text{is $L_0$-Lipschitz continuous for some $L_0>0$.}
\end{align}
Assume, in addition, that:
\begin{align}
	\label{F:abs:4}
	& \lim_{r \to -1^{+}} F'(r) = -\infty,
	\quad
	\lim_{r \to 1^-}F'(r) = +\infty,
	\\
	\label{F:abs:5}
	&
	\exists\, \rho \in (0,1), \;q \in (2,+\infty) : \; F' \in L^q(-1,-1+\rho)\cap C^0(-1,-1+\rho).
\end{align}
Arguing as in Appendix \ref{app}, it is easy to check that
conditions \eqref{F:abs:1}--\eqref{F:abs:5} imply
the following properties:
\begin{property}\label{MZ0:abs}
 For every compact subset $K\subset (-1,1)$,
  there exist constants $c_K, C_K>0$ such that,
  for every $r\in(-1,1)$ and for every $r_0\in K$ it holds that
  \begin{equation}\label{MZ_base:abs}
  c_K|F'_1(r)|
  \leq F'_1(r)(r-r_0)
  +C_K.
  \end{equation}
\end{property}
\begin{property}\label{MZ:abs}
  There exist constants $c_F, C_F>0$ such that, for every measurable
  $\phi:\Omega\to(-1,1)$ satisfying
  \[
  -1<\phi_\Omega< -1+ \rho,
  \]
  it holds that
  \begin{equation}\label{MZ_gen:abs}
  c_F\, (\phi_\Omega+1) \int_\Omega|F_1'(\phi)|
  \leq \int_\Omega F'_1(\phi)(\phi- \phi_\Omega)
  +C_F\, (\phi_\Omega+1)
  F'(\phi_\Omega).
  \end{equation}
\end{property}
In particular, we note that the logarithmic potential \eqref{Flog:mono}
satisfies \eqref{F:abs:1}--\eqref{F:abs:5}, hence also Properties~\ref{MZ0:abs}--\ref{MZ:abs}.

The existence result is given by the following theorem.
\begin{theorem}\label{THM:ABSTRACT}
Let $F$ be as above.
Assume also that
\begin{align*}
	\ph_0 \in V, \quad F(\ph_0)\in\Lx1, \quad (\ph_0)_\Omega \in [-1,1),
\end{align*}
and let $g \in \L p V \cap \Ls2 V$ for every $p\in (1,2)$ and every $\s \in (0,T)$.
If $S$ has the following structure:
\begin{align}
	\label{source:splitting}
	S(\ph)=  - m \ph + h(\ph,g),	 \quad \text{with $m>0$ and $h:\erre\times \erre\to\erre$,}
\end{align}
and such that
\begin{align}
	\label{source:splitting2}
	-m \leq h(x,y)\leq m \quad\text{for every~$(x,y)\in\erre^2$.}
\end{align}
Then, problem \eqref{sys:1:abstract}--\eqref{sys:4:abstract} admits a weak solution $(\ph,\mu)$ such that
\begin{align*}
	& \ph \in \H1 \Vp \cap \L\infty V \cap \L2 W,
	\\
	& \ph \in \Ls4 W \cap \Ls2 {W^{2,r}(\Omega)} \quad \forall \sigma \in (0,T),
	\\
	& \ph \in \L {2p} W \cap \L {p} {W^{2,r}(\Omega)} \quad \forall p\in (1,2),
	\\	
	& \ph \in L^\infty(Q), \quad |\ph(x,t)| <1 \quad\text{ for a.a. } (x,t) \in Q,
	\\
	& \mu \in \L p  V \cap \Ls2 V \quad \forall p\in (1,2) \quad  \forall \sigma \in (0,T),
    \\
    &\nabla \mu \in \L2 H,
\end{align*}
with $r=6$ if $d=3$ and $r$ arbitrary in $(1,+\infty)$ if $d=2$,
\begin{align*}
		& \mu=  - \Delta \ph + F'(\ph) +g \quad \text{$a.e.$ in $Q$},
\end{align*}
and
the variational equality
\begin{align*}
	& \<\dt \ph ,v>_{V}
	+ \iO \nabla \mu \cdot \nabla v = \iO S(\ph) v,
\end{align*}
is satisfied for every test function $ v\in V$, and almost everywhere in $(0,T)$.
Moreover, it holds that
\begin{align*}
	\ph(0)=\ph_0 \quad \text{a.e.~in $\Omega.$}
\end{align*}
\end{theorem}

\begin{remark}
\label{UC3D}
In the three-dimensional case we still cannot prove an equivalent result as Theorem \ref{THM:UNIQ:2d} (see Remark \ref{REM:UQ}).
However, we claim that this can be achieved  for systems like \eqref{sys:1:abstract}--\eqref{sys:4:abstract} if $h(\ph,g)$ in \eqref{source:splitting} is constant.
In fact, this latter assumption entails that, for two given solutions $(\ph_i,\mu_i)_i$ associated with data $g_i$, $i=1,2$, it holds that
\begin{align*}
	S(\ph_1,g_1)-S(\ph_2,g_2) = - m (\ph_1-\ph_2)
\end{align*}
leading to a nice equation for the difference of the means (see \eqref{ODEmean}) allowing us to argue as in \cite{GGM}.
\end{remark}

In the following subsections we briefly discuss two concrete and significant examples strictly related to the above case: the Cahn--Hilliard--Oono equation and tumor growth models based on the phase field approach.

\subsection{The Cahn--Hilliard--Oono equation}
\label{SEC:CHO}
The well-known Cahn--Hilliard--Oono equation is a particular example of equation \eqref{sys:1:abstract} where
\begin{align*}
	g\equiv 0,
	\quad
	h(\ph,g)= h(\ph)= m c,
	\quad
	c\in(-1,1).
\end{align*}
We refer the reader to \cite{GGM} and references therein, where the Flory--Huggins type potential is analyzed (see also \cite{MT}).
Recalling \cite[Eq{.} (10)]{GDM}, the constant $c$ represents the relative backward reaction rate in the phase-separating mixture
at a given temperature below a certain critical temperature. The phase separation process can be induced by the chemical reaction only. Therefore, one
may start with $\ph(0)=-1$, a case which was excluded in the analysis performed in \cite{GGM}. Theorem \ref{THM:ABSTRACT} ensures the existence of a weak solution
also in this case. In addition, uniqueness may hold (see Remark \ref{UC3D}) and we can say more in dimension two (see Theorem~\ref{THM:REGULARITY:2d}).

\subsection{Tumor growth models}
\label{SEC:TUMOR}
The diffuse interface approach is widely used to model tumor growth dynamics and the resulting systems are related to the one analyzed here
(see, e.g., \cite{CL, GLSS}).
In particular, the multiphase approach is used in \cite{KS2, GLNS, FLRS} to account for the complex nature of the tumor which may exhibit a stratified nature, where every layer possesses specific properties.
The mathematical structure of those systems resembles system \eqref{SYS:1}--\eqref{SYS:4}, that is, a multiphase Cahn--Hilliard equation with source term is coupled with reaction-diffusion equations for some chemicals (e.g., nutrients).

For the sake of simplicity, let us focus on the case of scalar order parameters (i.e., two-species tumor) to highlight how our previous analysis may help to study systems like \eqref{sys:1:abstract}--\eqref{sys:4:abstract}.
In that scenario, the order parameter $\ph$ describes the presence/absence of tumor cells with the convention that the level set $\{\ph=-1\}:= \{x \in \Omega: \ph(x)=-1\}$ describes the healthy region, whereas $\{\ph=1\}$ the tumor region.
Then, the Cahn--Hilliard-type equation with source term accounts for cell-to-cell adhesion
effects, whereas the reaction-diffusion equation describe the evolution of a nutrient
species $n$ surrounding the tumor acting as the primary source of nourishment for the tumor mass: we refer to \cite{GLSS} for more details on the modeling.
The latter is usually ruled by a reaction-diffusion equation of the form
\begin{align}	\label{nutrient:abs}
		\dt n - \div \big(\eta(\ph)\nabla (n - \chi \ph )\big) = S_n(\ph,n) \quad \text{in $Q$,}
\end{align}
endowed with no-flux or Robin boundary conditions as well as a suitable initial condition.
There, $S_n(\ph,n)$ is a source term, $\eta(\ph)$ is the nutrient mobility, and $\chi\geq  0$ is the chemotaxis sensitivity.
Concerning regular potentials, the literature is huge and, without the claim to be exhaustive, let us mention \cite{FGR, GL1} and the references therein.
On the other hand, the case of singular potentials has been less investigated.
In this direction, it is worth mentioning that a popular alternative to include more general potentials (and much more general source terms) consists in introducing suitable regularizations (e.g. a viscosity term like in \eqref{SYS:2_app_mod}) in order overcome the complications arising from the singularity of
$F'(\ph)$ in \eqref{sys:2:abstract} (see, e.g., \cite{CGH,SS} and the references therein).

Let us mention that the system analyzed in \cite{FLRS} (see also \cite{GLRS, He})
readily frames into \eqref{sys:1:abstract}--\eqref{sys:4:abstract}
(neglecting the velocity field and using a simplified version of \eqref{nutrient:abs}) where
\begin{align*}
	g = -\chi n, \quad S(\ph)=S(\ph,n)=-m \ph + h(\ph,n).
\end{align*}

Focusing on Flory--Huggins like potentials, the advantage of our approach consists in permitting the initial configuration $\ph_0$
to be chosen as the healthy region $\ph_0=-1,$ which would have been otherwise prevented if one requires the standard $(\ph_0)_\Omega\in (-1,1)$.
As a final remark, let us notice that the mathematical analysis can be challenging if one chooses \eqref{nutrient:abs} as the nutrient equation. Indeed, in this case,
the regularity of the nutrient variable, and thus of the source $g$ in Theorem \ref{THM:ABSTRACT}, cannot be deduced a priori. This is an interesting open issue for future investigations.

\appendix
\section{Proof {of} Proposition~\ref{prop:log}}
\label{app}
Let us show that the {convex part of the} multiphase logarithmic potential $\Psi^{(1)}$ defined in \eqref{Flog}
satisfies Property~\ref{MZ} in the {stronger} sense of Proposition~\ref{prop:log}.

Given an arbitrary $\bphi=(\phi^1,\phi^2)$ as in Property~\ref{MZ}, let us set for brevity of notation
\[
  m:=\min\{\phi^1_\Omega, \phi^2_\Omega\}, \qquad
  M:=\phi^1_\Omega + \phi^2_\Omega.
\]
We introduce now a suitable partition of $\simap$ depending on $\bphi$.
First of all, {see Figure \ref{fig:schema:D}}, we define the sets
\begin{align*}
  D_1&:=\{(x,y)\in\simap:0<x\leq1/4,\quad 1-2x< y< 1-x\},\\
  D_2&:=\{(x,y)\in\simap:0<y\leq1/4,\quad 1-2y< x< 1-y\},
\end{align*}
and
\[
  D:=D_1\cup D_2.
\]

{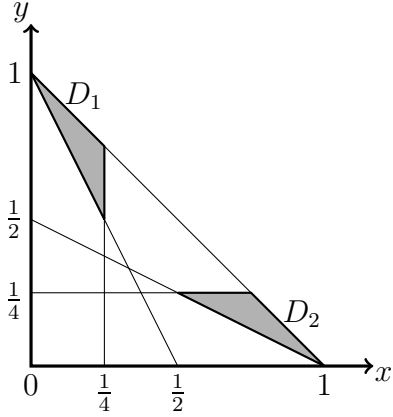
\begin{figure}[h]
\centering
\begin{tikzpicture}[inner sep=0pt,label distance = 1em, scale=1.3]
\draw[very thick,<->] (3.5,0) node[anchor=north west] {$x$} -| (0,3.5) node[anchor=south east] {$y$} ;
\filldraw[color=black!100, fill=black!30, thick, line join=bevel] (0,3) -- (0.75,1.5) -- (0.75, 2.25) -- cycle;
\filldraw[color=black!100, fill=black!30, thick, line join=bevel] (3,0) -- (1.5,0.75) -- (2.25,0.75) -- cycle;
\node (D1) [xshift=.7cm, yshift=3.6cm] {$D_1$};
\node (D2) [xshift=3.6cm, yshift=.7cm] {$D_2$};
%
\draw (0,0) node[below=3pt] {$0$} -- (1.5,0) node[below=3pt] (m) {$\frac12$} --
(3,0) node[below=3pt] {$1$} -- (0,3) node[left=3pt] {$1$} -- cycle;
\draw (1.5,0) -- (0,3);
\draw (0,1.5) node[left=3pt] {$\frac 12$} -- (3,0);
\draw (0.75,0) node[below=3pt] {$\frac 14$} -- (0.75,2.25);
\draw (0,0.75) node[left=3pt, yshift=-3pt] {$\frac 14$} -- (2.25,0.75);
\end{tikzpicture}
\caption{Schematics for the partition of the simplex $\Delta_\circ$: definition of $D_1$ and $D_2$.}
\label{fig:schema:D}
\end{figure}}
Moreover, we partition $\simap\setminus D${, see Figure \ref{fig:schema:ABC},} in the two following ways:
\begin{align*}
  A_1&:=\{(x,y)\in\simap:0<x\leq m/2,\quad 0<y\leq 1-2x\},\\
  B_1&:=\{(x,y)\in\simap:m/2<x\leq 1/2,\quad 0<y\leq 1-2x\},\\
  C_1&:=\Delta_\circ\setminus (A_1\cup B_1\cup D),
\end{align*}
and
\begin{align*}
  A_2&:=\{(x,y)\in\simap:0<y\leq m/2,\quad 0<x\leq 1-2y\},\\
  B_2&:=\{(x,y)\in\simap:m/2<y\leq 1/2,\quad 0<x\leq 1-2y\},\\
  C_2&:=\Delta_\circ\setminus (A_2\cup B_2\cup D).
\end{align*}
{\begin{figure}[h]
\centering
\begin{tabular}{cc}
\begin{tikzpicture}[inner sep=0pt,label distance = 1em, scale=1.3]
\draw[very thick,<->] (3.5,0) node[anchor=north west] {$x$} -| (0,3.5) node[anchor=south east] {$y$} ;
\filldraw[color=black!100, fill=black!30, thick, line join=bevel](0,0) -- (0,3) -- (0.5,2) -- (0.5,0) -- cycle;
\filldraw[color=black!100, fill=black!15, thick](0.5,0) -- (0.5,2) -- (1.5,0) -- cycle;
\filldraw[color=black!100, fill=black!0, thick, line join=bevel] (1.5,0)  -- (0,3) -- (3,0) -- cycle;
\node (A_1) [xshift=.3cm, yshift=.8cm] {$A_1$};
\node (B_1) [xshift=1.1cm, yshift=.8cm] {$B_1$};
\node (C_1) [xshift=2.3cm, yshift=.8cm] {$C_1\cup D$};
\draw (0,0) node[below=3pt] {$0$} -- (1.5,0) node[below=3pt] (m) {$\frac12$} --
(3,0) node[below=3pt] {$1$} -- (0,3) node[left=3pt] {$1$} -- cycle;
\draw (1.5,0) -- (0,3);
\draw (0.5,0) node[below=3pt] {$\frac m2$} -- (0.5,2);
\end{tikzpicture}
    & \hspace{1cm}
\begin{tikzpicture}[inner sep=0pt,label distance = 1em, scale=1.3]
\draw[very thick,<->] (3.5,0) node[anchor=north west] {$x$} -| (0,3.5) node[anchor=south east] {$y$} ;
\filldraw[color=black!100, fill=black!30, thick, line join=bevel](0,0) -- (3,0) -- (2,0.5) -- (0,0.5) -- cycle;
\filldraw[color=black!100, fill=black!15, thick](0,0.5) -- (2,0.5) -- (0,1.5) -- cycle;
\filldraw[color=black!100, fill=black!0, thick, line join=bevel] (0,1.5)  -- (3,0) -- (0,3) -- cycle;
\node (A_2) [xshift=1cm, yshift=.3cm] {$A_2$};
\node (B_2) [xshift=1cm, yshift=1cm] {$B_2$};
\node (C_2) [xshift=1cm, yshift=2cm] {$C_2\cup D$};
\draw (0,0) node[below=3pt] {$0$} -- (0,1.5) node[left=3pt] (m) {$\frac12$} --
(0,3) node[left=3pt] {$1$} -- (3,0) node[below=3pt] {$1$} -- cycle;
\draw (0,1.5) -- (3,0);
\draw (0,0.5) node[left=3pt] {$\frac m2$} -- (2,0.5);
\end{tikzpicture}
\end{tabular}
\caption{Schematics for the partition of the simplex $\Delta_\circ$: definition of $A_i,B_i$ and $C_i\cup D$, $i=1,2$.}
\label{fig:schema:ABC}
\end{figure}
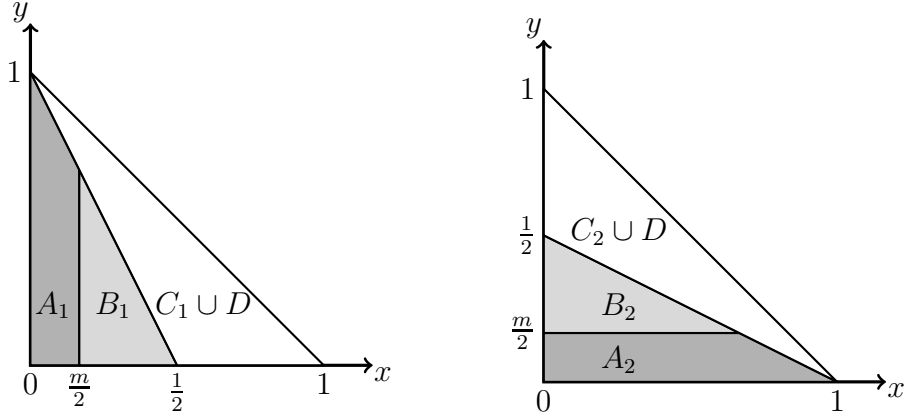}
With this notation, note that we have
\[
  \simap\setminus D=A_1\cup B_1\cup C_1 = A_2\cup B_2\cup C_2.
\]
Now, it is clear that
\begin{equation}\label{MZ_aux1}
  \int_\Omega\Psi^{(1)}_\bph(\bphi)\cdot(\bphi-\bphi_\Omega)=
  \int_{\{\bphi\in D\}}\Psi^{(1)}_\bph(\bphi)\cdot(\bphi-\bphi_\Omega)+
  \int_{\{\bphi\in \simap\setminus D\}}\Psi^{(1)}_\bph(\bphi)\cdot(\bphi-\bphi_\Omega)
\end{equation}
and let us estimate the two contributions on the right-hand side separately.
As far as the first one is concerned,
for every $x,y\in(0,1)$ we introduce the sections
\begin{alignat*}{2}
  &f_{1}^y:[0,1-y]\to[0,+\infty), \qquad
  &&f_{1}^y(r):=\Psi^{(1)}(r,y), \quad r\in\erre,\\
  &f_{2}^x:[0,1-x]\to[0,+\infty), \qquad
  &&f_{2}^x(r):=\Psi^{(1)}(x,r), \quad r\in\erre.
\end{alignat*}
Clearly, for all $x,y\in(0,1)$,
the functions $f_1^y$ and $f_2^x$
are convex, differentiable in $(0,1-y)$ and $(0,1-x)$ respectively,  and satisfy
\begin{align*}
  &(f_1^y)'(r)=(D_1\Psi^{(1)})(r,y) \qquad\forall\,r\in(0,1-y),\\
  &(f_2^x)'(r)=(D_2\Psi^{(1)})(x,r) \qquad\forall\,r\in(0,1-x).
\end{align*}
Moreover, their convex conjugates are defined as
\begin{alignat*}{2}
  (f_1^y)^*(z) &= \sup_{r\in[0,1-y]}\left\{rz-\Psi^{(1)}(r,y)\right\}, \quad z\in\erre,\\
  (f_2^x)^*(z) &= \sup_{r\in[0,1-x]}\left\{rz-\Psi^{(1)}(x,r)\right\}, \quad z\in\erre,
\end{alignat*}
and
thanks to the Young (in)equality we infer that,
for every $(x,y), (z,w)\in\simap$,
\begin{align*}
  (D_1\Psi^{(1)})(x,y)(x-z) &=
  (f_1^y)'(x)(x - z) = f_1^y(x) + (f_1^y)^*((f_1^y)'(x))
  -(f_1^y)'(x)z\\
  &\geq
  f_1^y(x) + (f_1^y)^*((f_1^y)'(x)) - f_1^y(z) - (f_1^y)^*((f_1^y)'(x))\\
  &\geq
  - \max_{\rr\in\simcl}\Psi^{(1)}(\rr),
\end{align*}
and
\begin{align*}
  (D_2\Psi^{(1)})(x,y)(y-w)&=
  (f_2^x)'(y)(y-w) = f_2^x(y) + (f_2^x)^*((f_2^x)'(y)) - (f_2^x)'(y)w\\
  &\geq f_2^x(y) + (f_2^x)^*((f_2^x)'(y)) -
  f_2^x(w) - (f_2^x)^*((f_2^x)'(y))\\
  &\geq - \max_{\rr\in\simcl}\Psi^{(1)}(\rr).
\end{align*}
It follows in particular that
\begin{align*}
  &\int_{\{\bphi\in D_1\}}(D_1\Psi^{(1)})(\bphi)(\phi^1-\phi^1_\Omega)
  +\int_{\{\bphi\in D_2\}}(D_2\Psi^{(1)})(\bphi)(\phi^2-\phi^2_\Omega)\\
  &\quad \geq - \max_{\rr\in\simcl}|\Psi^{(1)}(\rr){|}\left(|\{\bphi\in D_1\}| + |\{\bphi\in D_2\}|\right){.}
\end{align*}
Now, noting that $\{\bphi\in D_1\}\subset\{\phi^2\geq1/2\}$ and
$\{\bphi\in D_2\}\subset\{\phi^1\geq1/2\}$, we have
by the Chebyshev inequality and the positivity of $\phi^1,\phi^2$ that
\begin{align*}
  |\{\bphi\in D_1\}| + |\{\bphi\in D_2\}|
  &\leq |\{\phi^2\geq 1/2\}| + |\{\phi^1\geq 1/2\}|\\
  &\leq 2 \int_\Omega(\phi^1 + \phi^2) = 2|\Omega|(\phi^1_\Omega + \phi^2_\Omega).
\end{align*}
Putting this information together, we infer that
\begin{equation}\label{MZ_aux2}
\int_{\{\bphi\in D_1\}}(D_1\Psi^{(1)})(\bphi)(\phi^1-\phi^1_\Omega)
  +\int_{\{\bphi\in D_2\}}(D_2\Psi^{(1)})(\bphi)(\phi^2-\phi^2_\Omega)
\geq  - C(\phi^1_\Omega + \phi^2_\Omega),
\end{equation}
where the constant $C$ only depends on $\Psi^{(1)}$ and $\Omega$.
Furthermore, noting that the logarithmic potential \eqref{Flog}
satisfies $D_1\Psi^{(1)}(\bphi)\geq 0$ and $D_2\Psi^{(1)}(\bphi)\geq 0$ on $D$, we have
again by the Young inequality and the boundedness and positivity of $\Psi^{(1)}$ that
\begin{align*}
  &\phi^1_\Omega\int_{\{\bphi\in D_1\}}|(D_1\Psi^{(1)})(\bphi)|=
  \phi^1_\Omega\int_{\{\bphi\in D_1\}}(D_1\Psi^{(1)})(\bphi)=
  \phi^1_\Omega\int_{\{\bphi\in D_1\}}\Psi_\bph^{(1)}(\bphi)\cdot(1,0)\\
  &\quad \leq\phi_\Omega^1\left[\int_{\{\bphi\in D_1\}}(\Psi^{(1)})^*(\Psi_\bph^{(1)}(\bphi))
  +\int_{\{\bphi\in D_1\}}\Psi^{(1)}((1,0))
  \right]\\
  &\quad \leq \phi_\Omega^1\left[\int_{\{\bphi\in D_1\}}(\Psi^{(1)})^*(\Psi_\bph^{(1)}(\bphi))
  +\int_{\{\bphi\in D_1\}}\Psi^{(1)}(\bphi)\right]
  +\phi^1_\Omega|\Omega|\max_{\rr\in\simcl}|\Psi^{(1)}(\rr)|\\
  &\quad =\phi_\Omega^1\int_{\{\bphi\in D_1\}}\Psi_\bph^{(1)}(\bphi)\cdot\bphi
  +\phi^1_\Omega|\Omega|\max_{\rr\in\simcl}|\Psi^{(1)}(\rr)|\\
  &\quad =\phi_\Omega^1\int_{\{\bphi\in D_1\}}(D_1\Psi^{(1)})(\bphi)\phi^1
  +\phi_\Omega^1\int_{\{\bphi\in D_1\}}(D_2\Psi^{(1)})(\bphi)\phi^2
  +\phi^1_\Omega|\Omega|\max_{\rr\in\simcl}|\Psi^{(1)}(\rr)|.
\end{align*}
Now, since $\phi^1\leq\frac14$ on $D_1$, one readily has that
\[
  \phi_\Omega^1\int_{\{\bphi\in D_1\}}(D_1\Psi^{(1)})(\bphi)\phi^1
  \leq \frac14\phi_\Omega^1\int_{\{\bphi\in D_1\}}(D_1\Psi^{(1)})(\bphi)
  =\frac14\phi_\Omega^1\int_{\{\bphi\in D_1\}}|(D_1\Psi^{(1)})(\bphi)|.
\]
Also, since $\phi^1_\Omega+\phi^2_\Omega\leq\frac18$, it holds in particular that
$\phi^1_\Omega\leq\frac12 - \phi^2_\Omega$. Taking this into account,
together with the fact that $\phi^2\geq\frac12$ and $(D_2\Psi^{(1)})(\bphi)\geq0$
on $D_1$, we also deduce that
\begin{align*}
  \frac12\phi^1_\Omega\int_{\{\bphi\in D_1\}}|(D_2\Psi^{(1)})(\bphi)|&\leq
  \phi_\Omega^1\int_{\{\bphi\in D_1\}}(D_2\Psi^{(1)})(\bphi)\phi^2\\
  &\leq\left(\frac12 - \phi^2_\Omega\right)\int_{\{\bphi\in D_1\}}(D_2\Psi^{(1)})(\bphi)\\
  &\leq \int_{\{\bphi\in D_1\}}(D_2\Psi^{(1)})(\bphi)(\phi^2-\phi^2_\Omega).
\end{align*}
Putting this information together and rearranging the terms yield
\begin{align*}
  &\frac34\phi_\Omega^1\int_{\{\bphi\in D_1\}}|(D_1\Psi^{(1)})(\bphi)|
  +\frac12\phi^1_\Omega\int_{\{\bphi\in D_1\}}|(D_2\Psi^{(1)})(\bphi)|\\
  &\quad \leq 2\int_{\{\bphi\in D_1\}}(D_2\Psi^{(1)})(\bphi)(\phi^2-\phi^2_\Omega)
  +C\phi^1_\Omega,
\end{align*}
where the constant $C$ only depends on $\Psi^{(1)}$ and $\Omega$.
{Finally, we obtain} that
\begin{equation}\label{MZ_aux3}
  \frac14\phi^1_\Omega\int_{\{\bphi\in D_1\}}|\Psi_\bph^{(1)}(\bphi)|
  \leq \int_{\{\bphi\in D_1\}}(D_2\Psi^{(1)})(\bphi)(\phi^2-\phi^2_\Omega)
  +C\phi^1_\Omega.
\end{equation}
Clearly, it is not difficult to check that the same arguments on $D_2$ ensure that
\begin{equation}\label{MZ_aux4}
  \frac14\phi^2_\Omega\int_{\{\bphi\in D_2\}}|\Psi_\bph^{(1)}(\bphi)|
  \leq \int_{\{\bphi\in D_2\}}(D_1\Psi^{(1)})(\bphi)(\phi^1-\phi^1_\Omega)
  +C\phi^2_\Omega.
\end{equation}
At this point, summing up the estimates \eqref{MZ_aux2}, \eqref{MZ_aux3}, and \eqref{MZ_aux4},
we obtain after elementary rearrangements that
\begin{equation}\label{MZ_aux5}
  \frac14\min\{\phi^1_\Omega,\phi^2_\Omega\}\int_{\{\bphi\in D\}}|\Psi_\bph^{(1)}(\bphi)|
  \leq \int_{\{\bphi\in D\}}\Psi^{(1)}_\bph(\bphi)\cdot(\bphi-\bphi_\Omega)
  +C(\phi^1_\Omega+\phi^2_\Omega),
\end{equation}
where $C>0$ only depends of $\Psi^{(1)}$ and $\Omega$ (and not of $\bphi$).
This allows to estimate the first term on the right-hand side of \eqref{MZ_aux1}.
Let us focus now on the second term on the right-hand side of \eqref{MZ_aux1}:
to this end, we have
\begin{align*}
  \int_{\{\bphi\in \simap\setminus D\}}\Psi^{(1)}_\bph(\bphi)\cdot(\bphi-\bphi_\Omega)
  &=\sum_{i=1,2}\int_{\{\bphi\in \simap\setminus D\}}(D_i\Psi^{(1)})(\bphi)(\phi^i-\phi^i_\Omega)\\
  &=\sum_{i=1,2}\int_{\{\bphi\in A_i\cup B_i\cup C_i\}}(D_i\Psi^{(1)})(\bphi)(\phi^i-\phi^i_\Omega).
\end{align*}
Now, let $i\in\{1,2\}$ be arbitrary. By definition of $A_i$
we have $\phi^i\leq m/2$ on $\{\bphi\in A_i\}$, and
the logarithmic potential \eqref{Flog} satisfies
$(D_i\Psi^{(1)})(\bphi)\leq0$ on $\{\bphi\in A_i\}$. Consequently, we get
\begin{align*}
  \int_{\{\bphi\in A_i\}}(D_i\Psi^{(1)})(\bphi)(\phi^i-\phi^i_\Omega)
  &=\int_{\{\bphi\in A_i\}}|(D_i\Psi^{(1)})(\bphi)|(\phi^i_\Omega-\phi^i)\\
  &\geq \left(\phi^i_\Omega - \frac m2\right)\int_{\{\bphi\in A_i\}}|(D_i\Psi^{(1)})(\bphi)|\\
  &\geq \frac m2\int_{\{\bphi\in A_i\}}|(D_i\Psi^{(1)})(\bphi)|.
\end{align*}
Secondly, we note that on the set $\{\bphi\in B_i\}$ the function
$(D_i\Psi^{(1)})(\bphi)$ is negative and bounded. More precisely,
with the choice of the logarithmic potential \eqref{Flog}
one can easily check that it holds that
\begin{align*}
 \max_{\{\bphi\in B_1\}}|(D_1\Psi^{(1)})(\bphi)|
 &=\max\{(D_1\Psi^{(1)})(\rr): m/2<r_1\leq1/2,\quad 0 < r_2 \leq 1-2r_1\}\\
 &=|(D_1\Psi^{(1)})(m/2, 0)|=|\ln(m/2) - \ln (1-m/2)\,|\\
 &\leq |\ln(m/2)\,| + |\ln(1-m/2)\,| \leq 2|\ln(m/2)\,|,
\end{align*}
where the last inequality follows since $\frac m2 \leq 1 - \frac m2$,
and similarly
\begin{align*}
 \max_{\{\bphi\in B_2\}}|(D_2\Psi^{(1)})(\bphi)| \leq 2|\ln(m/2)\,|.
\end{align*}
We infer then that
\begin{align*}
  \int_{\{\bphi\in B_i\}}(D_i\Psi^{(1)})(\bphi)(\phi^i-\phi^i_\Omega) &=
  \int_{\{\bphi\in B_i\}}|(D_i\Psi^{(1)})(\bphi)|(\phi^i_\Omega-\phi^i) \\
  &\geq - \int_{\{\bphi\in B_i\}}|(D_i\Psi^{(1)})(\bphi)|\phi^i \\
  &\geq - 2|\Omega||\ln(m/2)\,|\phi^i_\Omega,
\end{align*}
as well as
\[
  \phi^i_\Omega \int_{\{\bphi\in B_i\}}|(D_i\Psi^{(1)})(\bphi)|
  \leq 2|\Omega||\ln(m/2)\,|\phi^i_\Omega.
\]
Eventually, noting that $(D_i\Psi^{(1)})(\bphi)\geq0$
and $\phi^i\geq \frac14$ on $\{\bphi\in C_i\}$,
and recalling that $\phi_\Omega^i\leq \frac14 - \phi_\Omega^i$,  we have
\begin{align*}
  \int_{\{\bphi\in C_i\}}(D_i\Psi^{(1)})(\bphi)(\phi^i-\phi^i_\Omega) &=
  \int_{\{\bphi\in C_i\}}|(D_i\Psi^{(1)})(\bphi)|(\phi^i-\phi^i_\Omega) \\
  &\geq \left(\frac14 - \phi^i_\Omega\right) \int_{\{\bphi\in C_i\}}|(D_i\Psi^{(1)})(\bphi)| \\
  &\geq \phi^i_\Omega \int_{\{\bphi\in C_i\}}|(D_i\Psi^{(1)})(\bphi)|.
\end{align*}
Taking everything into account, this yields for $i=1,2$ that
\[
  \int_{\{\bphi\in A_i\cup B_i\cup C_i\}}(D_i\Psi^{(1)})(\bphi)(\phi^i-\phi^i_\Omega)
  \geq \frac m2\int_{\{\bphi\in A_i\cup B_i \cup C_i\}}|(D_i\Psi^{(1)})(\bphi)|
  - 4|\Omega| |\ln(m/2)\,| \phi^i_\Omega.
\]
Summing up on $i$ these estimates we infer that
\begin{equation}\label{MZ_aux6}
\frac m2 \int_{\{\bphi\in \simap\setminus D\}}|\Psi^{(1)}_\bph(\bphi)| \leq
  \int_{\{\bphi\in \simap\setminus D\}}\Psi^{(1)}_\bph(\bphi)\cdot(\bphi-\bphi_\Omega)
  +C(\phi^1_\Omega + \phi^2_\Omega)|\ln(m/2)\,|
\end{equation}
where the constant $C$ only depends on $\Omega$ and $\Psi^{(1)}$. Combining
\eqref{MZ_aux5} with \eqref{MZ_aux6} the thesis follows.

\medskip
\noindent
{\bf Acknowledgements.} We thank the reviewers for their careful reading of the manuscript as well as for their useful comments. We also thank Andrea Poiatti who pointed out
the difficulties hidden in the proof of the separation property with the solvent $S$. The authors are members of Gruppo Nazionale per l'Analisi Ma\-te\-ma\-ti\-ca, la Probabilit\`{a} e le loro Applicazioni (GNAMPA), Istituto Nazionale di Alta Matematica (INdAM). The first and the third author have been partially funded by MIUR-PRIN research grant n. 2020F3NCPX.

\end{document}